\documentclass[a4paper]{article}
\usepackage{enumitem}
\usepackage{calc}

\usepackage{lmodern} 

\usepackage{graphicx} 

\usepackage{amsmath, accents}
\usepackage{amssymb,tikz}
\usepackage{pict2e}
\usepackage{amsthm}
\usepackage{amscd}
\usepackage{mathrsfs}
\usepackage{todonotes}
\usetikzlibrary{calc} 

\usepackage{yfonts}

\usepackage{marvosym}
\usepackage[percent]{overpic}

\newtheorem{theorem}{Theorem} [section]
\newtheorem{proposition}[theorem]{Proposition}

\newtheorem{remark}[theorem]{Remark}

\theoremstyle{definition}

\newcommand{\C}{\mathbb{C}}

\newcommand{\Z}{\mathbb{Z}}
\newcommand{\N}{\mathbb{N}}

\tikzset{
	master/.style={
		execute at end picture={
			\coordinate (lower right) at (current bounding box.south east);
			\coordinate (upper left) at (current bounding box.north west);
		}
	},
	slave/.style={
		execute at end picture={
			\pgfresetboundingbox
			\path (upper left) rectangle (lower right);
		}
	}
}

\let\oldbibliography\thebibliography
\renewcommand{\thebibliography}[1]{\oldbibliography{#1}
\setlength{\itemsep}{-0.5pt}}

\def\XXint#1#2#3{{\setbox0=\hbox{$#1{#2#3}{\int}$}
\vcenter{\hbox{$#2#3$}}\kern-.5\wd0}}

\usepackage{tikz}
\usetikzlibrary{arrows}
\usetikzlibrary{decorations.pathmorphing}
\usetikzlibrary{decorations.markings}
\usetikzlibrary{patterns}
\usetikzlibrary{automata}
\usetikzlibrary{positioning}
\usepackage{tikz-cd}
\tikzset{->-/.style={decoration={
				markings,
				mark=at position #1 with {\arrow{latex}}},postaction={decorate}}}
	
	\tikzset{-<-/.style={decoration={
				markings,
				mark=at position #1 with {\arrowreversed{latex}}},postaction={decorate}}}

\usetikzlibrary{shapes.misc}\tikzset{cross/.style={cross out, draw, 
         minimum size=2*(#1-\pgflinewidth), 
         inner sep=0pt, outer sep=0pt}}

\usepackage{pgfplots}
	
\allowdisplaybreaks	

\numberwithin{equation}{section}

\def\ds{\displaystyle}
\def\bigO{{\cal O}}

\makeatletter
\newcommand{\oset}[3][0ex]{%
  \mathrel{\mathop{#3}\limits^{
    \vbox to#1{\kern-2\ex@
    \hbox{$\scriptstyle#2$}\vss}}}}
\makeatother
\def\be{\begin{equation}}
\def\ee{\end{equation}}

\usepackage[colorlinks=true]{hyperref}
\hypersetup{urlcolor=blue, citecolor=red, linkcolor=blue}

\begin{document}
\title{\vspace{-1cm} Counting domino and lozenge tilings of reduced domains \\ with Pad\'{e}-type approximants}
\author{Christophe Charlier and Tom Claeys}

\maketitle

\begin{abstract}
We introduce a new method for studying gap probabilities in a class of discrete determinantal point processes with double contour integral kernels. This class of point processes includes uniform measures of domino and lozenge tilings as well as their doubly periodic generalizations. We use a Fourier series approach to simplify the form of the kernels and to characterize gap probabilities in terms of Riemann-Hilbert problems. 

As a first illustration of our approach, we obtain an explicit expression for the number of domino tilings of reduced Aztec diamonds in terms of Padé approximants, by solving the associated Riemann-Hilbert problem explicitly. As a second application, we obtain an explicit expression for the number of lozenge tilings of (simply connected) reduced hexagons in terms of Hermite-Pad\'e approximants. For more complicated domains, such as hexagons with holes, the number of tilings involves a generalization of Hermite-Padé approximants.
\end{abstract}



\section{Introduction}

We will study gap probabilities in a class of discrete determinantal point processes and characterize them by Riemann-Hilbert problems. The simplest examples are  point processes associated to domino tilings of Aztec diamonds and to lozenge tilings of hexagons. Here, the gap probabilities encode the number of (domino or lozenge) tilings of certain reductions of the initial domains (Aztec diamonds or hexagons), in which regions are removed.
We will start by introducing the tiling models in detail and by stating our results specialized to these models. Afterwards, we will explain our general methodology which leads to these results. In view of future research, we strongly believe that our Riemann-Hilbert characterizations will open the door towards asymptotic analysis in various tiling models.

\subsection{Domino tilings of reduced Aztec diamonds}
\paragraph{Domino tilings of the Aztec diamond.}
The Aztec diamond $A_N$ of size $N$ is a two-dimensional region consisting of the union of all $2N(N+1)$ squares $S_{j,k}:=[j,j+1]\times [k,k+1]\subset \mathbb R^2$ for which $j,k\in\mathbb Z$ and $S_{j,k}\subset\{(x,y)\in\mathbb R^2: |x|+|y|\leq N+1\}$. The Aztec diamond was first introduced in 1992 by Elkies, Kuperberg, Larsen and Propp \cite{EKLP}. A domino is a $1\times 2$ or $2\times 1$ rectangle of the form $[j,j+1]\times [k,k+2]$ (called  vertical domino) or $[j,j+2]\times [k,k+1]$ (called  horizontal domino). A horizontal domino is called a north domino if $j+k+N$ is even, and a south domino if $j+k+N$ is odd. Similarly, a vertical domino is called an east domino if $j+k+N$ is even, and a west domino if $j+k+N$ is odd. A domino tiling $T$ of $A_N$ consists of a collection of $N(N+1)$ dominoes which cover the whole region $A_N$ and whose interiors are disjoint, see Figure \ref{fig:Aztec} (left). 

\begin{figure}
\begin{center}
\begin{tikzpicture}[master]
\node at (0,0) {\includegraphics[width=6cm]{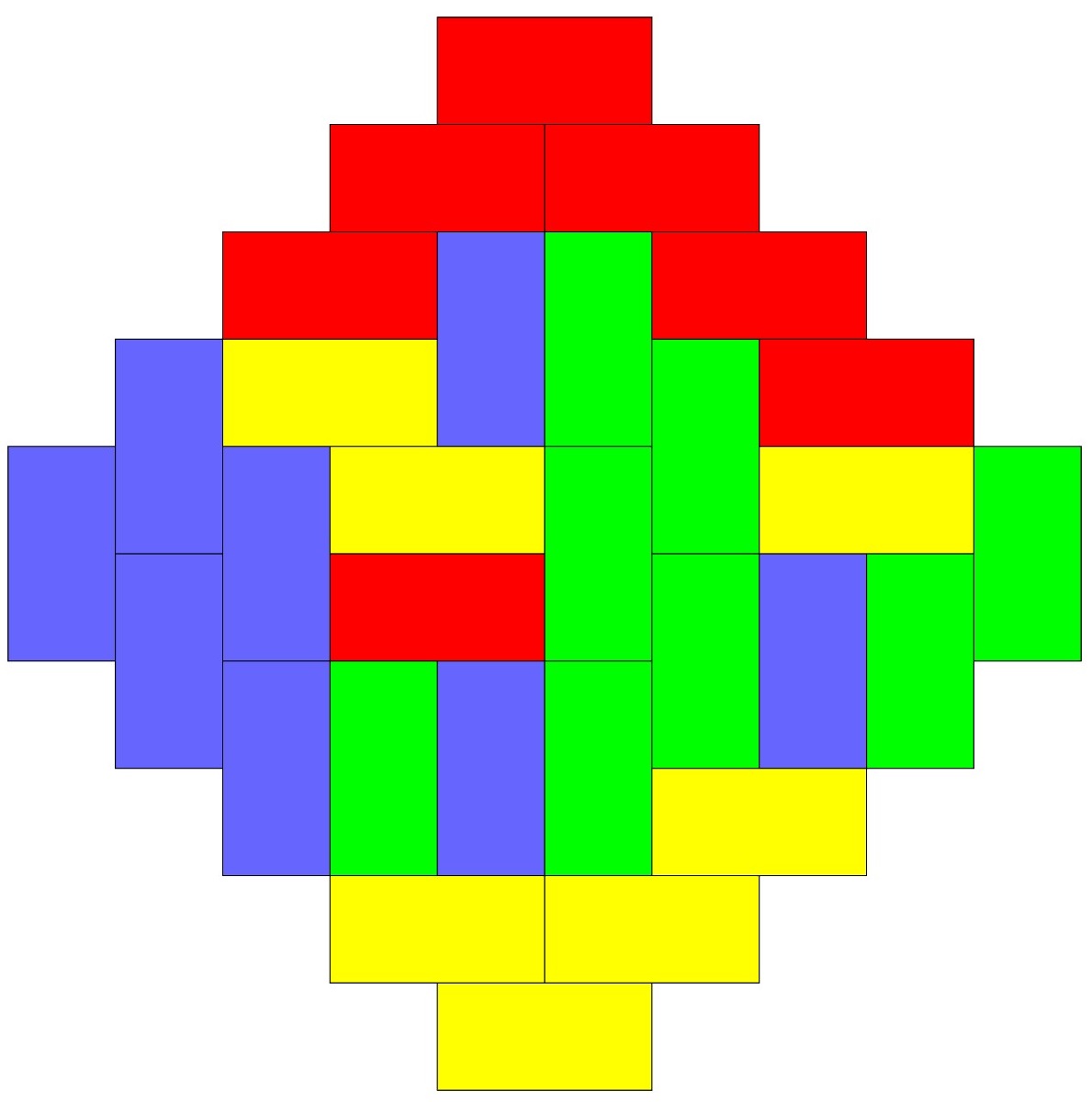}};
\end{tikzpicture}
\begin{tikzpicture}[slave]
\node at (0,0) {\includegraphics[width=6cm]{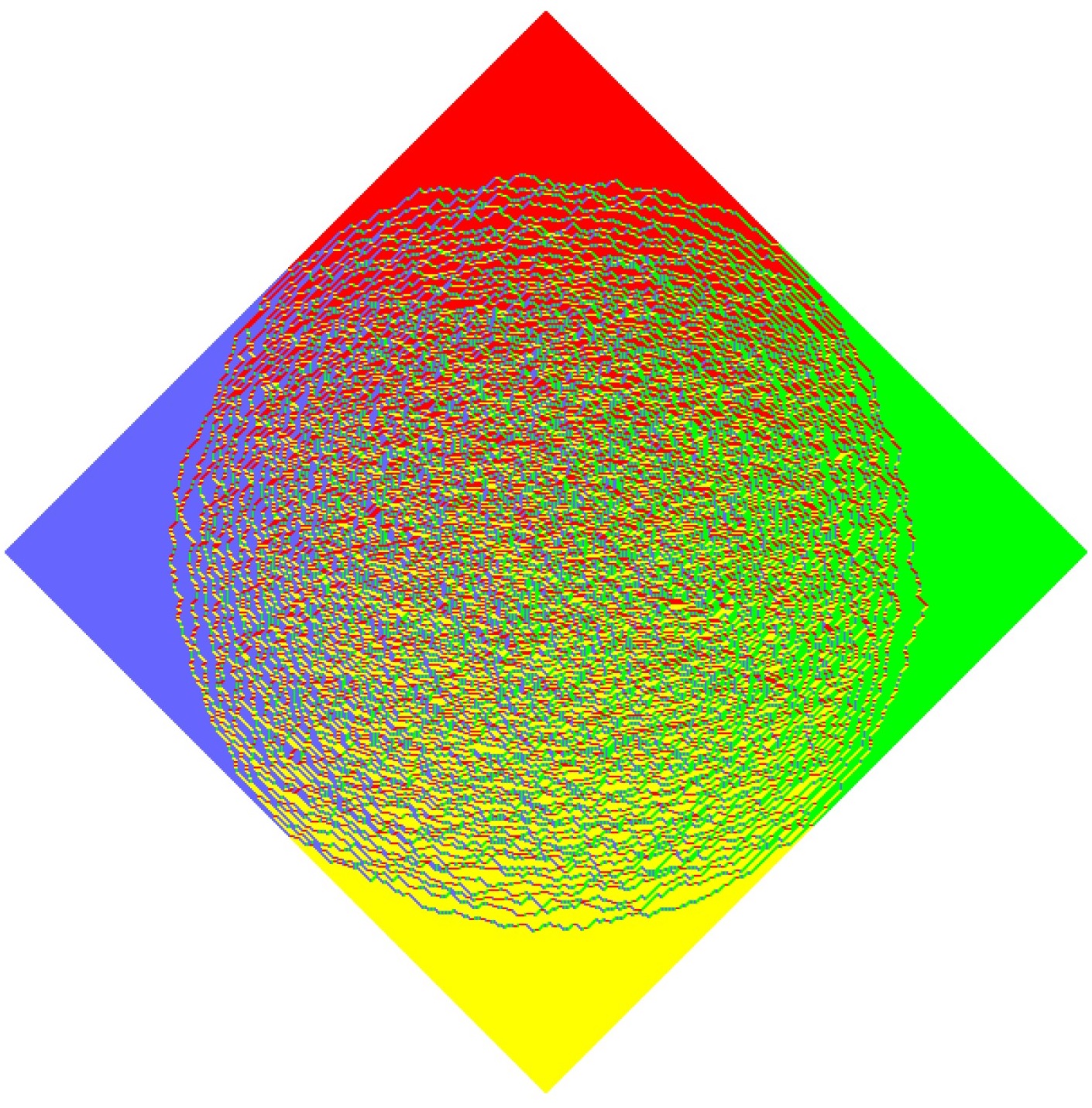}};
\end{tikzpicture}
\end{center}
\caption{\label{fig:Aztec}Left: a tiling of $A_{5}$. The north, south, east and west dominoes are shown in red, yellow, green and blue, respectively. Right: a tiling of $A_{300}$ chosen uniformly at random.}
\end{figure}

A celebrated result from \cite{EKLP} states that the number of domino tilings of $A_N$ is equal to $2^{\frac{N(N+1)}{2}}$. 
More generally, the generating function for the number of vertical dominoes $v(T)$ in a tiling $T$ has an explicit expression: we have the identity
\begin{equation}\label{eq:ADT}F_N(a):=\sum_{T\in\mathcal T(A_N)}a^{v(T)}=(1+a^2)^{\frac{N(N+1)}{2}},\qquad 0<a\leq 1,\end{equation}
where the sum runs over the set $\mathcal T(A_N)$ of all domino tilings of $A_N$ \cite{EKLP}.
The Aztec diamond is one of the rare non-trivial examples of sequences of domains for which the number of domino tilings becomes large but an explicit expression is still available. Other examples of such domains are rectangles \cite{K, TF} and certain triangular deformations of the Aztec diamond \cite{CHK, dF}. Domino tilings of the Aztec diamond have been studied intensively in the past two decades because of their remarkable asymptotic behavior in the limit where the order $N$ gets large. One can then distinguish a frozen (or solid) region, outside the circle inscribed in the square $|x|+|y|=N+1$, and a rough (or liquid) region inside the circle, see Figure \ref{fig:Aztec} (right). This phenomenon is known as the arctic circle theorem, and the transition between the rough and the frozen phase can be described in terms of the Airy process, a universal point process arising in a large number of random matrix and random interface growth models \cite{Joh2}.
In recent years, more complicated doubly-periodic probability distributions on domino tilings of Aztec diamonds have been studied. In such models, one can identify a third, smooth (or gas), phase in the large $N$ asymptotics \cite{CY2014, CJ, BCJ, DK2021, BeD, Berggren, Ruelle2022, BD2023, BB2023, BT2024, BobBob24, BB2024, Pior2024, KP2024, BN2025, Shea2025}.

\medskip

\paragraph{Domino tilings of reduced Aztec diamonds and Pad\'e approximants.}
In this paper, we consider reduced Aztec diamonds, i.e.\ Aztec diamonds with certain unit squares removed. The simplest situation of such a domain is an Aztec diamond with an approximate rectangle removed, see Figure \ref{fig:ADreduced}. Such regions were first introduced by Colomo and Pronko in \cite{CP2013}. We restrict ourselves to this situation to present our result here, although we will consider more general domains with holes later in Section \ref{section:tiling}.
More precisely, for $m=1,\ldots, {N}$ and $k=1,\ldots, m+1$, $A_N^{m,k}$ is the union of all unit squares $S_{i,j}$ for which $i,j\in\mathbb Z$ and 
\[S_{i,j}\subset\{(x,y)\in\mathbb R^2: |x|+|y|\leq N+1\ \mbox{and}\ y\leq \max\{2m-1-N-x,x-2m-1+N+2k\}\}.\] This implies that $N-m+1$ diagonal rows of $m-k+1$ horizontal dominoes are cut out of the Aztec diamond. Naturally, one expects that the removed region, if sufficiently large, will affect the rough and smooth regions of a random tiling, as illustrated in Figure \ref{fig:ADreduced} (right), and will dramatically decrease the number of possible tilings. While such asymptotic questions are of great interest, we will focus here on exact formulas for finite $N,k,m$. We plan to address such  asymptotic problems in future research\footnote{In our recent preprint \cite{CC2025}, which we completed after the present paper, we established detailed large $N$ asymptotics for the (weighted) number of domino tilings of reduced Aztec diamonds with a removed region of macroscopic size}.

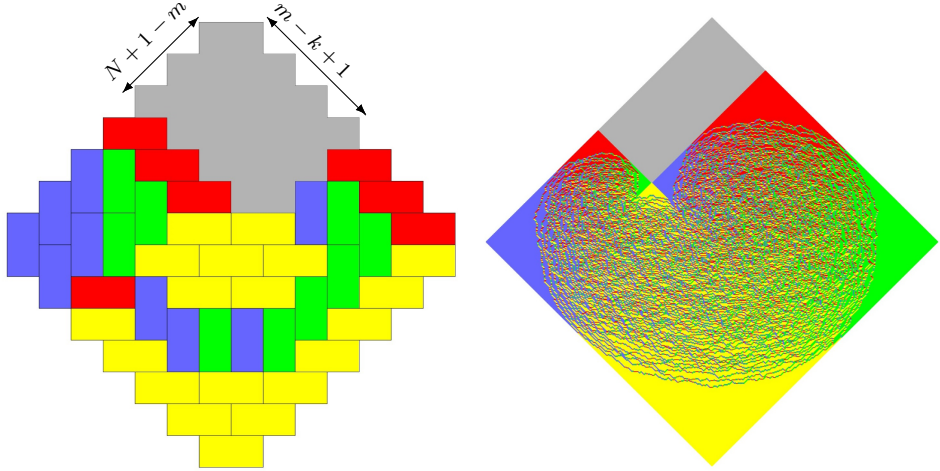
\begin{figure}
\begin{center}
\begin{tikzpicture}[master]
\node at (0,0) {\includegraphics[width=6cm]{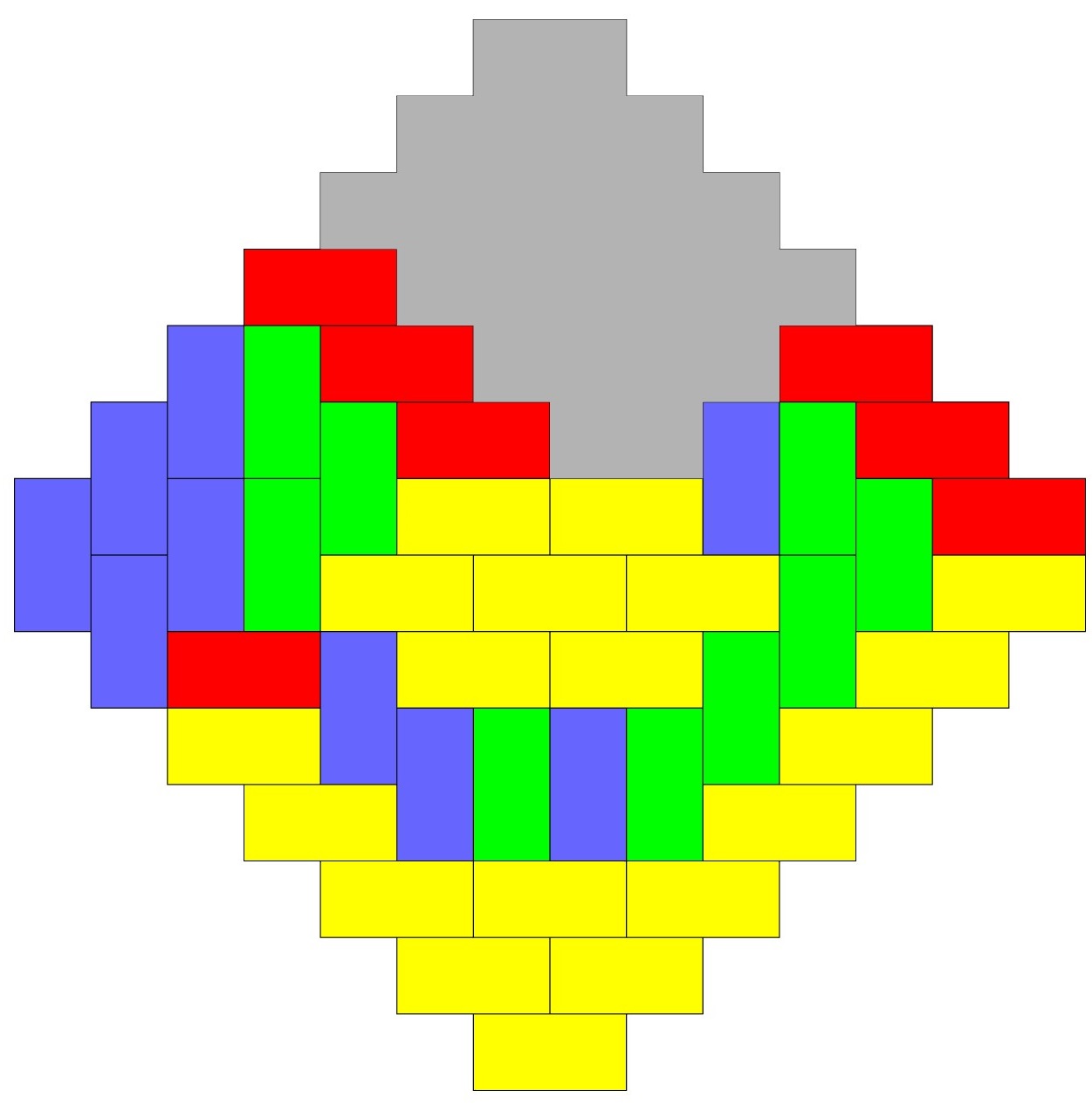}};
\draw[->-=0.9999, -<-=0] (-1.4,2)--(-0.4,3);
\node[rotate=45] at (-1.1,2.7) {\footnotesize $N+1-m$};
\draw[->-=0.9999, -<-=0] (0.5,3)--(1.8,1.7);
\node[rotate=-45] at (1.1,2.7) {\footnotesize $m-k+1$};
\end{tikzpicture}
\begin{tikzpicture}[slave]
\node at (0,0) {\includegraphics[width=6cm]{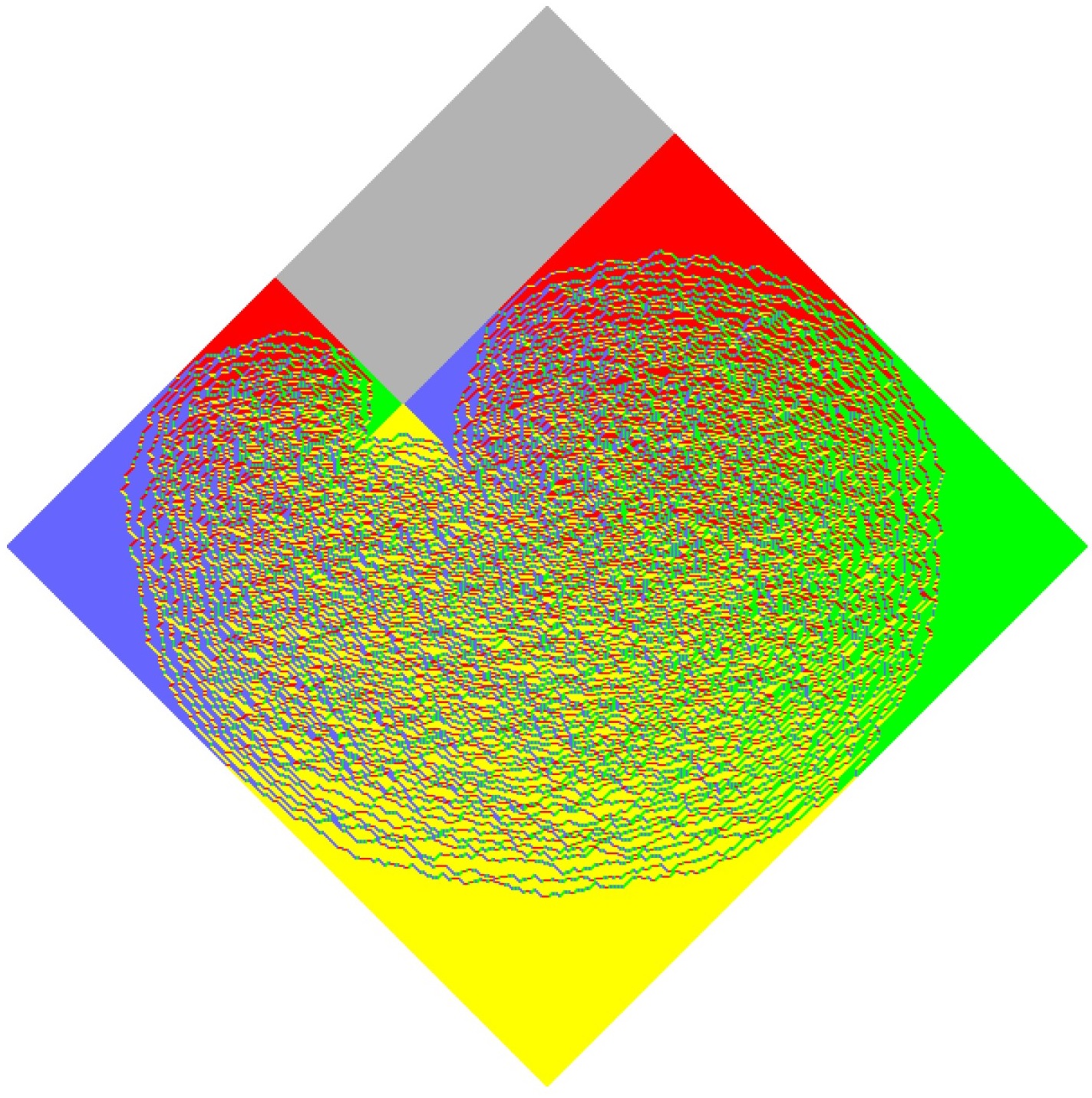}};
\end{tikzpicture}
\end{center}
\caption{\label{fig:ADreduced}Left: a tiling of $A_{7}^{5,2}$. The north, south, east and west dominoes are shown in red, yellow, green and blue, respectively. The set $A_{N}\setminus A_{N}^{m,k}$ is shown in grey; it contains $N+1-m$ corners on the north-west side and $m-k+1$ corners on the north-east side. Right: a tiling of $A_{300}^{150,80}$ chosen uniformly at random.}
\end{figure}

\begin{figure}
\begin{center}
\begin{tikzpicture}[master]
\node at (0,0) {\includegraphics[width=6cm]{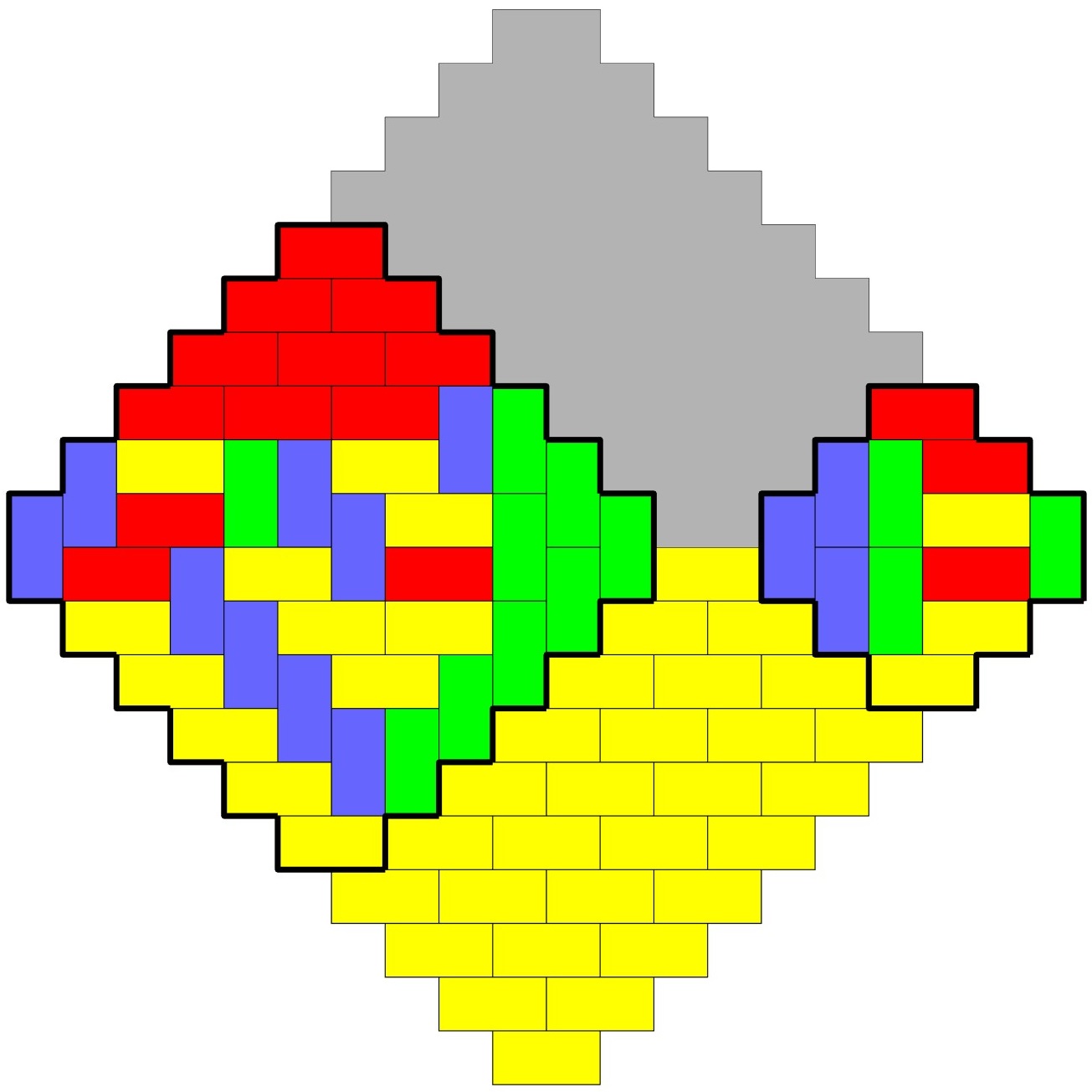}};
\end{tikzpicture}
\begin{tikzpicture}[slave]
\node at (0,0) {\includegraphics[width=6cm]{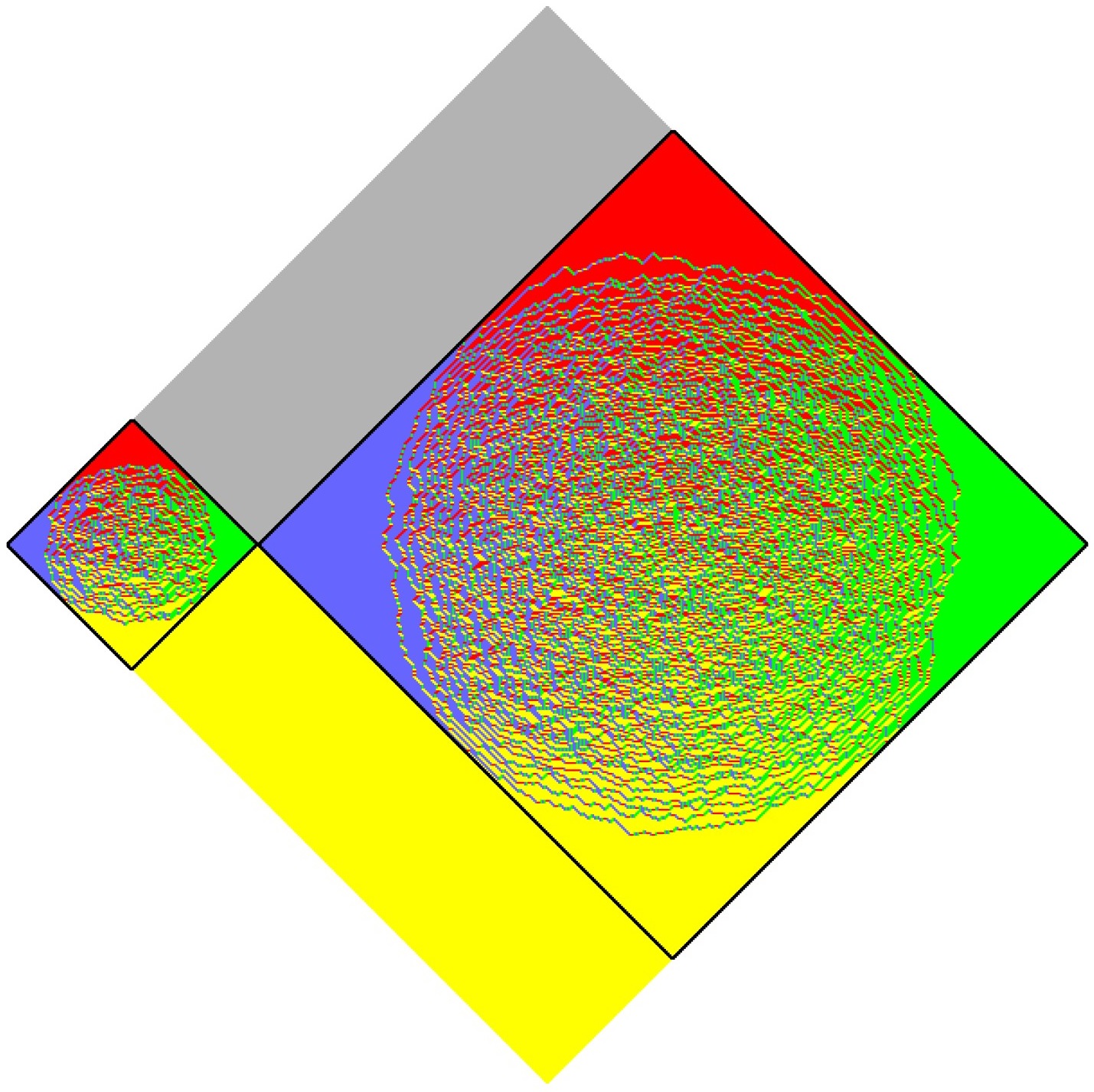}};
\end{tikzpicture}
\end{center}
\caption{\label{fig:mirror frozen}Any tiling of $A_{N}^{m,1}$ decouples into a tiling of an Aztec diamond of order $m-1$ and an Aztec diamond of order $N-m$, whose boundaries are shown with thick black lines. The set $A_{N}\setminus A_{N}^{m,k}$ is shown in grey, and its mirror image with respect to $y=0$ is tiled only with south dominoes.  Left: a tiling of $A_{10}^{7,1}$. Right: a tiling of $A_{300}^{70,1}$ chosen uniformly at random.}
\end{figure}

\begin{figure}
\begin{center}
\begin{tikzpicture}[master]
\node at (0,0) {\includegraphics[width=6cm]{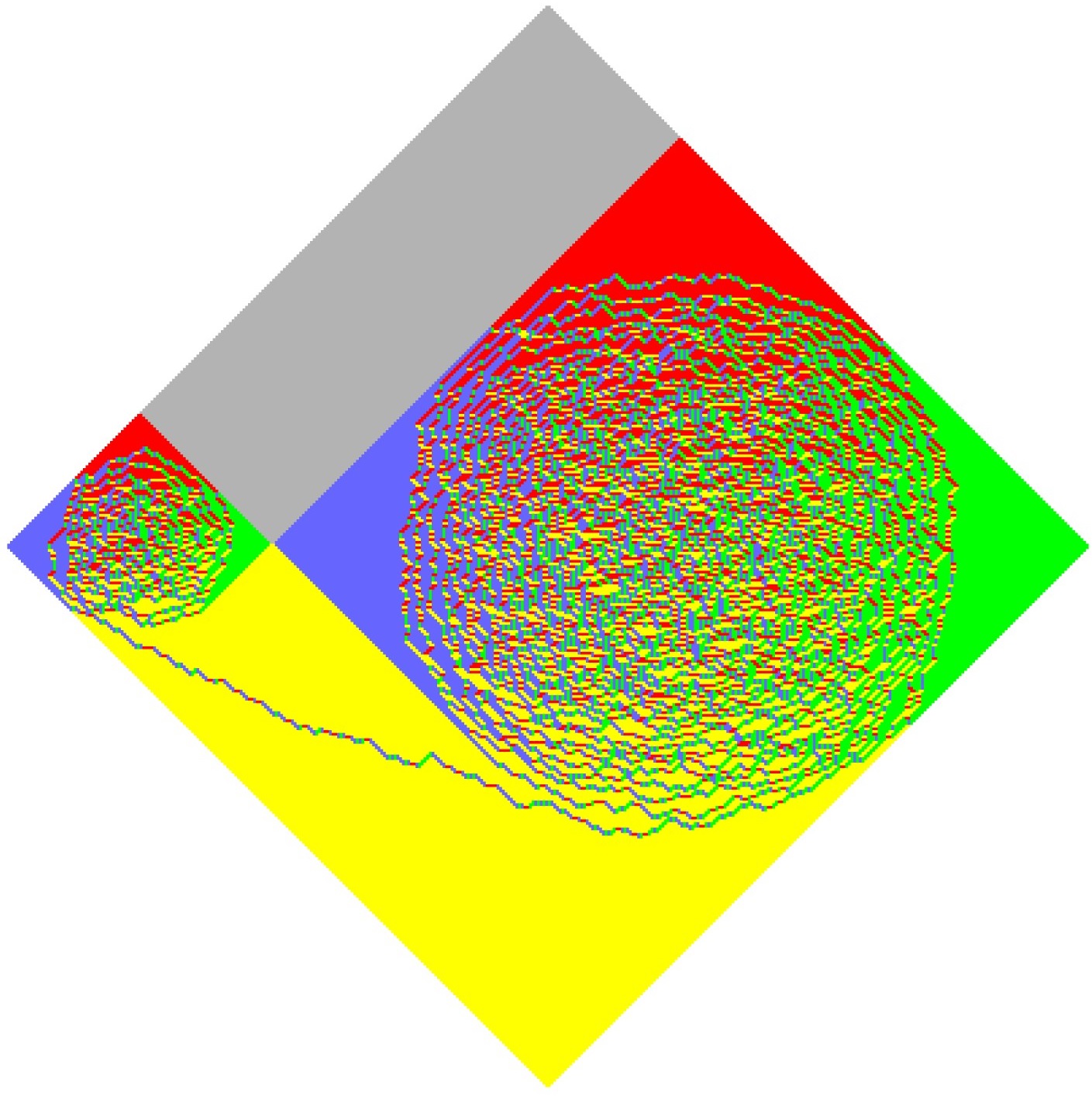}};
\end{tikzpicture}
\begin{tikzpicture}[slave]
\node at (0,0) {\includegraphics[width=6cm]{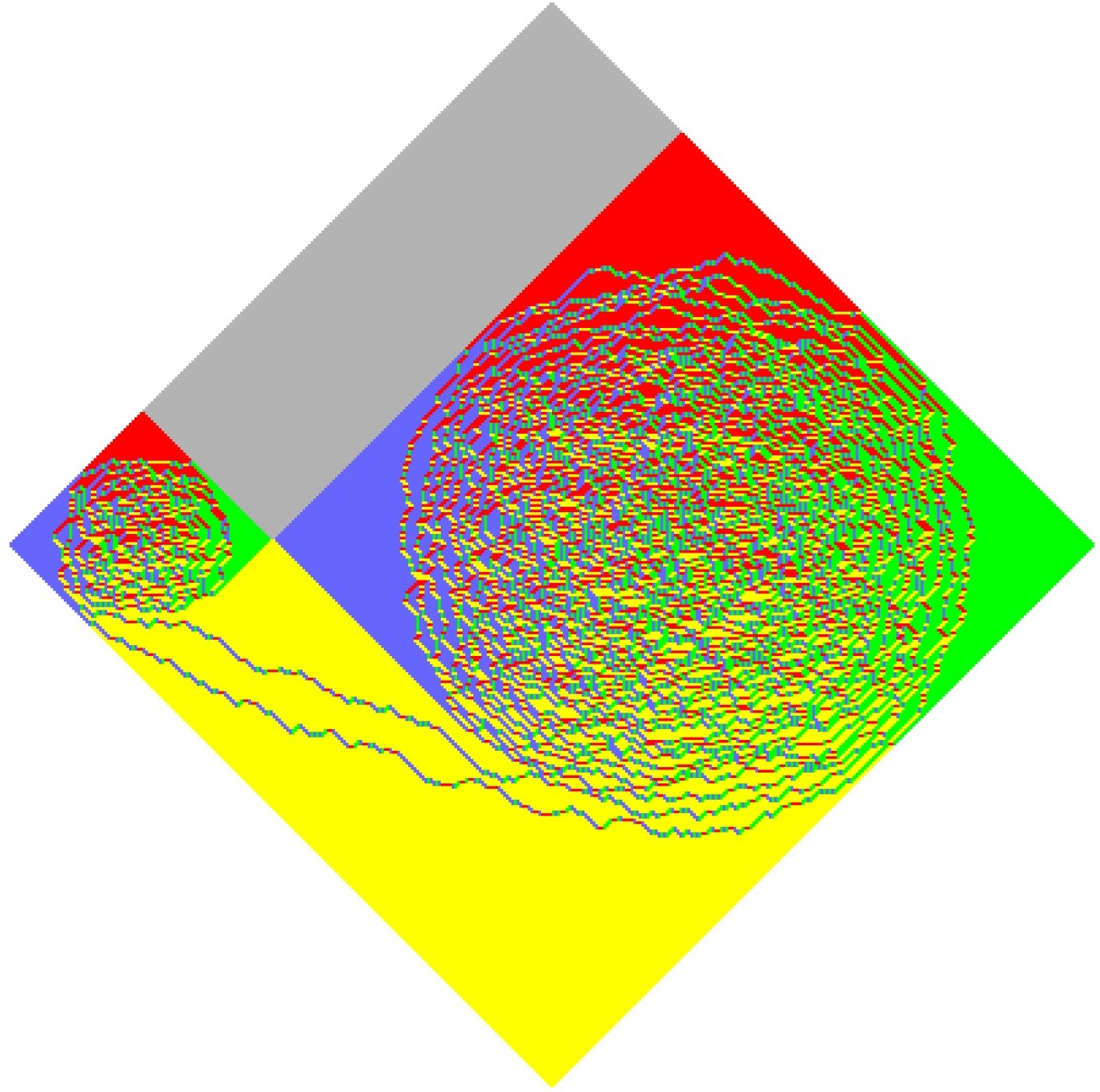}};
\end{tikzpicture}
\end{center}
\caption{\label{fig:one or two paths}Left: a tiling of $A_{200}^{50,2}$ chosen uniformly at random. Right: a tiling of $A_{200}^{50,3}$ chosen uniformly at random.}
\end{figure}

We will prove an explicit expression for the generating function
\be
F_{N}^{m,k}(a):=\sum_{T\in\mathcal T(A_N^{m,k})}a^{v(T)},\qquad 0<a\leq 1,
\ee
where $\mathcal T(A_N^{m,k})$ is the set of all domino tilings of the reduced Aztec diamond $A_N^{m,k}$.
Observe that $F_{N}^{m,k}(1)$ is the number of tilings of $A_N^{m,k}$, and that $F_{N}^{m,k}(a)$ is the partition function for a biased probability measure on $A_N^{m,k}$, for which tilings with fewer vertical dominoes are more likely to occur.

There are a number of noticeable special cases, in which we have simple exact expressions for $F_N^{m,k}(a)$.
\begin{enumerate}
\item For $k=m+1$, no dominoes are removed from the Aztec diamond, such that 
\[F_N^{m,m+1}(a)=F_N(a).\]  
\item We did not admit $k\leq 0$, because in this case $A_{N}^{m,k}$ would not be tileable. This will become clear later on, see Remark \ref{remark:notilings}.
 \item For $k=1$, we will show that the removed region of $(N-m+1)\times m$ north dominoes, implies that the mirror image of this region with respect to the horizontal axis $y=0$ contains only south dominos. The remaining part of $A_N$ consists of two disjoint Aztec diamonds, one of order $m-1$ on the left, and one of order $N-m$ on the right, see Figure \ref{fig:mirror frozen} and Remark \ref{remark:notilings}. We then have
\begin{equation}\label{eq:formulamirror}F_N^{m,1}(a)=F_{m-1}(a)F_{N-m}(a)=\left(1+a^2\right)^{\frac{N(N+1)}{2}-m(N+1-m)}.
\end{equation}
\item For $m=k=N$, only the horizontal top domino is removed from $A_N$. This excludes relatively few tilings, namely the ones with at least one vertical domino on top of the Aztec diamond. But for such tilings, both top dominos are necessarily vertical, and consequently so are all the dominos at the north-west and north-east edges of the Aztec diamond. In other words, the tilings that are excluded correspond to tilings of an Aztec diamond of order $N-1$. We thus have
\[F_N^{N,N}(a)=F_N(a)-a^{2N}F_{N-1}(a).\]
\end{enumerate}

In order to present our result, we need a certain {\em Pad\'e approximation problem}. Consider the degree $N-j+1$ polynomial 
\begin{align}\label{def:f}
f_N^{m,j}(z;a)=z^{m-j}(1-az)^{N-m+1}.
\end{align}
The Padé approximant \cite{Pade} of this function near $-a$ of type $[m-j,j-1]$ is the rational function $p_N^{m,j}(z;a)/q_N^{m,j}(z;a)$, where $p_N^{m,j}$ is a polynomial of degree at most $m-j$ and where $q_N^{m,j}$ is a polynomial of degree $j-1$, which approximates $f_N^{m,j}$ to order $O((z+a)^{m})$ as $z\to -a$. More precisely, $p_N^{m,j}$ and $q_N^{m,j}$ satisfy 
\begin{align}\label{def of p and q in intro}
p_N^{m,j}(z;a)-f_N^{m,j}(z;a)q_N^{m,j}(z;a)=O((z+a)^m),\qquad z\to -a.
\end{align}
Observe that this is a linear system of $m$ equations for the $(m-j+1)+j=m+1$ unknown coefficients of the polynomials  $p_N^{m,j}$ and $q_N^{m,j}$.
We normalize $q_N^{m,j}$ such that $q_N^{m,j}(0)=1$, which leaves us with a linear system of $m$ equations for $m$ unknown coefficients. 
As we will see, this Pad\'e approximant exists uniquely for $N\in\mathbb N$, $m\in\{1,\ldots, N\}$, $j\in\{1,\ldots, m+1\}$.
We denote $\kappa_N^{m,j}(a)$ for the leading coefficient of $p_N^{m,j}(z;a)$. Our first main result, Theorem \ref{theorem:tilings}, states that $\kappa_N^{m,k}(a)$ is equal to the ratio $F_{N}^{m,k+1}/F_{N}^{m,k}$. 
\begin{theorem}\label{theorem:tilings}
{For $N\in \N_{>0}:=\{1,2,\ldots\}$,} $m\in\{1,\ldots, N \}$, $k\in\{1,\ldots, m+1\}$, $0<a\leq 1$, we have 
\begin{align}\label{eq in main thm}
F_{N}^{m,k}(a)=(1+a^2)^{\frac{N(N+1)}{2}}\prod_{j=k}^{m}\frac{1}{\kappa_N^{m,j}(a)}.
\end{align}
In particular, the number of tilings of $A_N^{m,k}$ is equal to $2^{\frac{N(N+1)}{2}}\prod_{j=k}^{m}\frac{1}{\kappa_N^{m,j}(1)}$.
\end{theorem}
\begin{remark}\label{remark:numerical check}
The right-hand side of \eqref{eq in main thm} can be evaluated numerically by solving the associated linear systems in \eqref{def of p and q in intro}. On the other hand, the left-hand side of \eqref{eq in main thm} can be computed independently using Propp's shuffling algorithm \cite{ProppShuffling}. Using this approach, we verified Theorem \ref{theorem:tilings} numerically for $N=1,\ldots,10$, all admissible values of $m,k$, and several values of $a$. We verified many results in this paper numerically in the same way, such as Theorems \ref{theorem:hextiling}, \ref{theorem:tilingmulti} and \ref{thm:hexagon multi gap} below.
\end{remark}
\begin{remark}
Since $F_N^{m,1}(a)=\left(1+a^2\right)^{\frac{N(N+1)}{2}-m(N+1-m)}$, we find the remarkable identity
\[\prod_{j=1}^{m}{\kappa_N^{m,j}(a)}=\left(1+a^2\right)^{m(N+1-m)}.\]
As a consequence, we can write our result also as
\[F_{N}^{m,k}(a)=(1+a^2)^{\frac{N(N+1)}{2}-m(N+1-m)}\prod_{j=1}^{k-1}\kappa_N^{m,j}(a).\]
\end{remark}
\begin{remark}
Let us count tilings in the simple cases $N=2$ and $N=3$:
\begin{itemize}
\item When $N=2$, we have
$F_2(1)=8$. For the reduced Aztec diamonds $A_2^{m,k}$, we count
\[\left(F_2^{m,k}(1)\right)_{k,m=1}^{3,2}=
\begin{pmatrix}
2&2\\8&6\\
-&8
\end{pmatrix},\qquad \left(F_2^{m,k+1}(1)/F_2^{m,k}(1)\right)_{k,m=1}^{2}=
\begin{pmatrix}
4&3\\-&4/3
\end{pmatrix}
.\]
The corresponding functions $f_N^{m,k}$ and their Pad\'e approximants are
\[\left(f_2^{m,k}(z;1)\right)_{k,m=1}^{2}=
\begin{pmatrix}
(1-z)^2&z(1-z)\\
-&1-z
\end{pmatrix},\quad \left(\frac{p_2^{m,k}(z;1)}{q_2^{m,k}(z;1)}\right)_{k,m=1}^{2}=
\begin{pmatrix}
4&3z+1\\
-&\frac{4/3}{1+z/3}
\end{pmatrix}
.\]
\item Similarly, for $N=3$, we have
$F_3(1)=64$ and
\begin{align*}
& \left(F_3^{m,k}(1)\right)_{k,m=1}^{4,3}=
\begin{pmatrix}
8 & 4 & 8 \\
64 & 32 & 32 \\
- & 64 & 56 \\
- & - & 64
\end{pmatrix}, & & \left(f_3^{m,k}(z;1)\right)_{k,m=1}^{3}=
\begin{pmatrix}
(1-z)^3 & z(1-z)^{2} & z^{2}(1-z) \\
- & (1-z)^{2} & z(1-z) \\
- & - & 1-z
\end{pmatrix}, \\
& \left(\frac{F_3^{m,k+1}(1)}{F_3^{m,k}(1)}\right)_{k,m=1}^{3}=
\begin{pmatrix}
8 & 8 & 4 \\
- & 2 & \frac{7}{4} \\
- & - & \frac{8}{7}
\end{pmatrix}, & & \left(\frac{p_3^{m,k}(z;1)}{q_3^{m,k}(z;1)}\right)_{k,m=1}^{3}=
\begin{pmatrix}
8 & 8z+4 & 4z^{2}+3z+1 \\
- & \frac{2}{1+\frac{z}{2}} & \frac{\frac{7}{4}z+\frac{1}{4}}{1+\frac{z}{4}} \\
- & - & \frac{\frac{8}{7}}{1+\frac{4}{7}z+\frac{1}{7}z^{2}}
\end{pmatrix}.
\end{align*}
\end{itemize} 
In both cases, the leading coefficients of the numerators in the matrix $(p_N^{m,k}(z;1)/q_N^{m,k}(z;1))_{k,m=1}^{N}$ match with the entries of $(F_N^{m,k+1}(1)/F_N^{m,k}(1))_{k,m=1}^{N}$, which provides a consistency check of our result.
\end{remark}
\begin{remark}
We now list special cases of $\kappa_{N}^{m,k}$ for general $N,m$ and special values of $k$.
\begin{enumerate}
\item For $k=1$, we have $q_N^{m,1}(z;a)=1$, such that $p_N^{m,1}(z;a)$ is the degree $m-1$ Taylor polynomial of $f_N^{m,1}(z;a)$ around $-a$. We then have
\begin{align*}
\frac{F_N^{m,2}(a)}{F_N^{m,1}(a)}=\kappa_N^{m,1}(a)=(1+a^{2})^{N-m+1} \sum_{v=0}^{m-1} \sum_{j=0}^{v} (-1)^{1+m+v} \binom{m}{j}\binom{N-m+1}{v-j} \bigg(\frac{a^2}{1+a^2} \bigg)^{v-j}
\end{align*}
after a straightforward computation.
\item For $k={m}$, we have
$f{_N^{m,m}(z;a)}=(1-az)^{N-m+1}$. The Padé approximant is then equal to the inverse of the Taylor approximant of $\frac{1}{(1-az)^{N-m+1}}$ around $-a$ of degree $m-1$, and 
\begin{align*}
\frac{F_N^{m,m+1}(a)}{F_N^{m,m}(a)} =\kappa_N^{m,m}(a) =\frac{(1+a^{2})^{N-m+1}}{\sum_{j=0}^{m-1} \binom{N-m+j}{j}\big(\frac{a^2}{1+a^2}\big)^{j}}.
\end{align*}
\end{enumerate}
\end{remark}

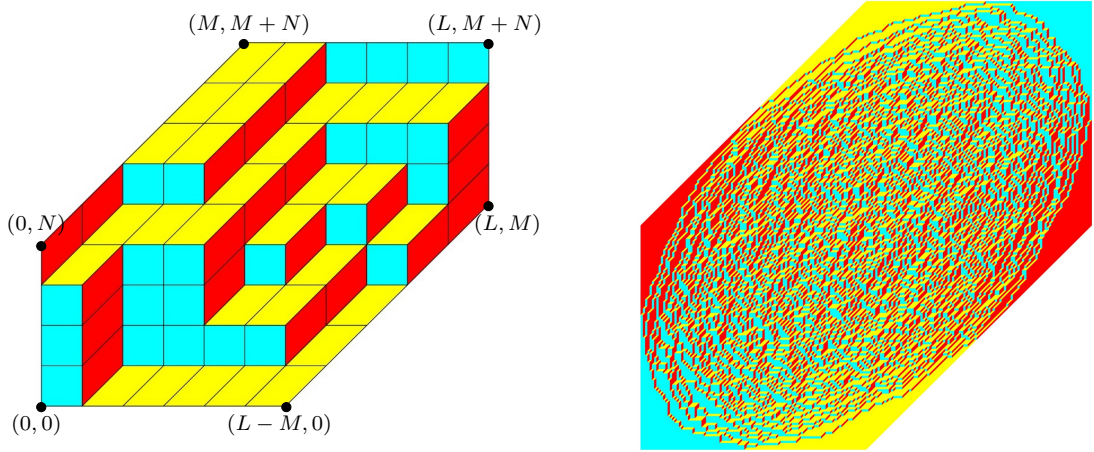
\begin{figure}
\begin{center}
\begin{tikzpicture}[master]
\node at (0,0) {\includegraphics[width=6cm]{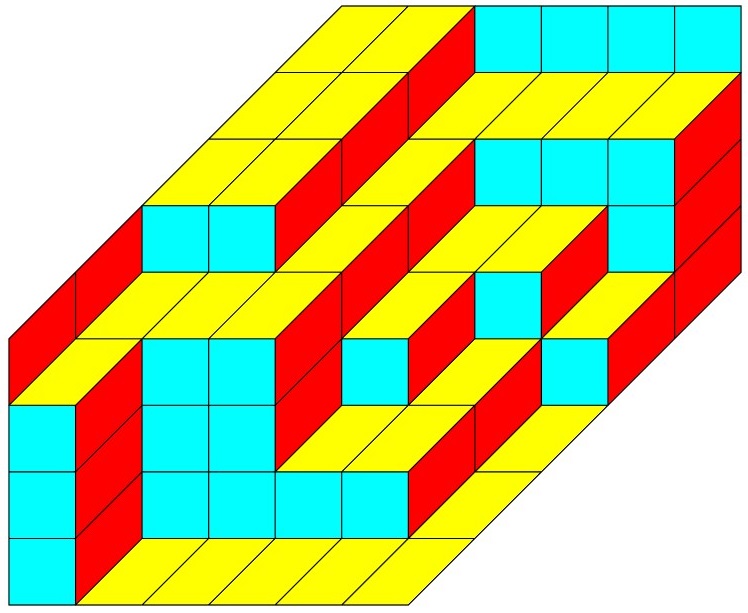}};
\node at (-3,-2.65) {\footnotesize $(0,0)$};
\draw[black,fill] (-2.935,-2.415) circle (0.6mm);
\node at (0.2,-2.65) {\footnotesize $(L-M,0)$};
\draw[black,fill] (0.29,-2.415) circle (0.6mm);
\node at (3.2,0) {\footnotesize $(L,M)$};
\draw[black,fill] (2.95,0.25) circle (0.6mm);
\node at (2.9,2.65) {\footnotesize $(L,M+N)$};
\draw[black,fill] (2.95,2.4) circle (0.6mm);
\node at (-0.2,2.65) {\footnotesize $(M,M+N)$};
\draw[black,fill] (-0.27,2.4) circle (0.6mm);
\node at (-3,0) {\footnotesize $(0,N)$};
\draw[black,fill] (-2.935,-0.28) circle (0.6mm);
\end{tikzpicture} \hspace{0.4cm}
\begin{tikzpicture}[slave]
\node at (0,0) {\includegraphics[width=6cm]{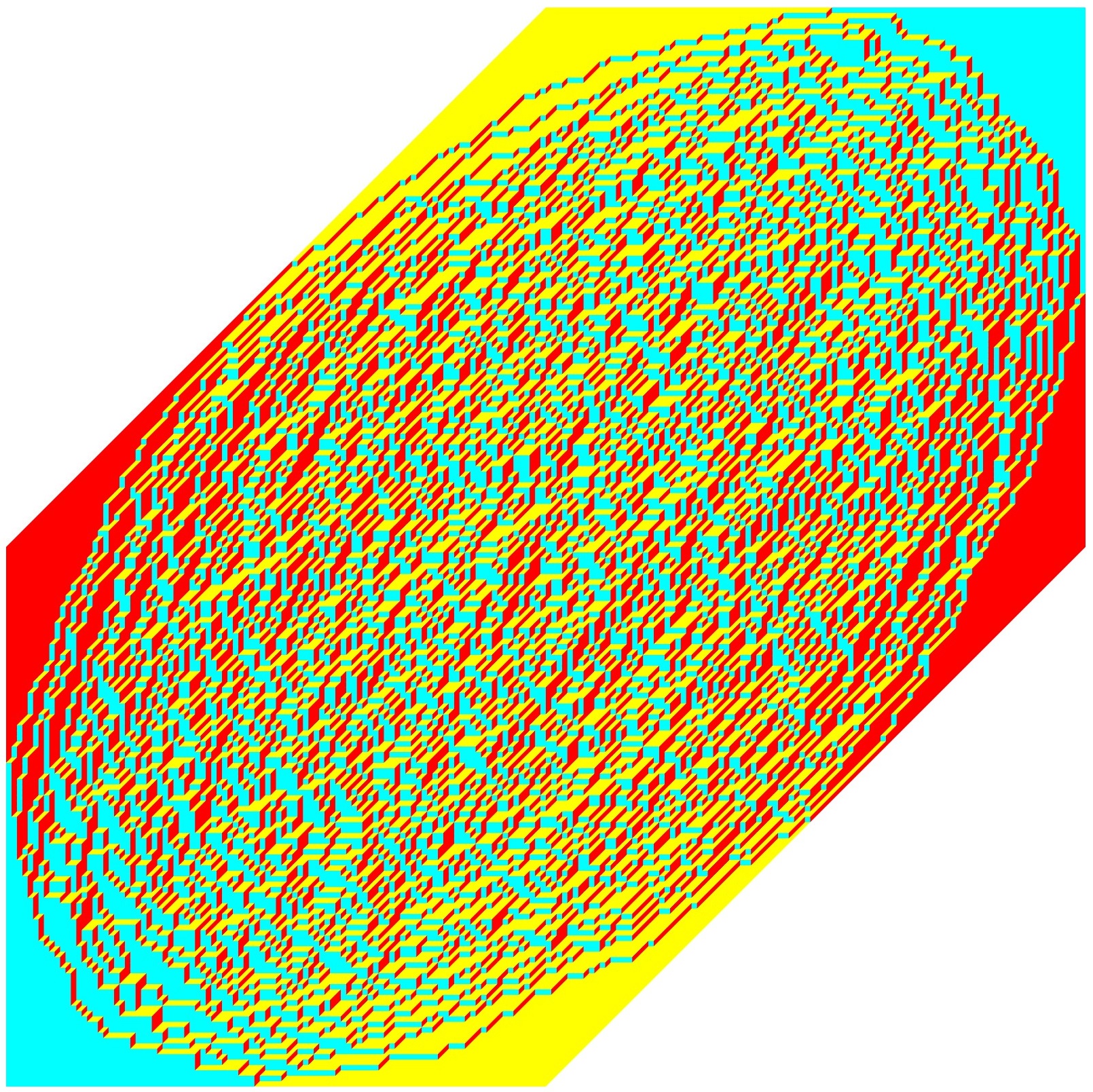}};
\end{tikzpicture}
\end{center}
\caption{\label{fig:two types of hexagons}
Left: a tiling of $H_{L,M,N}$ with $L=11$, $M=5$ and $N=4$. Right: a tiling of $H_{200,100,100}$ chosen uniformly at random.}
\end{figure}

\subsection{Lozenge tilings of reduced hexagons}
\paragraph{Lozenge tilings of hexagons.}
Given $L,M,N\in\mathbb N_{>0}$ with $M<L$, we consider the hexagon $H_{L,M,N}$ in the plane with corners $(0,0), (L-M,0), (L,M), (L,M+N), (M,M+N), (0,N)$ as shown in Figure \ref{fig:two types of hexagons} (left). We consider three types of lozenges {(\tikz[scale=.25,baseline={([yshift=-0.5ex]current bounding box.center)}]{\draw (0,0)--(0,1)--(1,2)--(1,1)--(0,0);}, \tikz[scale=.25,baseline={([yshift=-0.5ex]current bounding box.center)}]{ \draw (0,0)--(0,1)--(1,1)--(1,0)--(0,0); }, \tikz[scale=.25,baseline={([yshift=-0.5ex]current bounding box.center)}]{\draw (0,0)--(1,1)--(2,1)--(1,0)--(0,0);})} and consider all possible tilings of $H_{L,M,N}$ by {these} lozenges\footnote{There are two natural ways of representing a tiling of the hexagon: (a) if the hexagon has all inner angles equal to 120 degrees, then the corresponding lozenges are \tikz[scale=.2,baseline={([yshift=-0.5ex]current bounding box.center)}]{ \draw[cm={1,0,0,1,(2.2,0)}] (0,0)--(0,1)--(0.866025,1.5)--(0.866025,0.5)--(0,0); }, \tikz[scale=.2,baseline={([yshift=-0.5ex]current bounding box.center)}] { \draw[cm={1,0,0,1,(2.2,0)}] (0,0)--(0,1)--(0.866025,0.5)--(0.866025,-0.5)--(0,0);}, \tikz[scale=.2,baseline={([yshift=-0.5ex]current bounding box.center)}] {  \draw[cm={1,0,0,1,(3,0)}] (0,0)--(0.866025,0.5)--(1.73205,0)--(0.866025,-0.5)--(0,0);}, while (b) if the hexagon has inner angles of 135, 90, and 135 degrees, then the corresponding parallelograms are \tikz[scale=.2,baseline={([yshift=-0.5ex]current bounding box.center)}]{\draw (0,0)--(0,1)--(1,2)--(1,1)--(0,0);}, \tikz[scale=.2,baseline={([yshift=-0.5ex]current bounding box.center)}]{ \draw (0,0)--(0,1)--(1,1)--(1,0)--(0,0); }, \tikz[scale=.2,baseline={([yshift=-0.5ex]current bounding box.center)}]{\draw (0,0)--(1,1)--(2,1)--(1,0)--(0,0);}. In this paper, we will only use option (b), as it allows for a simpler coordinate system. We will also refer to \tikz[scale=.2,baseline={([yshift=-0.5ex]current bounding box.center)}]{\draw (0,0)--(0,1)--(1,2)--(1,1)--(0,0);}, \tikz[scale=.2,baseline={([yshift=-0.5ex]current bounding box.center)}]{ \draw (0,0)--(0,1)--(1,1)--(1,0)--(0,0); }, \tikz[scale=.2,baseline={([yshift=-0.5ex]current bounding box.center)}]{\draw (0,0)--(1,1)--(2,1)--(1,0)--(0,0);} as ``lozenges" to match with existing literature, and also because they correspond to true lozenges in the coordinate system (a).}.
The number of such tilings is given by MacMahon's formula \cite{MacMahon} and equal to 
\begin{equation}\label{def:GN}
G_{L,M,N}=\prod_{i=1}^{L-M}\prod_{j=1}^M\prod_{k=1}^N\frac{i+j+k-1}{i+j+k-2}.
\end{equation}
Similarly as for domino tilings of the Aztec diamond, one observes a frozen and a rough region in uniformly random lozenge tilings of hexagons. For instance, in the regular hexagon with $N=M=L-M$, the frozen and rough regions are separated approximately for large $N$ by the ellipse inscribed in the hexagon, see Figure \ref{fig:two types of hexagons} (right).  
Exact expressions for the number of lozenge tilings of various domains close to hexagons, like hexagons with dents, have been obtained e.g. in \cite{ByunLai, Ciucu1, Ciucu2, Ciucu3, Ciucu4}.
Lozenge tiling of polygons have been widely studied, see e.g.\ \cite{Baiketal, Jptrf, Gorin, Petrov1, Petrov2, AvMJ, Amol2023, HYZ2024} for asymptotic results on the uniform measure, and \cite{DK2021, CDKL2020, C doubly, C MVOP, GK2021, Pior2024, K2024} for periodic weightings on the hexagon. We also refer to \cite{GorinSurvey} for a recent survey.

\paragraph{Lozenge tilings of reduced hexagons and Hermite-Padé approximation.}
We will remove a parallelogram from a corner of the hexagon, which will affect the number of lozenge tilings of the domain and, if the parallelogram is sufficiently large, will also affect the asymptotic boundary curve.
For $r\in\{1,\dots, L-1\}$ and $\max\{N,N-L+M+r\}\leq k\leq \min\{M+N-1,r+N-1\}$, we define $H_{L,M,N}^{r,k}$ to be the intersection of the hexagon $H_{L,M,N}$ with the region $\{(x,y)\in\mathbb R^2: y\leq \max\{k, x+k-r\}\}$, see Figure \ref{fig:hexagons grey regions examples}. We denote $G_{L,M,N}^{r,k}$ for the number of lozenge tilings of this domain.

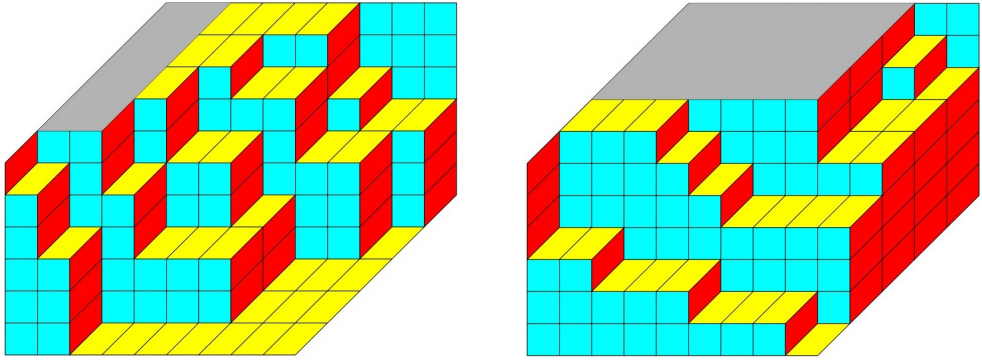
\begin{figure}
\begin{center}
\begin{tikzpicture}[master]
\node at (0,0) {\includegraphics[width=6cm]{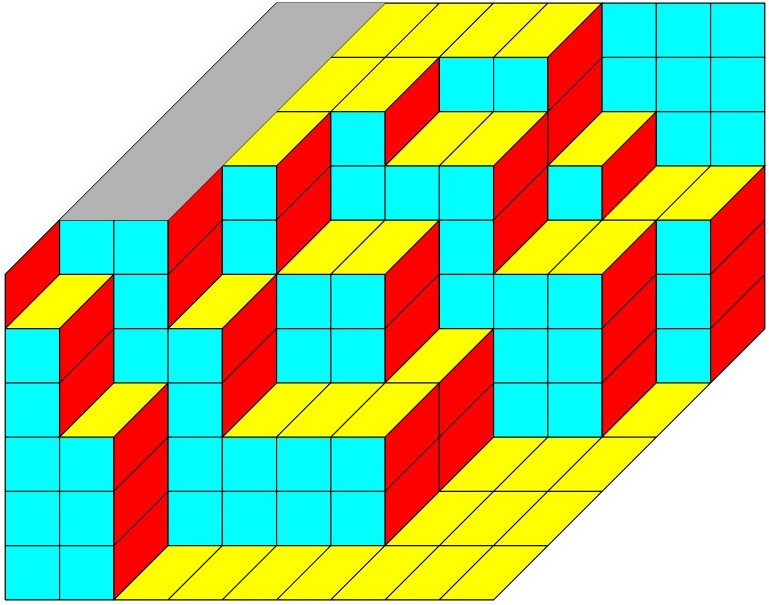}};

\node at (-3,-2.6) {\footnotesize $(0,0)$};
\draw[black,fill] (-2.96,-2.33) circle (0.6mm);
\node at (0.7,-2.6) {\footnotesize $(L-M,0)$};
\draw[black,fill] (0.825,-2.33) circle (0.6mm);
\node at (3.2,-0.5) {\footnotesize $(L,M)$};
\draw[black,fill] (2.96,-0.23) circle (0.6mm);
\node at (2.9,2.55) {\footnotesize $(L,M+N)$};
\draw[black,fill] (2.96,2.35) circle (0.6mm);
\node at (-3.4,0.25) {\footnotesize $(0,N)$};
\draw[black,fill] (-2.96,0.25) circle (0.6mm);

\node at (-1.7,0.85) {\footnotesize $(r,k)$};
\draw[black,fill] (-1.7,0.65) circle (0.6mm);
\node at (-0,2.55) {\footnotesize $(M+N-k+r,M+N)$};
\draw[black,fill] (-0.01,2.35) circle (0.6mm);
\node at (-2.8,0.85) {\footnotesize $(k-N,k)$};
\draw[black,fill] (-2.54,0.65) circle (0.6mm);
\end{tikzpicture} \hspace{-0.3cm}
\begin{tikzpicture}[slave]
\node at (0,0) {\includegraphics[width=6cm]{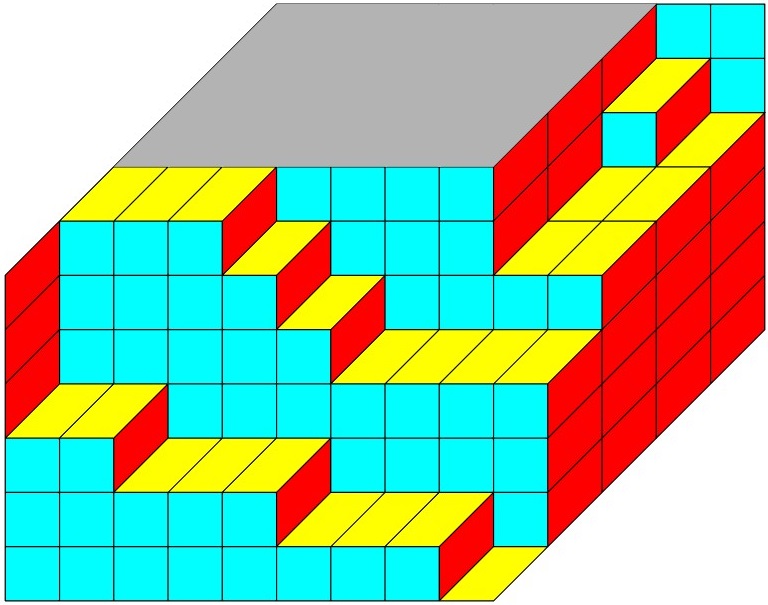}};
\end{tikzpicture}
\end{center}
\caption{\label{fig:hexagons grey regions examples}
Left: a tiling of $H_{L,M,N}^{r,k}$ with $L=14$, $M=5$, $N=6$, $r=3$ and $k=7$. Right: a tiling of $H_{L,M,N}^{r,k}$ with $L=14$, $M=5$, $N=6$, $r=9$ and $k=8$.}
\end{figure}

To state our result, we need the following {\em Hermite-Pad\'e approximation} problem.
We search for a polynomial $q_{M-1}$ of degree at most $M-1$, a polynomial $p_{N-1+r-k}$ of degree at most $N-1+r-k$, and a monic polynomial $P_{L-M-N+k-r}$ of degree $L-M-N+k-r$ which satisfy the conditions
\begin{align}\label{Hermite Padé hexagon}
q_{M-1}(z)-z^{M+N}P_{L-M-N+k-r}(z)+(1+z)^{L-r}z^kp_{N-1+r-k}(z)=\bigO((z+1)^L)\qquad\mbox{as $z\to -1$.}
\end{align}
This is a type I Hermite-Pad\'e approximation problem, see e.g.\ \cite{VanAssche}. In practice it is equivalent to a linear system of $L$ equations for the $M+(L-M-N+k-r)+(N+r-k)=L$ coefficients of the polynomials. We will see that this system has a unique solution. Our second main result, Theorem \ref{theorem:hextiling}, establishes that $p_{N-1+r-k}(0)$ is equal to the ratio $G_{L,M,N}^{r,k+1}/G_{L,M,N}^{r,k}$.

\begin{theorem}\label{theorem:hextiling}
Let $L,M,N\in\mathbb N_{>0}$ with $M<L$, and let $r\in\{1,\dots, L-1\}$, $\max\{N,N-L+M+r\}\leq k\leq \min\{M+N-1,r+N-1\}$.
We have the identity
\[G_{L,M,N}^{r,k}= G_{L,M,N} \prod_{j=k}^{\min\{M+N-1,r+N-1\}}\frac{1}{p_{N-1+r-j}(0)}.\]
\end{theorem}

\subsection{General methodology}

Theorem \ref{theorem:tilings} and Theorem \ref{theorem:hextiling} are applications of a general method that we develop. This method consists of characterizing Fredholm determinants on $\ell^2(\mathbb Z)$ of kernels with a specific double contour integral form in terms of a Riemann-Hilbert (RH) problem. 
A simple example of such kernels arises in the Krawtchouk ensemble which characterizes statistics of the uniform measure on the domino tilings of Aztec diamonds. In this case, the associated RH problem can be solved in terms of the Padé approximation problem \eqref{def of p and q in intro}.
A slightly more complicated example of such a kernel arises for uniformly distributed lozenge tilings of hexagons. In this case, additional steps are needed before one can solve the RH problem explicitly, but in the end it is equivalent to the type I Hermite-Padé approximation problem \eqref{Hermite Padé hexagon}.

In this sense, Theorem \ref{theorem:tilings} and Theorem \ref{theorem:hextiling} serve as simple illustrations of the power of our method, which also applies to other discrete determinantal point processes, like doubly-periodic probability distributions on domino/lozenge tilings, and which we expect also to be useful for asymptotic analysis.

We will now give an informal description of our method, but specialized to the case of domino tilings of the Aztec diamond in order to keep this introduction as accessible as possible.
The proof of Theorem \ref{theorem:tilings} consists of the following steps.
\paragraph{Step 1: Fredholm determinant representation.} First, we express $F_{N}^{m,k}(a)$ as a Fredholm determinant of an operator acting on $\ell^2(\mathbb Z)$.
This step is rather standard, as it relies on well-known results \cite{Jptrf} connecting domino tilings with non-intersecting paths and determinantal point processes.
It will consist of proving the following result.
\begin{proposition}\label{prop:dominoFredholm}
{For} $N\in \N_{>0}$, $m\in\{1,\ldots, N\}$, $k\in\{1,\ldots, m+1\}$, we have
\be\label{eq:detid}
F_N^{m,k}(a)
=F_N(a)\ \det\left(1-1_{\Z\cap [k,+\infty)}K_{N,m}\right)_{\ell^2(\mathbb Z)},\ee
where
$K_{N,m}$ is given by
\begin{align*}
K_{N,m}\left(n,n'\right)=
\frac{1}{(2\pi i)^2}\int_{\Sigma_2}d u \int_{\Sigma_1} d v \frac{u^{-n-1}v^{n'}}{u-v}\frac{W(v)}{W(u)},\qquad n,n'\in\mathbb Z,
\end{align*}
where $\Sigma_1$ is a simple closed contour around $a$ but not enclosing $-1/a$, and $\Sigma_2$ is a simple closed contour around $\Sigma_1$ and $0$, both oriented positively, and
\be\label{def:W}W(z)=\frac{z^{N-m}}{(z-a)^{N-m+1}(1+az)^m}.\ee
\end{proposition}
We will prove this result in Section \ref{section:paths}.
We should note that one can rewrite the determinant on the right hand side of \eqref{eq:detid} as the determinant of an $N\times N$ Hankel matrix associated with a truncation of a (discrete) Krawtchouk weight. This identity was exploited by Colomo and Pronko \cite{CP2013, CP2015} to analyze the free energy for tilings of reduced Aztec diamonds. Such Hankel determinants are connected to discrete RH problems, see e.g.\ \cite{Baiketal}, but these discrete RH problems are of a much more complicated form than the ones we will obtain.

\paragraph{Step 2: Fourier series conjugation.} Next, we use the Fourier series transform and properties of Fredholm determinants to rewrite $F_{N}^{m,k}(a)$ as a
Fredholm determinant of an operator $\mathcal{M}$ acting on $L^2(\Sigma_1\cup\Sigma_2)$ which is of a simpler, \textit{integrable}, form; here $\Sigma_1,\Sigma_2$ will be two disjoint simple closed contours in the complex plane. An operator $\mathcal{M}$ is $n$-integrable in the sense of Its, Izergin, Korepin and Slavnov \cite{IIKS} if its kernel $M$ can be written in the form
\begin{align}\label{def of integrable kernel}
M(z,z') = \frac{\mathrm{g}^{T}(z)\mathrm{h}(z')}{z-z'}, \qquad \mathrm{g}^{T}(z)\mathrm{h}(z)=0,
\end{align}
where $\mathrm{g},\mathrm{h}$ are vector-valued functions of size $n\times 1$.
This part is novel; it can be seen as the discrete counterpart of a method developed in \cite{BertolaCafasso} to simplify the analysis of certain types of Fredholm determinants by conjugating the corresponding operator with the Fourier transform. This method allows to study gap probabilities of infinite determinantal point processes arising in random matrix theory, see e.g.\ \cite{Girotti1,Girotti2,Girotti3, CGS2019}. Recently, a similar method was used for finite determinantal point processes in \cite{ClaeysMauersberger}. Here, we also use it for finite determinantal point processes, but on a discrete lattice. It is the latter feature that naturally leads us to use  the Fourier series transform instead of the Fourier transform.

\begin{proposition}\label{prop:dominointegrable}
For $N\in \N_{>0}$, $m\in\{1,\ldots, N\}$, $k\in\{1,\ldots, m+1\}$, we have
\[\det\left(1-1_{\Z\cap[k,+\infty)}K_{N,m}\right)_{\ell^2(\mathbb Z)}=\det\left(1-\mathcal M\right)_{L^2(\Sigma_1)\oplus L^2(\Sigma_2)},\]
where the integral operator $\mathcal M:L^2(\Sigma_1)\oplus L^2(\Sigma_2)\to L^2(\Sigma_1)\oplus L^2(\Sigma_2)$ has kernel
\begin{align*}
M(z,z')=\frac{1_{\Sigma_2}(z')1_{\Sigma_1}(z) ( \frac{z}{z'} )^{k} - 1_{\Sigma_1}(z')1_{\Sigma_2}(z)\frac{W(z')}{W(z)}}{2\pi i (z-z')}.
\end{align*}
\end{proposition}
We will prove this result in Section \ref{section:Fourier}; in fact, we will prove a more general result in which we replace $[k,+\infty)$ by a finite union of intervals, and in which we take a larger class of kernels $K$. 
We present this method and result in a much larger setting, because it can be used for other discrete determinantal point processes, in particular doubly-periodic domino tilings of Aztec diamonds as well as doubly-periodic lozenge tilings of hexagons (when restricted to a single level, so that the point process lies in a subset of $\Z$).
We study a class of discrete determinantal point processes whose kernel admits a double contour integral representation of the form
\begin{align}\label{eq:def K intro}
K\left(n,n'\right)=
\frac{1}{(2\pi i)^2}\int_{\Sigma_2}d u \int_{\Sigma_1}d v \frac{u^{-n-1}v^{n'}G(u,v)}{u-v},\qquad n,n'\in\mathbb Z,
\end{align}
where $\Sigma_1$ is the unit circle in the complex plane and $\Sigma_2$ is a simple closed contour outside the unit disk. The function $G(u,v)$ takes the form 
\begin{align*}
G(u,v)=\sum_{j=1}^{d} g_j(u)h_j(v), \qquad d\geq 1,
\end{align*}
for some functions $\{g_j, h_j\}_{j=1}^{d}$.

\paragraph{Step 3: RH characterization.}
Then we express $F_{N}^{m,k+1}(a)/F_{N}^{m,k}(a)$ in terms of the solution $U$ to the following RH problem.
\subsubsection*{RH problem for $U$}
\begin{itemize}
\item[(a)] $U:\mathbb C\setminus\left(\Sigma_1\cup\Sigma_2\right)\to \mathbb C^{2\times 2}$ is analytic, where $\Sigma_1$ is a simple positively oriented closed contour around $a$ but not enclosing $-1/a$, and $\Sigma_2$ is a simple positively oriented contour encircling $\Sigma_1$ and $0$.
\item[(b)] $U$ admits continuous boundary values $U_\pm(z)$ as $z\in\Sigma_1\cup\Sigma_2$ is approached from the left ($+$) or right ($-$) according to the orientation of the contour, and they are related by
$U_+(s)=U_-(z)J_U(z)$ for $z\in \Sigma_1\cup\Sigma_2$, with
\begin{align*}
J_U(z)=\begin{cases}
\begin{pmatrix}
1& \frac{z^{N-m+k}}{(z-a)^{N-m+1}(1+az)^m}\\ 0& 1
\end{pmatrix},&z\in\Sigma_1,\\
\begin{pmatrix}
1& 0\\ -\frac{(z-a)^{N-m+1}(1+az)^m}{z^{N-m+k}}
& 1
\end{pmatrix},&z\in\Sigma_2.\end{cases}
\end{align*}\item[(c)]{As $z\to \infty$, $U(z)=I_{2}+\bigO(z^{-1})$}.
\end{itemize}

\begin{proposition}\label{prop:dominoRH}
For $N\in \N_{>0}$, $m\in\{1,\ldots, N\}$, $k\in\{1,\ldots, m+1\}$, the RH problem for $U$ is uniquely solvable, and we have
\begin{align*}
\frac{\det\left(1-1_{\Z\cap [k+1,+\infty)}K_{N,m}\right)}{\det\left(1-1_{\Z\cap [k,+\infty)}K_{N,m}\right)}=U_{11}(0).
\end{align*}
\end{proposition}
We will prove this result in Section \ref{section:RH}; again, we will prove a more general version of this result corresponding to several intervals, and to a more general class of kernels.

\paragraph{Step 4: Solve the RH problem explicitly in terms of a Pad\'e approximation problem.}
Finally, we construct an explicit solution of the RH problem for $U$ in terms of a Pad\'e approximation problem, in the special case corresponding to domino tilings of the Aztec diamond, i.e.\ when $W$ is given by \eqref{def:W}.  In particular we prove the following.
\begin{proposition}\label{prop:dominoPade}
For $N\in \N_{>0}$, $m\in\{1,\ldots, N\}$, $k\in\{1,\ldots, m+1\}$, we have 
\begin{align*}
U_{11}(0)=\kappa_{N}^{m,k}(a),
\end{align*}
where $\kappa_{N}^{m,k}(a)$ is the leading coefficient of $p_{N}^{m,k}(z;a)$, with $p_{N}^{m,k}(z;a)/q_{N}^{m,k}(z;a)$ the
type $[m-k,k-1]$ Pad\'e approximant of the function $f_{N}^{m,k}(z;a)$, as defined before Theorem \ref{theorem:tilings}.
\end{proposition}
We will prove this result in Section \ref{section:tiling}.
Finally, Theorem \ref{theorem:tilings} is obtained by taking a telescopic product
\[\prod_{j=k}^m\frac{\det\left(1-1_{\Z\cap [j+1,+\infty)}K_{N,m}\right)}{\det\left(1-1_{\Z\cap [j,+\infty)}K_{N,m}\right)}=\prod_{j=k}^m\kappa_N^{m,j}(a)\]
and by observing that $\det\left(1-1_{\Z\cap [m+1,+\infty)}K_{N,m}\right)=1$.

\section{Tilings and non-intersecting paths}\label{section:paths}

\subsection{Domino tilings of the Aztec diamond}

We consider the Aztec diamond $A_N$. 
Following \cite{Joh2}, we can associate a collection of non-intersecting paths to a domino tiling by applying the following simple rule. On each south domino $[a,a+2]\times [b,b+1]$, we draw a horizontal line connecting $(a,b+1/2)$ with $(a+2,b+1/2)$. On each west domino $[a,a+1]\times [b,b+2]$, we draw an ascending diagonal line segment connecting $(a,b+1/2)$ with $(a+1,b+3/2)$. On each east domino $[a,a+1]\times[b,b+2]$, we draw a descending diagonal line segment connecting $(a,b+3/2)$ with $(a+1,b+1/2)$. No lines are drawn on north dominoes. One then checks easily that these line segments altogether form a configuration of $N$ non-intersecting paths connecting the south-west edge of the Aztec diamond with the south-east edge, see Figure \ref{fig:paths} (left). Such path configurations are in bijection with the set of all domino tilings of $A_N$, as one can easily reconstruct the domino tiling by covering horizontal line segments by south dominoes, diagonal line segments by vertical dominoes, and filling in the remaining holes with north dominoes.

\begin{figure}
\begin{center}
\hspace{-1cm}\begin{tikzpicture}[master]
\node at (0,0) {\includegraphics[width=6cm]{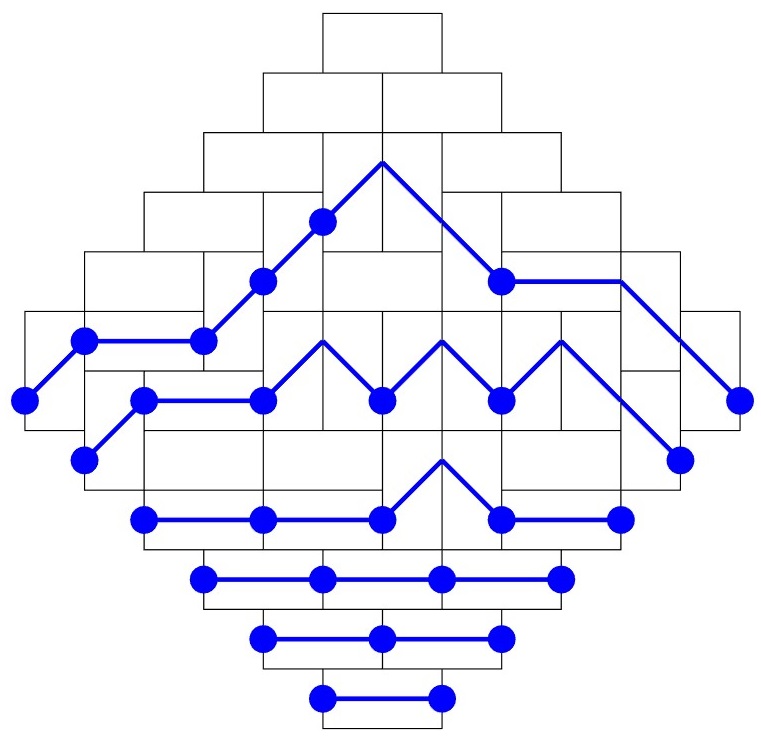}};
\end{tikzpicture} \hspace{1.5cm}
\begin{tikzpicture}[slave]
\node at (0,0) {\includegraphics[width=6cm]{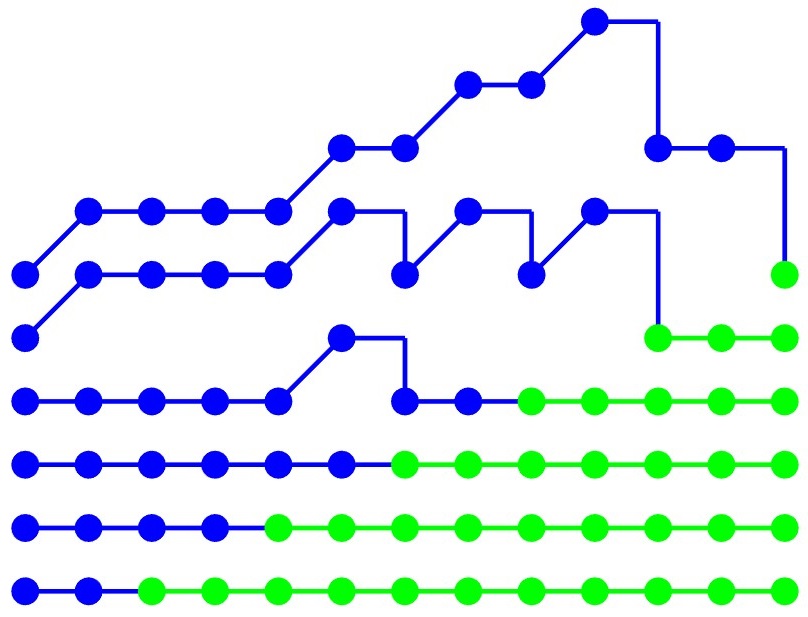}};
\draw[->-=1] (-4,-2.7)--(3.3,-2.7);
\node at (3.2,-2.95) {$r$};
\draw[->-=1] (-4,-2.7)--(-4,2.5);
\node at (-3.8,2.5) {$k$};
\node at (-3.3,0.25) {\footnotesize $(0,0)$};
\node at (-3.61,-2.1) {\footnotesize $(0,1-N)$};
\node at (3.4,0.25) {\footnotesize $(2N,0)$};
\node at (3.73,-2.1) {\footnotesize $(2N,1-N)$};
\end{tikzpicture}
\end{center}
\caption{\label{fig:paths}Left: a tiling of $A_{6}$, and the associated $6$ non-intersecting paths. Right: the corresponding system of $6$ non-intersecting paths on the lattice $\{0,1,\ldots,2N\}\times \Z$. The green dots are added for convenience; they are independent of the tiling. The dots in the left figure correspond to the blue points in the right figure at even time levels.}
\end{figure}
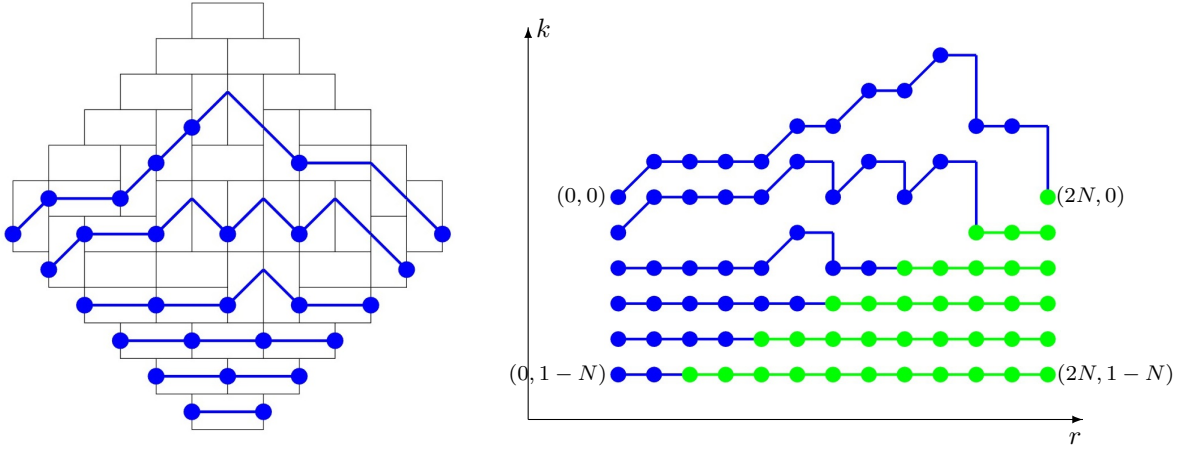

In a next step, we transform the configuration of $N$ non-intersecting paths to a configuration of $N$ non-intersecting paths on the lattice $\{0,1,\ldots,2N\}\times \mathbb Z$, with starting points $(0,0), (0,-1), (0,-2), \ldots,$ and endpoints $(2N,0), (2N,-1), (2N,-2),\ldots$.
To explain how this transformation works, we start with the highest path in the domino tiling, starting at $(-N,-1/2)$ and ending at $(N,-1/2)$. This will correspond to the path on $\{0,1,\ldots,2N\} \times \mathbb Z$ starting at $(0,0)$ and ending at $(2N,0)$.
We replace every ascending diagonal segment $(a,b+1/2)\to (a+1,b+3/2)$ in the domino tiling by two steps: a diagonal step $(\alpha,\beta)\to (\alpha+1,\beta+1)$ followed by a horizontal step $(\alpha+1,\beta+1)\to (\alpha+2,\beta+1)$; we replace every horizontal segment $(a,b+1/2)\to (a+2,b+1/2)$ by two horizontal steps $(\alpha,\beta)\to (\alpha+1,\beta)\to (\alpha+2,\beta)$; finally, we replace every descending diagonal segment $(a,b+3/2)\to (a+1,b+1/2)$ by a vertical down-step $(\alpha,\beta)\to (\alpha,\beta-1)$. By following the first path in the domino tiling, we obtain a path connecting $(0,0)$ with $(2N,0)$. 
By applying the same procedure on the second path in the domino tiling, we obtain a path connecting $(0,-1)$ with $(2N-2,-1)$, which we complement with horizontal steps connecting $(2N-2,-1)$ with $(2N,-1)$.
By doing the same for the $(j+1)$-th path in the domino tiling, we obtain a path connecting $(0,-j)$ with $(2N-2j,-j)$, which we complement with horizontal steps connecting $(2N-2j,-j)$ with $(2N,-j)$. 
See Figure \ref{fig:paths} (right).

Finally, we associate points to this collection of paths on $\{0,1,\ldots,2N\}\times \mathbb Z$, by marking each point $(\alpha,\beta)\in  \{0,1,\ldots,2N\}\times\mathbb Z$ belonging to a path, except the points $(\alpha,\beta)$ which are starting points of a down-step $(\alpha,\beta)\to (\alpha,\beta-1)$. This results in a point configuration of $(2N+1)\times N$ points: each path delivers exactly one point on each vertical line $\{r\}\times \mathbb Z$, $r=0,1,\ldots, 2N$, see again Figure \ref{fig:paths} (right). One can revert this whole construction from domino tiling to point configuration, thus proving that there is a bijection between domino tilings and admissible point configurations. It is sometimes convenient to draw the points corresponding to even time levels also in the domino tiling, as shown in Figure \ref{fig:paths} (left).

\medskip

Now, we introduce a probability measure on the set of domino tilings: to a domino tiling $T\in \mathcal T(A_N)$, we assign the probability
\begin{align}\label{prob measure}
\mathbb P(T)=\frac{a^{v(T)}}{(1+a^2)^{\frac{N(N+1)}{2}}},\qquad 0<a\leq 1,
\end{align}
where $v(T)$ is the number of vertical dominoes in $T$. Recall the identity \eqref{eq:ADT}, which implies that \eqref{prob measure} is indeed a probability measure. Note also that $a=1$ corresponds to the uniform measure.
Johansson \cite{Joh2} proved that under this probability distribution, the corresponding configurations of points on $\{0,1,\ldots,2N\} \times \mathbb Z$ form a determinantal point process on $\{0,1,\ldots,2N\} \times \mathbb Z$. If we restrict to the $N$ points at level $r$, i.e.\ the points on $\{r\}\times \mathbb Z$, these $N$ points $x_1>\ldots>x_N$ alone also form a determinantal point process on $\mathbb Z$.
The correlation kernel is given by \cite[Formula (2.17)]{Joh2}
\begin{align}\label{eq:def K0}
K_{N,m}\left(n,n'\right)=
\frac{1}{(2\pi i)^2}\int_{|u|=r_2}d u \int_{|v|=r_1} d v \frac{u^{n-1}v^{-n'}}{v-u}\frac{\widetilde W(v)}{\widetilde W(u)},\qquad n,n'\in\mathbb Z,
\end{align}
where, if $a\neq 1$, $a<r_1<1/a$, $0<r_2<r_1$, and
\be\label{def:Wtilde}\widetilde W(z)=\frac{1}{(1-az)^{N-m+\epsilon}(1+a/z)^m},\ee
where $r=2m-\epsilon$, $m\in\{1,\ldots, 2N\}$, $\epsilon\in\{0,1\}$.
Note also that one can take the limit $a\to 1_-$ by suitably shifting the contours in \eqref{eq:def K0}.

By changing $u\mapsto 1/u$, $v\mapsto 1/v$, we obtain that
\begin{align}\label{eq:def K2}
K_{N,m}\left(n,n'\right)=
\frac{1}{(2\pi i)^2}\int_{\Sigma_2}d u \int_{\Sigma_1}d v \frac{u^{-n-1}v^{n'}}{u-v}\frac{W(v)}{W(u)},\qquad n,n'\in\mathbb Z,
\end{align}
where
$\Sigma_1$ goes around $0$ and $a$ but not around $-1/a$, $\Sigma_2$ lies outside $\Sigma_1$,
and
\begin{equation}\label{def:W2}
W(z)=\frac{z^{N-m+\epsilon-1}}{(z-a)^{N-m+\epsilon}(1+az)^m}.
\end{equation}
We omitted the dependence of $K_{N,m}$ in $\epsilon$ in the notation. For $\epsilon=1$, $K_{N,m}$ is precisely the kernel from Proposition \ref{prop:dominoFredholm}.

From the general theory of determinantal point processes, see e.g. \cite{Jptrf}, we then know that the probability distribution of the largest particle $x_1$ is given by
\[\mathbb P\left(x_1< k\right)=\det\left(1-1_{\Z\cap [k,+\infty)}K_{N,m}\right)_{\ell^2(\mathbb Z)},\]
where the determinant on the right-hand side is the Fredholm determinant (see Section \ref{section:Fourier} for a definition and properties of Fredholm determinants).
The event $x_1< k$ corresponds to the event where the first path on $\mathbb Z\times\mathbb Z$ does not cross the half line $\{2m-\epsilon\}\times[k,+\infty)$. From the shape of the paths, it then follows that there is a whole region free of paths, see Figure \ref{fig:gap at odd times} for an illustration in the case $\epsilon=1$.

\begin{figure}
\begin{center}
\hspace{-1cm}\begin{tikzpicture}[master]
\node at (0,0) {\includegraphics[width=6.5cm]{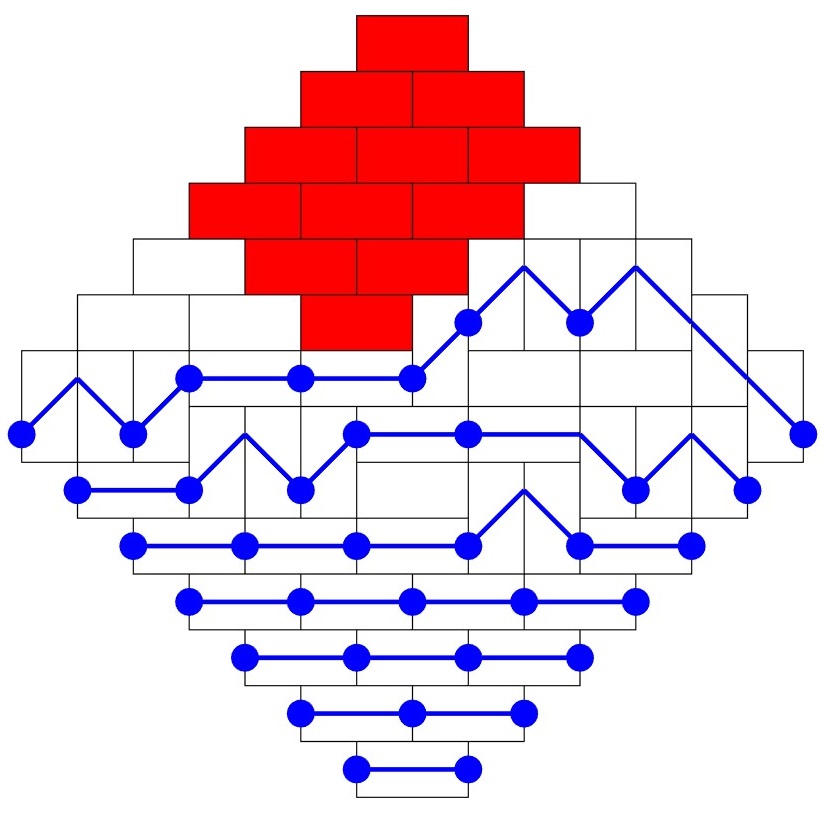}};
\end{tikzpicture} \hspace{0.5cm}
\begin{tikzpicture}[slave]
\node at (0,0) {\includegraphics[width=6.5cm]{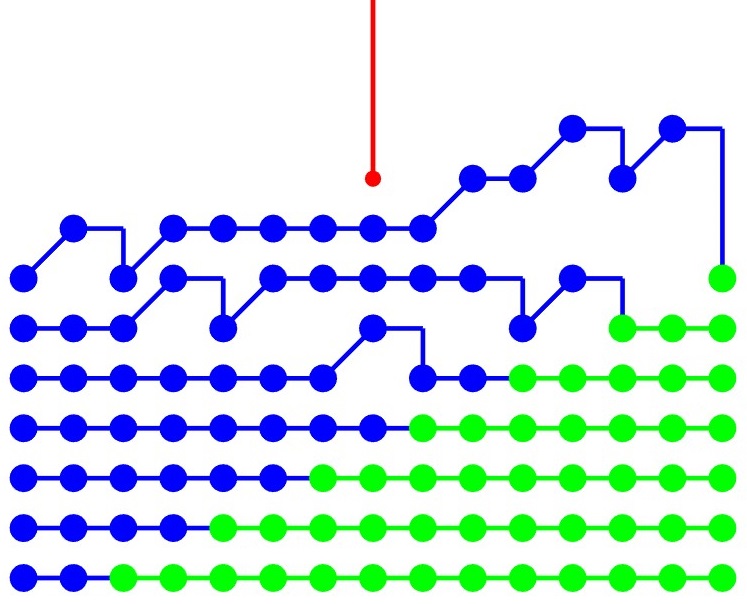}};
\end{tikzpicture}
\end{center}
\caption{\label{fig:gap at odd times}Any tiling of $A_{N}$ with only north dominoes outside the region $A_N^{m,k}$ (left) corresponds to a system of $N$ non-intersecting paths on $\{0,\ldots,2N\}\times \Z$ such that the highest particle at time $2m-1$ is less than $k$ (right). This is illustrated in the picture with $N=7$, $m=4$, $k=2$. The segment on the right is $\{2m-1\}\times [k,+\infty)$; it does not contain any blue points.}
\end{figure}

\begin{figure}
\begin{center}
\begin{tikzpicture}[master]
\node at (0,0) {\includegraphics[width=4.5cm]{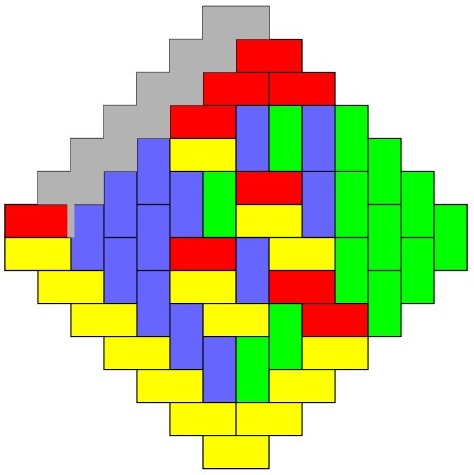}};
\end{tikzpicture}
\begin{tikzpicture}[slave]
\node at (0,0) {\includegraphics[width=4.5cm]{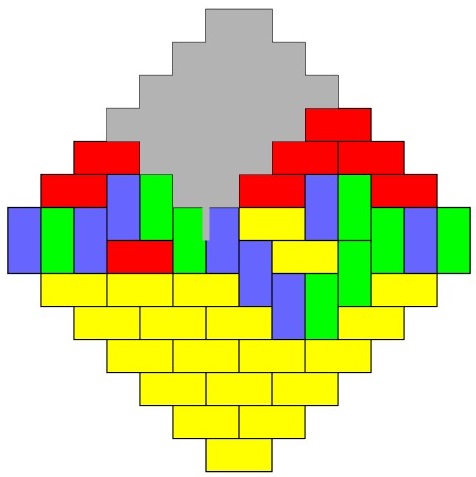}};
\end{tikzpicture}
\begin{tikzpicture}[slave]
\node at (0,0) {\includegraphics[width=4.5cm]{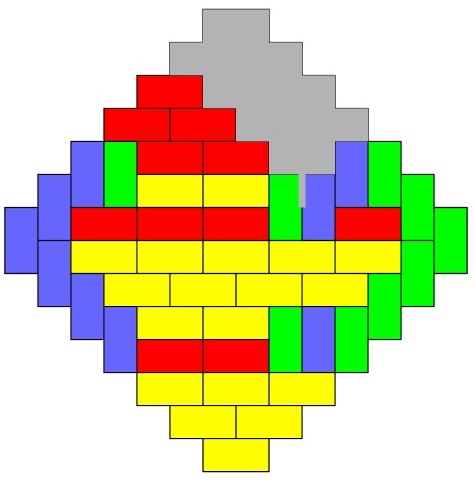}};
\end{tikzpicture}
\end{center}
\caption{\label{fig:ADreduced tilde}Left: a tiling of $\tilde{A}_{7}^{1,1}$. Middle: a tiling of $\tilde{A}_{7}^{3,1}$. Right: a tiling of $\tilde{A}_{7}^{5,2}$. The north, south, east and west dominoes are shown in red, yellow, green and blue, respectively. The set $A_{N}\setminus \tilde{A}_{N}^{m,k}$ is shown in grey; it contains $N-m$ corners on the north-west side, $m-k+1$ corners on the north-east side, as well as a vertical segment at the bottom. }
\end{figure}

{For $\epsilon=1$, the event $x_{1}<k$} corresponds to a domino tiling for which all dominoes in a certain region are north dominoes. More precisely, it corresponds to the event that there are only north dominoes outside the region $A_N^{m,k}$ defined in the introduction.
In other words, we have
\[\mathbb P\left(\mbox{only north dominoes outside }A_N^{m,k}\right)=\det\left(1-1_{\Z\cap [k,+\infty)}K_{N,m}\right)_{\ell^2(\mathbb Z)}.\]
But, by definition of $\mathbb{P}$, we also have
\[\mathbb P\left(\mbox{only north dominoes outside }A_N^{m,k}\right)=\frac{\sum_{T\in\mathcal T(A_N^{m,k})} a^{v(T)}}{(1+a^2)^{\frac{N(N+1)}{2}}}.\]
Combining the two above identities, we obtain Proposition \ref{prop:dominoFredholm}.

Proposition \ref{prop:dominoFredholm} admits a natural analogue for $\epsilon=0$. To state it, let us first define $\tilde{A}_{N}^{m,k}$ as
\begin{align}\label{def of A tilde}
\tilde{A}_{N}^{m,k}  = \begin{cases}
A_{N}^{m+1,k+1}\setminus v^{m,k}, & \mbox{if } m\in \{0,1,\ldots,N-1\}, \\
A_{N}, & \mbox{if } m = N,
\end{cases}
\end{align}
where $v^{m,k}$ is the vertical segment from $(2m+k-N-1,k-1)$ to $(2m+k-N-1,k)$, see Figure \ref{fig:ADreduced tilde}. Then
\begin{multline*}
\mathbb P\left( \begin{subarray}{l} \mbox{only north dominoes outside }\ds \tilde{A}_N^{m,k} \\[0.1cm] \mbox{and no domino contains $v^{m,k}$ in its interior} \end{subarray} \right)=\det\left(1-1_{\Z\cap [k,+\infty)}K_{N,m}\right)_{\ell^2(\mathbb Z)} \\
= \frac{\sum_{T\in\mathcal T(\tilde{A}_N^{m,k})} a^{v(T)}}{(1+a^2)^{\frac{N(N+1)}{2}}},
\end{multline*}
see also Figure \ref{fig:gap at even times}. Here $\mathcal T(\tilde{A}_N^{m,k})$ is the set of all domino tilings of the region $\tilde{A}_N^{m,k}$; in particular, no domino in such a tiling contains $v^{m,k}$ in its interior.
\begin{figure}
\begin{center}
\hspace{-1cm}\begin{tikzpicture}[master]
\node at (0,0) {\includegraphics[width=6.5cm]{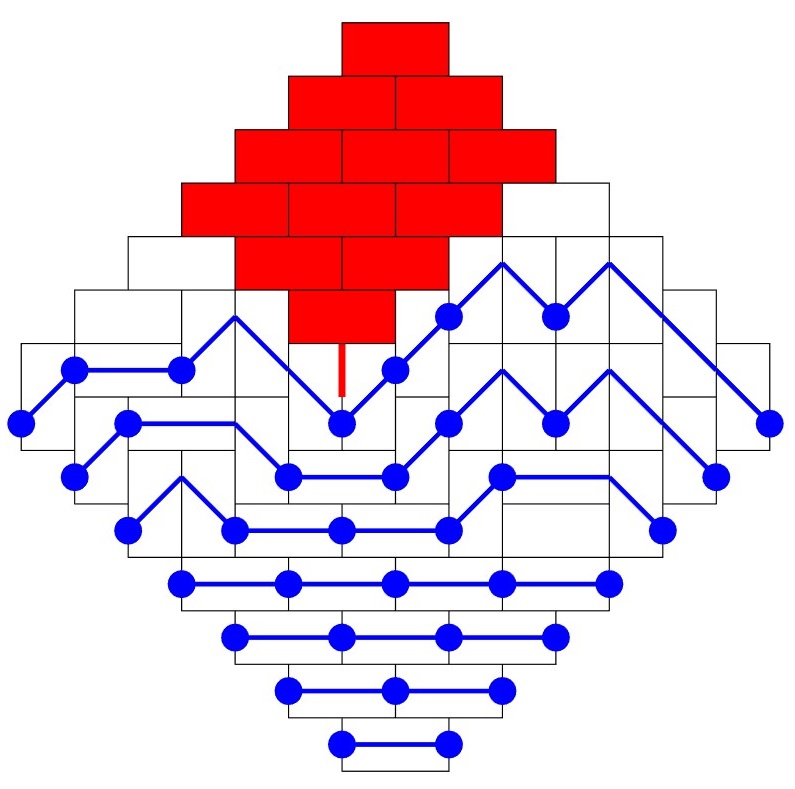}};
\end{tikzpicture} \hspace{0.5cm}
\begin{tikzpicture}[slave]
\node at (0,0) {\includegraphics[width=6.5cm]{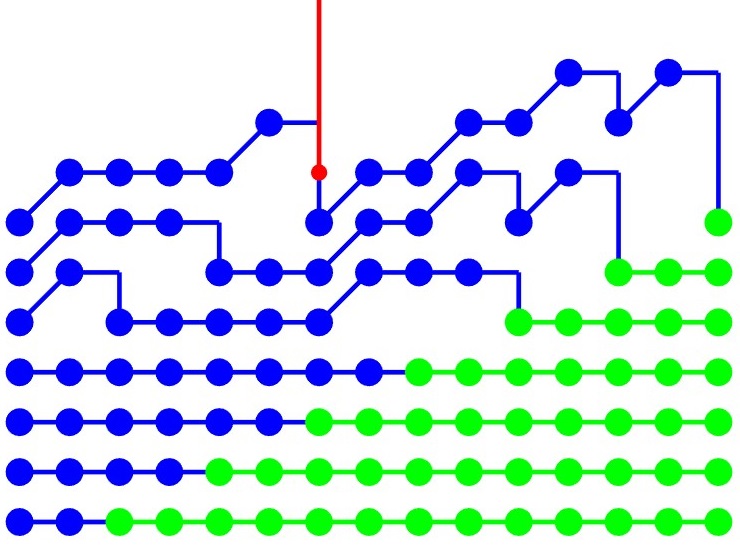}};
\end{tikzpicture}
\end{center}
\caption{\label{fig:gap at even times}Any tiling of $\tilde{A}_{N}^{m,k}$ (left) corresponds to a system of $N$ non-intersecting paths on $\{0,\ldots,2N\}\times \Z$ such that the highest particle at time $2m$ is less than $k$ (right). This is illustrated in the picture with $N=7$, $m=3$, $k=1$. The segment on the right is $\{2m\}\times [k,+\infty)$; it does not contain any blue points.}
\end{figure}
All the above results are summarized in the following generalization of Proposition \ref{prop:dominoFredholm}.
\begin{proposition}\label{prop:dominoFredholm 2}
For $N\in \N_{>0}$, $r\in\{0,\ldots, 2N\}$, $k\in\{1,\ldots, m+1\}$, we have
\begin{align*}
F_N^{m,k,\epsilon}(a)
=F_N(a)\ \det\left(1-1_{\Z\cap [k,+\infty)}K_{N,m}\right)_{\ell^2(\mathbb Z)},
\end{align*}
where $K_{N,m}$ is given by \eqref{eq:def K2} with $m=\lceil \frac{r}{2} \rceil$ and $\epsilon = 2m-r$, and
\begin{align*}
F_N^{m,k,\epsilon}(a) = \begin{cases}
\sum_{T\in\mathcal T(A_N^{m,k})} a^{v(T)} = F_N^{m,k}(a), & \mbox{if } \epsilon = 1, \\
\sum_{T\in\mathcal T(\tilde{A}_N^{m,k})} a^{v(T)}, & \mbox{if } \epsilon = 0.
\end{cases}
\end{align*}
\end{proposition}

\begin{remark}\label{remark:notilings}
By looking at the non-intersecting paths,
we can see that the case $k=1$ is special, both for $\epsilon=1$ and for $\epsilon=0$. In this case, the $N$ particles at level $2m-\epsilon$ have only $N$ admitted sites $(2m-\epsilon,0), (2m-\epsilon,-1),\ldots, (2m-\epsilon,-N+1)$, such that they are fixed. For $\epsilon=1$, this implies that all paths are horizontal in a triangular region below (and on the left of) $(2m-1,k)$. On the level of the domino tilings, this enforces south dominos in the mirror image of $A_N\setminus A_N^{m,k}$ with respect to the horizontal axis $y=0$, see Figure \ref{fig:mirror frozen}.
This observation leads indeed to the identity \eqref{eq:formulamirror} stated in the introduction.
For $\epsilon=0$, a similar argument applies: the mirror image of $A_N\setminus \tilde A_N^{m,k}$ with respect to the horizontal axis $y=0$ has to be tiled with south dominoes only, which implies that the remaining freedom lies in tiling two disjoint (touching) Aztec diamonds of order $m$ and of order $N-m$, such that we then have
\[F_N^{m,1,0}(a)=F_{m}(a)F_{N-m}(a)=\left(1+a^2\right)^{\frac{N(N+1)}{2}-m(N-m)}.\]
More generally, for a given $k\in \{1,\ldots,m+1\}$, the $N$ particles at level $2m-\epsilon$ have $N+k-1$ admitted sites. In particular, for $k=2$, there are $N+1$ admitted sites for the $N$ points, resulting in a region that is ``almost frozen", see Figure \ref{fig:one or two paths} (left). Figure \ref{fig:one or two paths} (right) corresponds to a case with $k=3$.
A similar argument implies also that the regions $A_N^{m,k}$ and $\tilde A_N^{m,k}$ are not tileable if $k\leq 0$: in these cases there are less than $N$ admissible sites at level $2m-\epsilon$ for the $N$ particles.

\end{remark}

\subsection{Lozenge tilings of hexagons}\label{subsection: lozenge tilings}
\begin{figure}
\begin{center}
\hspace{-1cm}\begin{tikzpicture}[master]
\node at (0,0) {\includegraphics[width=6.5cm]{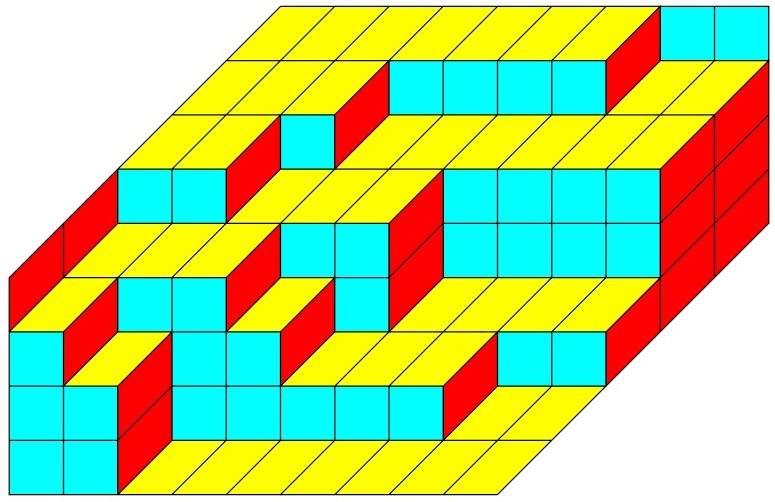}};
\end{tikzpicture} \hspace{0.5cm} \begin{tikzpicture}[slave]
\node at (0,0) {\includegraphics[width=6.5cm]{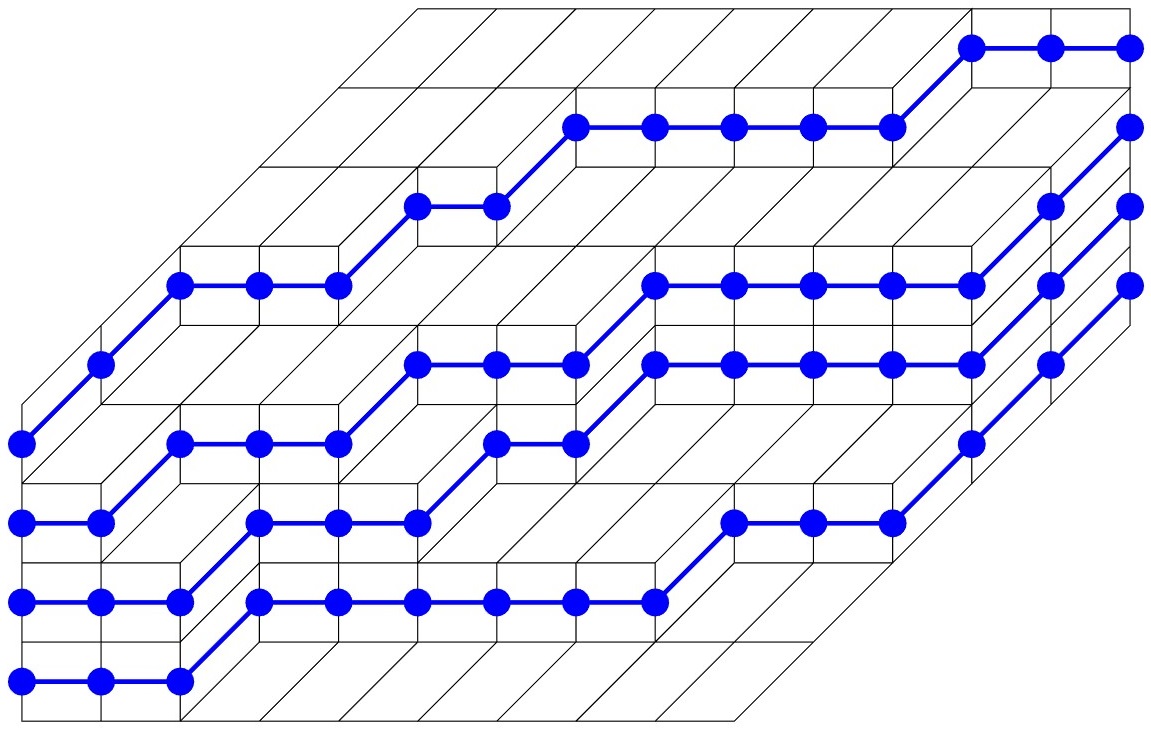}};
\end{tikzpicture} 
\end{center}
\caption{\label{fig:hexagon and non-intersecting paths}A tiling of $H_{L,M,N}$ with $L=14$, $M=5$ and $N=4$ (left), and the corresponding system of non-intersecting paths (right).}
\end{figure}

Let $L,M,N\in\mathbb N_{>0}$ with $M<L$. We associate $N$ paths to a lozenge tiling of $H_{L,M,N}$ by drawing a horizontal line along the middle of each horizontal lozenge $\tikz[scale=.25,baseline={([yshift=-0.5ex]current bounding box.center)}]{ \draw (0,0)--(0,1)--(1,1)--(1,0)--(0,0); \draw[thick] (0,.5)--(1,.5); \filldraw (0,0.5) circle(3pt);  \filldraw (1,0.5) circle(3pt); }$ and along the middle of each diagonal lozenge with vertical edges $\tikz[scale=.25,baseline={([yshift=-0.5ex]current bounding box.center)}]{\draw (0,0)--(0,1)--(1,2)--(1,1)--(0,0); \draw[thick] (0,.5)--(1,1.5); \filldraw (0,0.5) circle(3pt);  \filldraw (1,1.5) circle(3pt);}$. See Figure \ref{fig:hexagon and non-intersecting paths}. We also change our coordinate system ($y\to y-\frac{1}{2}$) such that the starting points of the paths are $(0,0),\ldots, (0,N-1)$ and the endpoints are $(L,M),\ldots, (L,M+N-1)$. 
Then we draw a point wherever these paths intersect the vertical lines $x=r$ for $r=0,1,\ldots, L$.
This gives a point configuration of $(L+1)N$ points ($N$ points at each time level $r$). 
Under the uniform measure on the set of lozenge tilings of the hexagon, the points at time level $r$ form a determinantal point process. 
By \cite[Theorem 4.7 and Proposition 4.9]{DK2021}, the kernel at time $r\in \{0,\ldots,L\}$ is given by
\begin{align}\label{def:Khexagon}
K_{L,M,N,r}(n,n') = \frac{1}{(2\pi i)^{2}} \int_{\gamma}\int_{\gamma} R(v,u) (u+1)^{r}\frac{(v+1)^{L-r}}{v^{M+N}} \frac{v^{n'}}{u^{n+1}} dudv,
\end{align}
where $\gamma$ is a closed contour oriented positively and surrounding $0$, and $R(v,u)$ is a polynomial in both variables $u,v$, and is of the form
\begin{align}\label{def:Rhexagon}
R(v,u) = \frac{1}{u-v} \begin{pmatrix}
0 & 1
\end{pmatrix} Y^{-1}(v) Y(u) \begin{pmatrix}
1 \\ 0
\end{pmatrix} = \frac{Y_{11}(v)Y_{21}(u)-Y_{11}(u)Y_{21}(v)}{u-v}.
\end{align}
By analyticity of the integrand, we can deform the first contour of integration $\gamma$ to any contour $\Sigma_1$ surrounding $0$ and the second to any contour $\Sigma_2$ enclosing $\Sigma_1$. Then this kernel is of the form \eqref{eq:def K intro} (or equivalently \eqref{eq:def K} below) with
\begin{align}\label{G hexagon}
G(u,v) = (u+1)^{r}\frac{(v+1)^{L-r}}{v^{M+N}} R(v,u)(u-v) = \sum_{j=1}^{2} g_j(u)h_j(v),
\end{align}
where
\begin{align}
&\label{def:ghex} g_{1}(u) = (1+u)^{r}Y_{21}(u), & & g_{2}(u) = (1+u)^{r}Y_{11}(u) \\
&\label{def:hhex} h_{1}(v) = \frac{(1+v)^{L-r}}{v^{M+N}}Y_{11}(v), & & h_{2}(v) = -\frac{(1+v)^{L-r}}{v^{M+N}}Y_{21}(v),
\end{align}
and $Y$ is the unique solution to the following RH problem. 

\subsubsection*{RH problem for $Y$}
\begin{itemize}
\item[(a)] $Y:\C\setminus \gamma \to \C^{2\times 2}$ is analytic.
\item[(b)] $Y$ satisfies
\begin{align*}
Y_{+}(z) = Y_{-}(z) \begin{pmatrix}
1 & \frac{(z+1)^{L}}{z^{M+N}} \\
0 & 1
\end{pmatrix}, \qquad z \in \gamma.
\end{align*}
\item[(c)] $Y(z) = (I+O(z^{-1}))z^{N\sigma_{3}}$ as $z\to \infty$, where $\sigma_{3} = \begin{pmatrix}
1 & 0 \\ 0 & -1
\end{pmatrix}$.
\end{itemize}
As pointed out in \cite[Section 4.8 {of arXiv:1712.05636}]{DK2021},  $Y$ can be written explicitly in terms of the Jacobi polynomials $P_{k}^{(\alpha,\beta)}(x)$ given by
\begin{align*}
P_{k}^{(\alpha,\beta)}(x) = \frac{\frac{d^{k}}{dx^{k}}\big( (x-1)^{k+\alpha}(x+1)^{k+\beta} \big)}{k!2^{k}(x-1)^{\alpha}(x+1)^{\beta}} = \sum_{s=0}^{k} \binom{k+\alpha}{k-s}\binom{k+\beta}{s}\bigg( \frac{x-1}{2} \bigg)^{s} \bigg( \frac{x+1}{2} \bigg)^{k-s}.
\end{align*}
More precisely, we have
\begin{align*}
& Y_{11}(z) = \eta_{1}P_{N}^{(-M-N,L)}(2z+1), & & \eta_{1} = \bigg(\sum_{s=0}^{N}\binom{-M}{N-s}\binom{N+L}{s} \bigg)^{-1}, \\
& Y_{21}(z) = \eta_{2}P_{N-1}^{(-M-N,L)}(2z+1), & & \eta_{2} = \bigg( \frac{-1}{2\pi i}\int_{\gamma}P_{N-1}^{(-M-N,L)}(2z+1) \; \frac{(z+1)^{L}}{z^{M+N}}z^{N-1}dz \bigg)^{-1}.
\end{align*}

Now, we observe that a tiling of the reduced hexagon $H_{L,M,N}^{r,k}$ corresponds to a configuration of points in which there are no points in $\{r\}\times\{k,k+1,k+2,\ldots\}$, see Figure \ref{fig:hexagons grey regions and gap}.
Thus, we have the following result, analogous to Proposition \ref{prop:dominoFredholm 2} (recall that $G_{L,M,N}^{r,k}$ is the number of lozenge tilings of $H_{L,M,N}^{r,k}$).

\begin{figure}
\begin{center}
\begin{tikzpicture}[master]
\node at (0,0) {\includegraphics[width=6cm]{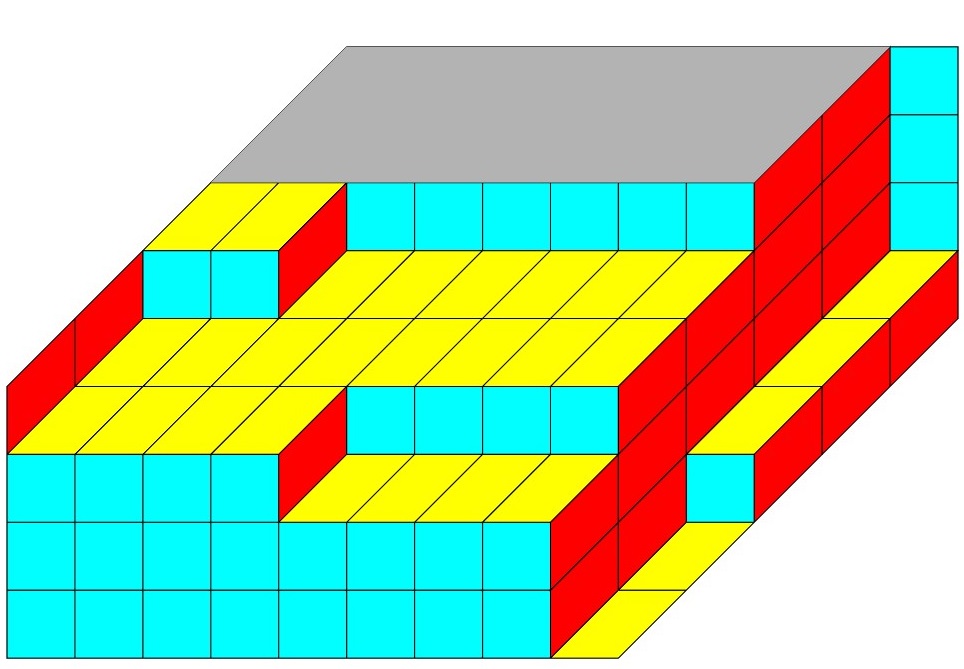}};
\end{tikzpicture} \hspace{0.4cm}
\begin{tikzpicture}[slave]
\node at (0,0) {\includegraphics[width=6cm]{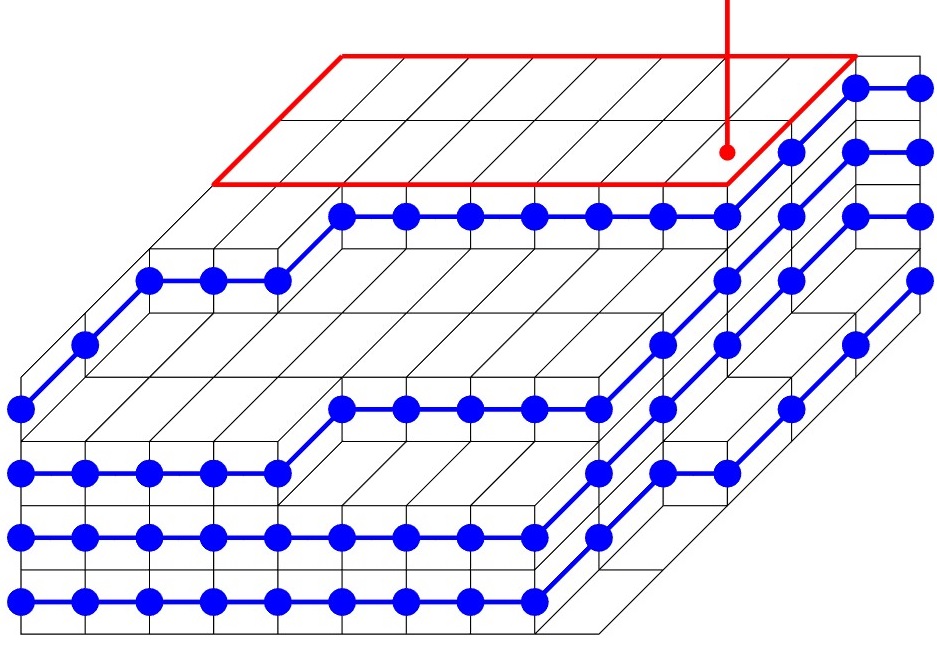}};
\end{tikzpicture}
\end{center}
\caption{\label{fig:hexagons grey regions and gap}
A tiling of $H_{L,M,N}^{r,k}$ with $L=14$, $M=5$, $N=6$, $r=11$ and $k=7$ (left) corresponds to a system of non-intersecting paths with no points in $\{r\}\times \{k,k+1,\ldots\}$ (right).}
\end{figure}

\begin{proposition}\label{prop:hexagonFredholm}
Let $L,M,N\in\mathbb N_{>0}$ with $M<L$, and let $r\in\{1,\dots, L-1\}$, $\max\{N,N-L+M+r\}\leq k\leq \min\{M+N-1,r+N-1\}$.
We have the identity
\begin{equation}\label{id:Fredholmhexagon}
G_{L,M,N}^{r,k}=G_{L,M,N}\det\left(1-1_{\Z\cap [k,\infty)}K_{L,M,N,r}\right)_{\ell^2(\mathbb Z)},
\end{equation}
where $K_{L,M,N,r}$ is given by \eqref{def:Khexagon}--\eqref{def:Rhexagon}, with $Y$ the unique solution to the above RH problem.
\end{proposition}

\section{Towards an integrable kernel}\label{section:Fourier}

In this section, we consider a general kernel of the form
\begin{align}\label{eq:def K}
K\left(n,n'\right)=
\frac{1}{(2\pi i)^2}\int_{\Sigma_2}d u \int_{\Sigma_1}d v \frac{u^{-n-1}v^{n'}G(u,v)}{u-v},\qquad n,n'\in\mathbb Z,
\end{align}
with $\Sigma_1$ the unit circle, and $\Sigma_2$ a simple closed curve outside the unit circle, both oriented positively. 
The function $G(u,v)$ takes the form 
\begin{align}\label{def:G}
G(u,v)=\sum_{j=1}^{d} g_j(u)h_j(v),
\end{align}
for some $d\geq 1$ and functions $\{g_j, h_j\}_{j=1}^{d}$.

We assume $g_j\in L^2(\Sigma_2)$ and $h_j\in L^2(\Sigma_1)$ for each $j\in \{1,\ldots,d\}$, such that
\[\int_{\Sigma_2}\int_{\Sigma_1}|G(u,v)|^2dvdu<\infty,\]
by the Cauchy-Schwarz inequality. 
Our main cases of interest correspond to $d=1$ and $h_1=1/g_1=W$ with $W$ given by \eqref{def:W2}, and to $d=2$ and $G$ given by \eqref{G hexagon}. We choose however to work in a more general context here, because other kernels of the general form \eqref{eq:def K} appear in other models, such that our result will be applicable there too. For instance, one can consider random domino tilings of the Aztec diamond or random lozenge tilings of the hexagon with respect to a non-uniform measure with periodic weightings in horizontal and vertical directions. These models give rise to determinantal point processes with kernels which are also of the form \eqref{eq:def K}--\eqref{def:G}. Indeed, by \cite[Theorem 4.7 and Proposition 4.9]{DK2021}, if the weight structure is periodic of period $p$ in the vertical direction, the corresponding kernel at time $r$ can be encoded in a $p\times p$ matrix as follows,
\begin{align}\label{double contour integrale duits kuijlaars}
[\mathrm{K}(pn+i,pn'+j)]_{i,j=0}^{p-1} = \frac{1}{(2\pi i)^{2}}\int_{\gamma} \int_{\gamma} A_{r,L}(w) \frac{\begin{pmatrix}
0_{p} & I_{p}
\end{pmatrix}Y^{-1}(w)Y(z) \begin{pmatrix}
I_{p} \\ 0_{p}
\end{pmatrix}}{z-w} \frac{A_{0,r}(z)w^{n'}}{z^{n+1}w^{M+N}}dzdw.
\end{align}
Here $A_{r,L}$ and $A_{0,r}$ are $p\times p$ matrix-valued functions, $0_p$ and $I_p$ denote the zero matrix and the identity matrix of size $p\times p$,
and $Y$ is the solution to a $2p\times 2p$ matrix RH problem, similar to the one stated below \eqref{def:ghex}--\eqref{def:hhex} corresponding to uniform lozenge tilings. Each entry in the matrix $[\mathrm{K}(pn+i,pn'+j)]_{i,j=0}^{p}$ thus has a similar structure as the kernel \eqref{def:Khexagon}--\eqref{def:Rhexagon}, and can be written in the form \eqref{eq:def K}--\eqref{def:G}, now with $d=2p$ instead of $d=2$.
Instead of representing the kernel as a $p\times p$ matrix, we can also write it at once in the form
\[\sum_{i,j=0}^{p-1}K(m,m')1_{\{m\equiv i\,{\rm mod }\, p,m'\equiv j\, {\rm mod }\, p\}},\]
where $K(m,m')$ is the kernel in the corresponding entry of the matrix. This kernel is of the form \eqref{eq:def K}--\eqref{def:G}, but now with (much larger) $d=2p^3$.

We write $\mathcal F:L^2\left(\Sigma_1\right)\to \ell^2(\mathbb Z)$ for the Fourier series transform defined as
\begin{align}\label{def:Fourier}
\mathcal F[f](n)= \int_{-1/2}^{1/2}f(e^{2\pi i  t})e^{-2\pi i nt}dt=\int_{\Sigma_{1}}f(z)z^{-n}\frac{dz}{2\pi i z},\qquad n\in\mathbb Z,
\end{align}
and
$\mathcal F^{-1}:\ell^2(\mathbb Z)\to L^2\left(\Sigma_1\right)$ for the inverse given by
\be\label{def:Fourierinv}
\mathcal F^{-1}[\sigma](z)= \sum_{n\in\mathbb Z}\sigma(n)z^n,\qquad z\in \Sigma_1.
\ee

Consider a subset $\mathcal I$ of $\mathbb Z$ which is bounded below and which is the disjoint union of finitely many clusters of consecutive integers.
We can parametrize such a set $\mathcal I$ as
\begin{equation}\label{def:I}\mathcal I=\sqcup_{j=0}^q \mathcal I_j,\qquad \mathcal I_j=[k_{2j+1},k_{2j}]\cap\mathbb Z.
\end{equation} 
Here we require that $k_1,\ldots, k_{2q+1}\in\mathbb Z$ and that $k_0\in\mathbb Z\cup\{+\infty\}$, with the inequalities
\begin{align}\label{inequalities between the kj's}
k_{2q+1}\leq k_{2q} <  \ldots \leq  k_{2}< k_{1}\leq k_{0}\leq +\infty
\end{align}
in order to ensure that the clusters $\mathcal I_0,\ldots, \mathcal I_q$ are non-empty and disjoint. 
Note that this parametrization of $\mathcal I$ is not unique, as we could split one cluster of consecutive points into several ones. It will however be advantageous to keep $q$ as small as possible, i.e. to choose $q$ so that $q+1$ is the minimal number of clusters of consecutive points in $\mathcal I$.
The simplest case occurs when $q=0$; then $\mathcal I=[k_1,k_0]\cap\mathbb Z$ is a single cluster of consecutive points. If moreover $k_0=+\infty$, we have a semi-infinite cluster $\mathcal I=\{k_1,k_1+1,k_1+2,\ldots\}$, and this will be our first case of interest later on.

Let $1_{\mathcal{I}_j}:\ell^2(\mathbb Z)\to\ell^2(\mathbb Z)$ for $j=0,1,\ldots, q$ be the projection operator on $\mathcal I_j$ defined by
\[1_{\mathcal{I}_j}[\sigma](n)=\begin{cases}\sigma(n)&\mbox{if $n\in \mathcal I_j$,}\\
0&\mbox{if $n\in\mathbb Z\setminus\mathcal I_j$,}\end{cases}\]
and let $\mathcal K:\ell^2(\mathbb Z)\to\ell^2(\mathbb Z)$ be
the integral operator with kernel $K$ given by
\begin{align*}
\mathcal{K}[\sigma](n) := \sum_{n'\in \mathbb Z}K(n,n')\sigma(n'), \qquad n \in \Z.
\end{align*}

\begin{proposition}\label{prop:Fourier}
For any $j\in \{0,1,\ldots,q\}$, the integral operator 
\begin{align}\label{def:L}
\mathcal L_j:=\mathcal F^{-1} 1_{\mathcal{I}_j}\mathcal K\mathcal F:L^2(\Sigma_1)\to L^2(\Sigma_1)
\end{align}
has kernel
\begin{align}\label{def:kernelL}
L_j(z,z')=\frac{1}{(2\pi i)^2}\int_{\Sigma_2}d u  \frac{( \frac{z}{u} )^{1+k_{2j}} - ( \frac{z}{u} )^{k_{2j+1}}}{z-u} \frac{G(u,z')}{u-z'},
\end{align}
where for $j=0$, we understand the term $( \frac{z}{u} )^{1+k_{0}}$ as $0$ if $k_0=+\infty$.
\end{proposition}
\begin{proof}
We can write $K$ as 
\[K(n,n')=\frac{1}{2\pi i}\int_{\Sigma_2}d u \ u^{-n-1}\mathcal F\left[\frac{G(u,\cdot)}{u-\cdot}\right](-n'-1).\]
Hence, the kernel $R(n,z')$ of $\mathcal K\mathcal F:L^2(\Sigma_1)\to\ell^2(\mathbb Z)$ is given by
\begin{align*}
R(n,z')&=\frac{1}{(2\pi i)^2}\int_{\Sigma_2}d u \ u^{-n-1}\sum_{n'\in\mathbb Z}(z')^{-n'-1}\mathcal F\left[\frac{G(u,\cdot)}{u-\cdot}\right](-n'-1)\\
&=\frac{1}{(2\pi i)^2}\int_{\Sigma_2}d u \ u^{-n-1}
\sum_{m\in\mathbb Z}(z')^{m}\mathcal F\left[\frac{G(u,\cdot)}{u-\cdot}\right](m)
\\&=\frac{1}{(2\pi i)^2}\int_{\Sigma_2}d u \ u^{-n-1}\left(\mathcal F^{-1}\circ\mathcal F\right)\left[\frac{G(u,\cdot)}{u-\cdot}\right](z')
\\&=\frac{1}{(2\pi i)^2}\int_{\Sigma_2}d u \ u^{-n-1}\frac{G(u,z')}{u-z'},
\qquad n\in\mathbb Z, z'\in \Sigma_1.
\end{align*}
Composing this from the left with $\mathcal F^{-1} 1_{\mathcal{I}_j}$, we obtain
that
\begin{align*}
L_j(z,z') & = \frac{1}{(2\pi i)^2}\int_{\Sigma_2}d u \ \frac{G(u,z')}{u-z'} \sum_{n\in \mathcal{I}_j} u^{-n-1}z^n \\
& = \frac{1}{(2\pi i)^2}\int_{\Sigma_2}d u \  \frac{G(u,z')}{(u-z')(z-u)} \bigg( \Big( \frac{z}{u} \Big)^{1+k_{2j}} - \Big( \frac{z}{u} \Big)^{k_{2j+1}} \bigg),
\end{align*}
which is the same as \eqref{def:kernelL}.
\end{proof}

We will now use the above factorization \eqref{def:L} to obtain an identity for Fredholm determinants. Recall that the Fredholm determinant of a trace-class integral operator can be written as a Fredholm series:
in the case of a trace-class operator $\mathcal K$ acting on $\ell^2(\mathbb Z)$ with kernel $K$, we have
\[
\det\left(1+\mathcal K\right)_{\ell^2(\mathbb Z)}
=1+\sum_{k=1}^\infty\frac{1}{k!}\sum_{n_1,\ldots, n_k\in\mathbb Z} \det\left(K(n_i,n_j)\right)_{i,j=1}^k,\]
while in the case of a trace-class operator $\mathcal K$ acting on $L^2(\Sigma)$ with kernel $K$ for some contour $\Sigma$ in the complex plane, we have
\[
\det\left(1+\mathcal K\right)_{L^2(\Sigma)}
=1+\sum_{k=1}^\infty\frac{1}{k!}\int_{\Sigma^k}\det\left(K(z_i,z_j)\right)_{i,j=1}^k\prod_{j=1}^kdz_j.\]
We will use a number of standard properties of Fredholm determinants. First, a bounded linear operator 
$\mathcal K:\mathcal H_1\to\mathcal H_1$ on a separable Hilbert space $\mathcal H_1$ is a trace-class operator  if and only if $\mathcal K=\mathcal A\mathcal B$ can be written as a composition of two Hilbert-Schmidt operators $\mathcal B:\mathcal H_1\to\mathcal H_2$ and $\mathcal A:\mathcal H_2\to\mathcal H_1$, for some separable Hilbert space $\mathcal H_2$.
Secondly, for any trace-class operator $\mathcal K:\mathcal H_1\to \mathcal H_1$ and for any invertible bounded linear operator $\mathcal B:\mathcal H_2\to\mathcal H_1$ between two separable Hilbert spaces,
we
have the identity
\[\det\left(1+\mathcal B^{-1}\mathcal K\mathcal B\right)_{\mathcal H_2}=\det\left(1+\mathcal K\right)_{\mathcal H_1}.\]
Finally, for two trace-class operators $\mathcal T_1,\mathcal T_2:\mathcal H\to\mathcal H$, we have \cite[Theorem 3.5 (a)]{Simon}
\[\det\left((1+\mathcal T_1)(1+\mathcal T_2)\right)_{\mathcal H}=\det\left(1+\mathcal T_1\right)_{\mathcal H}\det\left(1+\mathcal T_2\right)_{\mathcal H}.\]
See for instance \cite{Simon} for more details about Fredholm determinants and trace-class operators.

\begin{proposition}\label{prop:integrable}
For any $\gamma_0,\ldots, \gamma_q \in \C$, we have
\begin{align}\label{det 1 - sqrt gamma M}
\det\bigg(1-\sum_{j=0}^q\gamma_j 1_{\mathcal I_j}\mathcal K\bigg)_{\ell^2(\mathbb Z)}=\det\left(1-\mathcal M\right)_{L^2(\Sigma_{1}\cup\Sigma_{2})},
\end{align}
where $\mathcal M:L^2(\Sigma_1\cup\Sigma_2)\to L^2(\Sigma_1\cup\Sigma_2)$ has kernel \be\label{def:M}
M(z,z')=\frac{1_{\Sigma_2}(z')1_{\Sigma_1}(z)\sum_{j=0}^{q} \gamma_j\big( ( \frac{z}{z'} )^{k_{2j+1}} - ( \frac{z}{z'} )^{1+k_{2j}} \big) - 1_{\Sigma_1}(z')1_{\Sigma_2}(z)G(z,z')}{2\pi i (z-z')}.
\ee
If $k_0=+\infty$, we understand the term $( \frac{z}{z'} )^{1+k_{0}}$ as $0$. 
\end{proposition}
\begin{proof}
We can decompose $\mathcal L:=\sum_{j=0}^q\gamma _j\mathcal L_j=\mathcal Q\mathcal R$, with $\mathcal Q:L^2(\Sigma_2)\to L^2(\Sigma_1)$, $\mathcal R:L^2(\Sigma_1)\to L^2(\Sigma_2)$, where $\mathcal Q,\mathcal R$ have kernels
\begin{align*}
Q(z,u)=\frac{1}{2\pi i}\sum_{j=0}^{q} \gamma_j\frac{ ( \frac{z}{u} )^{k_{2j+1}} - ( \frac{z}{u} )^{1+k_{2j}} }{z-u},\qquad R(u,z')=\frac{G(u,z')}{2\pi i(z'-u)}.
\end{align*}
Recall that $\Sigma_{1}\cap  \Sigma_{2} = \emptyset$, so that both $\frac{1}{z-u}$ and $\frac{1}{z'-u}$ remain bounded for $u\in \Sigma_{2}$ and $z,z' \in \Sigma_{1}$. The operator $\mathcal Q$ is Hilbert-Schmidt as a direct consequence of the continuity of $Q(z,u)$ on the compact $\Sigma_{1}\times \Sigma_{2}$, and since we assume $\int_{\Sigma_1}\int_{\Sigma_2}|G(u,z)|^2dudz<\infty$, the operator $\mathcal R$ is also Hilbert-Schmidt. Consequently, $\mathcal L$ is trace-class.
Now it follows from Proposition \ref{prop:Fourier} and the invariance of the Fredholm determinant under conjugation with a bounded linear operator that 
\begin{align*}
\det\bigg(1-\sum_{j=0}^q\gamma_j 1_{\mathcal I_j}\mathcal K\bigg)_{\ell^2(\mathbb Z)} = \det\bigg(1-\mathcal F^{-1} \sum_{j=0}^q\gamma_j 1_{\mathcal I_j}\mathcal K \mathcal F\bigg)_{L^2(\Sigma_{1})}= \det\bigg(1-\sum_{j=0}^q\gamma _j\mathcal L_j\bigg)_{L^2(\Sigma_{1})}.
\end{align*}

Next, we are going to show that
$\mathcal Q, \mathcal R$ are not only Hilbert-Schmidt but also trace-class. This is indeed true, since
\begin{align*}
Q(z,u)=\frac{1}{(2\pi i)^2}\int_{\tilde C}\frac{{dt}}{(t-z)(t-u)} \sum_{j=0}^q\gamma_j\bigg( \Big( \frac{z}{u} \Big)^{k_{2j+1}} - \Big( \frac{z}{u} \Big)^{1+k_{2j}}\bigg),
\end{align*}
with $\tilde C$ a closed curve, oriented positively, surrounding $\Sigma_{1}$ and lying inside $\Sigma_{2}$; so \[\mathcal Q=\sum_{j=0}^{q}\gamma_j(\mathcal A_{j}\mathcal B_{j}-\tilde{\mathcal A}_{j}\tilde{\mathcal B}_{j}),\] where 
\begin{align*}
\Big\{\mathcal A_{j}:L^{2}(\tilde{C})\to L^{2}(\Sigma_{1}),\mathcal B_{j}:L^{2}(\Sigma_{2})\to L^{2}(\tilde{C}),\tilde{\mathcal A}_{j}:L^{2}(\tilde{C})\to L^{2}(\Sigma_{1}),\tilde{\mathcal B}_{j}:L^{2}(\Sigma_{2})\to L^{2}(\tilde{C})\Big\}_{j=0}^{q}
\end{align*}
are the Hilbert-Schmidt operators defined by their kernels
\begin{align*}
& A_{j}(z,t)=\frac{z^{k_{2j+1}}}{2\pi i (t-z)}, \quad \tilde{A}_{j}(z,t)=\frac{-z^{1+k_{2j}}}{2\pi i (t-z)}, \quad B_{j}(t,u)=\frac{u^{-k_{2j+1}}}{2\pi i (t-u)}, \quad \tilde{B}_{j}(t,u)=\frac{-u^{-1-k_{2j}}}{2\pi i (t-u)}.
\end{align*}
This shows that $\mathcal Q$ is trace-class as a finite sum of compositions of two Hilbert-Schmidt operators. A similar argument applies to $\mathcal R$.

Let us now extend $\mathcal L$ trivially to the trace-class operator $\begin{pmatrix}\mathcal L&0\\
0&0\end{pmatrix}$ on the larger space $L^2(\Sigma_1)\oplus L^2(\Sigma_2)$.
In this matrix notation the $(i,j)$-entry denotes an operator from $L^2(\Sigma_j)$ to $L^2(\Sigma_i)$; the matrix of operators acts on vectors of the form $\begin{pmatrix}f_1\\f_2\end{pmatrix}$, where $f_j\in L^2(\Sigma_j)$, using the usual rules of matrix multiplication.
We then have
\begin{align*}
\det\left(1-\mathcal L\right)_{L^2(\Sigma_1)}&=\det\begin{pmatrix}1-\mathcal Q\mathcal R&0\\0&1\end{pmatrix}_{L^2(\Sigma_1)\oplus L^2(\Sigma_2)}\\
&=\det\begin{pmatrix}1-\mathcal Q\mathcal R&0\\0&1\end{pmatrix}_{L^2(\Sigma_1)\oplus L^2(\Sigma_2)}\det\begin{pmatrix}1&0\\-\mathcal R&1\end{pmatrix}_{L^2(\Sigma_1)\oplus L^2(\Sigma_2)}\\
&=
\det\begin{pmatrix}1-\mathcal Q\mathcal R&0\\-\mathcal R&1\end{pmatrix}_{L^2(\Sigma_1)\oplus L^2(\Sigma_2)},
\end{align*}
where we used the multiplicativity of the Fredholm determinant of trace-class perturbations of the identity.
But the latter is equal to
\begin{align*}
\det\begin{pmatrix}1-\mathcal Q\mathcal R&0\\-\mathcal R&1\end{pmatrix}_{L^2(\Sigma_1)\oplus L^2(\Sigma_2)}
&=\det\begin{pmatrix}1&\mathcal Q\\0&1\end{pmatrix}_{L^2(\Sigma_1)\oplus L^2(\Sigma_2)}\det\begin{pmatrix}1&- \mathcal Q\\ -\mathcal R&1\end{pmatrix}_{L^2(\Sigma_1)\oplus L^2(\Sigma_2)}
\\&=
\det\begin{pmatrix}1&-\mathcal Q\\ -\mathcal R&1\end{pmatrix}_{L^2(\Sigma_1)\oplus L^2(\Sigma_2)}.
\end{align*}
We can also view this as a Fredholm determinant on $L^2(\Sigma_1\cup\Sigma_2)$ of the operator $1-\mathcal M$ with kernel
\[M(z,z')=Q(z,z')1_{\Sigma_1}(z)1_{\Sigma_2}(z') + R(z,z')1_{\Sigma_1}(z')1_{\Sigma_2}(z).\]
This is precisely the stated result.
\end{proof}

Observe that Proposition \ref{prop:dominointegrable}
is just a special case of Proposition
\ref{prop:integrable}: it corresponds to $G(u,v)=W(v)/W(u)$ with $W$ as in \eqref{def:W}, $q=0, k_0=+\infty, k_1=k$ and $\gamma_0=1$.

\section{RH problem and ratio identities}\label{section:RH}
In this section, we still consider a general kernel of the form \eqref{eq:def K}--\eqref{def:G}. Recall in particular that $\Sigma_1$ is the unit circle, that $\Sigma_2$ is a simple closed curve outside the unit circle, and that both are positively oriented. As before, we take a subset 
$\mathcal I$ of $\mathbb Z$ which is bounded below and consists of a disjoint union of $q+1$ clusters of consecutive points, and we write it again as 
\begin{align}\label{def of I}
\mathcal{I}:=\sqcup_{j=0}^q\mathcal I_j,\qquad \mathcal I_j=[k_{2j+1},k_{2j}] \cap\mathbb Z,
\end{align}
with $k_0 \in \Z \cup \{+\infty\}$ and $k_1,\ldots, k_{2q+1}\in\mathbb Z$ such that
\begin{align*}
k_{2q+1}\leq k_{2q} <  \ldots \leq  k_{2}< k_{1}\leq k_{0}\leq +\infty.
\end{align*}
By \eqref{def:G}, if we set $k_{0}=+\infty$, the kernel $M$ in \eqref{def:M} is of $(d+2q+1)$-integrable form 
\begin{align}\label{eq:Mint}
M(z,z')=\frac{\vec{\mathrm{g}}^T(z)\vec{\mathrm{h}}(z')}{z-z'},\quad \vec{\mathrm{g}}(z)=\frac{1}{2\pi i}\begin{pmatrix}\gamma_{0} 1_{\Sigma_1}(z)z^{k_{1}} \\
\begin{pmatrix}
-\gamma_{\ell}1_{\Sigma_1}(z)z^{1+k_{2\ell}} \\
\gamma_{\ell}1_{\Sigma_1}(z)z^{k_{2\ell+1}}
\end{pmatrix}_{\ell=1}^{q} \\
\left(1_{\Sigma_2}(z)g_j(z)\right)_{j=1}^{d}
\end{pmatrix},\quad \vec{\mathrm{h}}(z)=\begin{pmatrix}
1_{\Sigma_2}(z)z^{-k_{1}} \\
\begin{pmatrix}
1_{\Sigma_2}(z)z^{-1-k_{2\ell}} \\
1_{\Sigma_2}(z)z^{-k_{2\ell+1}}
\end{pmatrix}_{\ell=1}^{q} \\
\left(-1_{\Sigma_1}(z)h_j(z)\right)_{j=1}^{d}
\end{pmatrix},
\end{align}
while if $k_{0}<+\infty$, the kernel $M$ is of $(d+2q+2)$-integrable form 
\begin{align}\label{eq:Mint mod}
M(z,z')=\frac{\vec{\mathrm{g}}^T(z)\vec{\mathrm{h}}(z')}{z-z'},\quad \vec{\mathrm{g}}(z)=\frac{1}{2\pi i}\begin{pmatrix}\begin{pmatrix}
-\gamma_{\ell}1_{\Sigma_1}(z)z^{1+k_{2\ell}} \\
\gamma_{\ell}1_{\Sigma_1}(z)z^{k_{2\ell+1}}
\end{pmatrix}_{\ell=0}^{q} \\
\left(1_{\Sigma_2}(z)g_j(z)\right)_{j=1}^{d}
\end{pmatrix},\quad \vec{\mathrm{h}}(z)=\begin{pmatrix}
\begin{pmatrix}
1_{\Sigma_2}(z)z^{-1-k_{2\ell}} \\
1_{\Sigma_2}(z)z^{-k_{2\ell+1}}
\end{pmatrix}_{\ell=0}^{q} \\
\left(-1_{\Sigma_1}(z)h_j(z)\right)_{j=1}^{d}
\end{pmatrix}.
\end{align}

By the theory of integrable operators developed by Its, Izergin, Korepin, Slavnov, Deift and Zhou \cite{IIKS, DeiftItsZhou}, we can then characterize
the resolvent $\mathcal M(1-\mathcal M)^{-1}$ in terms of a $(d+2q+1)\times (d+2q+1)$ RH problem (if $k_{0}=+\infty$) or by a $(d+2q+2)\times (d+2q+2)$ RH problem (if $k_{0}<+\infty$), whose jump matrix is equal to $I-2\pi i \vec{\mathrm{g}} \vec{\mathrm{h}}^T$. The solution of this RH problem exists and is unique if $1-\mathcal M$ is invertible (or equivalently, if $\det(1-\mathcal{M})_{L^{2}(\Sigma_{1}\cup \Sigma_{2})} \neq 0$).
From this RH characterization of the resolvent, one can also extract information about the Fredholm determinant
 $\det\left(1-\mathcal M\right)_{L^2(\Sigma_1\cup\Sigma_2)}$, as we will explain below. But first we take a closer look at the relevant RH problem.

\subsubsection*{RH problem for $U$}
\begin{itemize}
\item[(a)] If $k_{0}=+\infty$, then $U:\mathbb C\setminus\left(\Sigma_1\cup\Sigma_2\right)\to \mathbb C^{ (d+2q+1 )\times (d+2q+1 )}$ is analytic, while

\noindent  \hspace{-0.09cm} if $k_{0}<+\infty$, then $U:\mathbb C\setminus\left(\Sigma_1\cup\Sigma_2\right)\to \mathbb C^{ (d+2q+2)\times (d+2q+2 )}$ is analytic.
\item[(b)] $U_+(s)=U_-(z)J_U(z)$ for $z\in \Sigma_1\cup\Sigma_2$, with\begin{align*}
J_U(z)=\begin{pmatrix}
I_{2q+1}& \substack{\ds\begin{pmatrix}
\gamma_{0}1_{\Sigma_1}(z)z^{k_{1}}h_j(z)
\end{pmatrix}_{j=1,\ldots,d} \\
\ds \begin{pmatrix}
-\gamma_{\ell}1_{\Sigma_1}(z)z^{1+k_{2\ell}}h_j(z) \\
\gamma_{\ell}1_{\Sigma_1}(z)z^{k_{2\ell+1}}h_j(z)
\end{pmatrix}_{\substack{\ell=1,\ldots,q \\ j=1,\ldots,d}}
}
  \\
-  \left( \substack{
  \ds \begin{pmatrix}
  1_{\Sigma_2}(z)z^{-k_{1}}g_j(z)
  \end{pmatrix}_{j=1,\ldots,d}
 \\ \ds \begin{pmatrix}
1_{\Sigma_2}(z)z^{-1-k_{2\ell}}g_j(z) \\
1_{\Sigma_2}(z)z^{-k_{2\ell+1}}g_j(z)
\end{pmatrix}_{\substack{\ell=1,\ldots,q \\ j=1,\ldots,d}} } \right)^{T}
& I_d
\end{pmatrix}
\end{align*}
if $k_{0}=+\infty$, while if $k_{0}<+\infty$ we have
\begin{align}\label{JU for k0 finite}
J_U(z)=\begin{pmatrix}
I_{2q+2}& \begin{pmatrix}
-\gamma_{\ell}1_{\Sigma_1}(z)z^{1+k_{2\ell}}h_j(z) \\
\gamma_{\ell}1_{\Sigma_1}(z)z^{k_{2\ell+1}}h_j(z)
\end{pmatrix}_{\substack{\ell=0,\ldots,q \\ j=1,\ldots,d}}  \\
-\begin{pmatrix}
1_{\Sigma_2}(z)z^{-1-k_{2\ell}}g_j(z) \\
1_{\Sigma_2}(z)z^{-k_{2\ell+1}}g_j(z)
\end{pmatrix}_{\substack{ \ell=0,\ldots,q  \\ j=1,\ldots,d }}^{T} 
& I_d
\end{pmatrix}.
\end{align}
\item[(c)] If $k_{0}=+\infty$, then $U(z)=I_{d+2q+1}+\bigO(z^{-1})$ as $z\to \infty$, while

\noindent  \hspace{-0.09cm} if $k_{0}<+\infty$, then $U(z)=I_{d+2q+2}+\bigO(z^{-1})$ as $z\to \infty$.
\end{itemize}

In the special case where $d=1$, $h_1=1/g_1=W$, $q=0$ and $\mathcal I=[k,+\infty)$, we have
\begin{align*}
M(z,z')=\frac{\vec{\mathrm{g}}^T(z)\vec{\mathrm{h}}(z')}{z-z'},\qquad \vec{\mathrm{g}}(z)=\frac{1}{2\pi i}\begin{pmatrix}\gamma_{0}1_{\Sigma_1}(z)z^k\\
1_{\Sigma_2}(z)/W(z)
\end{pmatrix},\qquad \vec{\mathrm{h}}(z')=\begin{pmatrix}1_{\Sigma_2}(z')z'^{-k}\\
-1_{\Sigma_1}(z')W(z')
\end{pmatrix}.
\end{align*}
The associated RH problem is then as follows.

\subsubsection*{RH problem for $U$}
\begin{itemize}
\item[(a)] $U:\mathbb C\setminus\left(\Sigma_1\cup \Sigma_2\right)\to \mathbb C^{2\times 2}$ is analytic,
\item[(b)] $U_+(s)=U_-(z)J_U(z)$ for $z\in \Sigma_1\cup  \Sigma_2$, with\begin{align*}
J_U(z)=\begin{cases}
\begin{pmatrix}
1& \gamma_{0}z^kW(z)\\ 0& 1
\end{pmatrix},&z\in \Sigma_1,\\
\begin{pmatrix}
1& 0\\ -\frac{1}{z^{k}W(z)}
& 1
\end{pmatrix},&z\in \Sigma_2.\end{cases}
\end{align*}
\item[(c)]{As $z\to \infty$, $U(z)=I_{2}+\bigO(z^{-1})$}.
\end{itemize}

\begin{remark}
In the case of the kernel $K_{N,m}$ relevant in domino tilings of the Aztec diamond, $W$ is given by \eqref{def:W2}, and the jump matrices further specialize to
\begin{align}\label{eq:JY2}
J_U(z)=\begin{cases}
\begin{pmatrix}
1& \frac{\gamma_{0}z^{N-m+k+\epsilon-1}}{(z-a)^{N-m+\epsilon}(1+az)^m}\\ 0& 1
\end{pmatrix},&z\in \Sigma_1,\\
\begin{pmatrix}
1& 0\\ -\frac{(z-a)^{N-m+\epsilon}(1+az)^m}{z^{N-m+k+\epsilon-1}}
& 1
\end{pmatrix},&z\in \Sigma_2.\end{cases}
\end{align}
\end{remark}

Recall that $\mathcal I=\sqcup_{j=0}^q\mathcal I_j$ is the disjoint union of clusters $\mathcal I_j=[k_{2j+1},k_{2j}]\cap\mathbb Z$ of consecutive points defined by \eqref{def of I} in terms of $\vec{k}=(k_{2q+1},k_{2q},\ldots,k_{1},k_{0})$. Let $j\in\{0,\ldots, 2q+1\}$, where $j=0$ is excluded in case $k_0=+\infty$, and let $u = \lceil \frac{j-1}{2} \rceil \in \{0,\ldots,q\}$. 
Let us define $\mathcal I_{u}^*$ to be the same cluster as $\mathcal I_{u}$, but with $k_{j}$ replaced by $k_{j}+1$. If $j$ is odd, this means that $\mathcal I_{u}^*=\mathcal I_{u}\setminus\{k_j\}$; if $j$ is even, we have $\mathcal I_{u}^*=\mathcal I_{u}\cup\{k_j+1\}$. 
In order to avoid overlaps, we need to assume in this construction that $k_j+1<k_{j-1}$ if $j\neq 0$ is even. (We do not need to avoid empty intervals, so the case $k_{j}=k_{j-1}$ for some odd $j$ is allowed.) For ease of notation, we also denote $\mathcal I_{v}^*=\mathcal I_{v}$ for $v\in \{0,\ldots,q\}$, $v\neq u$.

In the statement of the following proposition, we will write $U(z;\vec{k})$ instead of $U(z)$ for clarity. If $k_{0}<+\infty$, then we also adopt the convention to start row and column indices at $0$ instead of $1$, i.e. $U_{00}$ is the entry of $U$ in the first row and first column.
\begin{theorem}\label{thm:multi gap ratio}
\noindent
Let $j\in \{0,\ldots,2q+1\}$, where $j=0$ is excluded in case $k_0=+\infty$. If $j\neq 0$ is even, we further assume that $k_j+1<k_{j-1}$. Let $\mathcal I_{\ell}^*$ be defined as above, $\ell \in \{0,\ldots,q\}$.
For any $\gamma_0,\ldots, \gamma_q \in \C$ such that $\det\left(1-\sum_{\ell=0}^q\gamma_\ell 1_{\mathcal{I}_\ell}\mathcal K\right)\neq 0$, we have the identity
\begin{align}\label{eq:ratioid multi gap}
\frac{\det(1-\sum_{\ell=0}^q\gamma_\ell 1_{\mathcal{I}_{\ell}^*}\mathcal K)_{\ell^2(\mathbb Z)}}{\det\left(1-\sum_{\ell=0}^q\gamma_\ell 1_{\mathcal{I}_\ell}\mathcal K\right)_{\ell^2(\mathbb Z)}}=U_{jj}(0;\vec{k}).
\end{align}
In particular, if $q=0$, $k_{0}=+\infty$ and $k_{1}=k\in \Z$,
\begin{align}\label{eq:ratioid}
\frac{\det\left(1-\gamma_{0}1_{\{k+1,k+2,\ldots\}}\mathcal K\right)_{\ell^2(\mathbb Z)}}{\det\left(1-\gamma_{0}1_{\{k,k+1,\ldots\}}\mathcal K\right)_{\ell^2(\mathbb Z)}}=U_{11}(0;(k_1=k,k_0=\infty)).
\end{align}
\end{theorem}
\begin{proof}
Let $\vec{k}=(k_{2q+1},k_{2q},\ldots,k_{1},k_{0})$, and let $\vec{k}^{(j)}$ denote the vector obtained by replacing $k_{j}$ by $k_{j}+1$ in $\vec{k}$. Let us write $M_{\vec{k}}$ for the kernel \eqref{eq:Mint}--\eqref{eq:Mint mod}, where the dependence of $M$ on $\vec{k}$ has been made explicit.
Proposition \ref{prop:integrable} implies that
\begin{align}\label{lol1}
\frac{\det(1-\sum_{\ell=0}^q\gamma_\ell 1_{\mathcal{I}_{\ell}^*}\mathcal K)_{\ell^2(\mathbb Z)}}{\det\left(1-\sum_{\ell=0}^q\gamma_\ell 1_{\mathcal{I}_{\ell}}\mathcal K\right)_{\ell^2(\mathbb Z)}} = \frac{\det(1-M_{\vec{k}^{(j)}})_{L^2(\Sigma_1\cup\Sigma_2)}}{\det(1-M_{\vec{k}})_{L^2(\Sigma_1\cup\Sigma_2)}}.
\end{align}
Now, note that
\begin{align*}
\Delta(z,z'):=M_{\vec{k}^{(j)}}(z,z')-M_{\vec{k}}(z,z')= \begin{cases} 
\ds \frac{\gamma_\frac{j-1}{2}}{2\pi i z'}1_{\Sigma_{1}}(z)1_{\Sigma_{2}}(z')\left(\frac{z}{z'}\right)^{k_{j}}, & \mbox{if } j \mbox{ is odd}, \\[0.2cm]
\ds \frac{-\gamma_{\frac{j}{2}}}{2\pi i z'}1_{\Sigma_{1}}(z)1_{\Sigma_{2}}(z')\left(\frac{z}{z'}\right)^{k_{j}+1}, & \mbox{if } j \mbox{ is even}.
\end{cases}
\end{align*}
This is the kernel of the rank one operator 
\[\Delta[f](z)=\vec{\mathrm{g}}_j(z)\int_{\Sigma_{2}}f(z')\vec{\mathrm{h}}_j(z')\frac{dz'}{z'}.\]
(If $k_{0}<+\infty$, then we again start indices at $0$, i.e. the first element of $\vec{\mathrm{h}}$ is denoted by $\vec{\mathrm{h}}_{0}$, etc, while if $k_{0}=+\infty$, then we start indices at $1$, i.e. the first element of $\vec{\mathrm{h}}$ is denoted by $\vec{\mathrm{h}}_{1}$. In particular, the notation $\vec{\mathrm{h}}_{j}$ should not be confused with $h_{j}$ from \eqref{def:G}. Similarly, $\vec{\mathrm{g}}_{j}$ should not be confused with $g_{j}$.)
By the matrix determinant lemma, we have
\begin{align*}
\det(1-M_{\vec{k}^{(j)}})_{L^2(\Sigma_1\cup\Sigma_2)}&=\det(1-M_{\vec{k}}-\Delta)_{L^2(\Sigma_1\cup\Sigma_2)}\\
&=
\det\left(1-M_{\vec{k}}\right)_{L^2(\Sigma_1\cup\Sigma_2)}\ \left(1- \int_{\Sigma_{2}}\frac{\vec{\mathrm{h}}_j(z)}{z}(1-M_{\vec{k}})^{-1}[\vec{\mathrm{g}}_j](z)dz\right).
\end{align*}
We can relate $(1-M_{\vec{k}})^{-1}[\vec{\mathrm{g}}_j](z)$ to the RH solution $U$. By \cite[Lemma 2.12]{DeiftItsZhou}, we have that
\begin{align*}
(1-M_{\vec{k}})^{-1}[\vec{\mathrm{g}}_j](z) & = \left(U_+(z)\vec{\mathrm{g}}(z)\right)_{j} = U_{j0,+}(z)\vec{\mathrm{g}}_0(z)+...+U_{(j,2q+1),+}(z)\vec{
\mathrm{g}}_{2q+1}(z)\\
& +U_{(j,2q+2),+}(z)\vec{\mathrm{g}}_{2q+2}(z) + \cdots + U_{(j,2q+1+d),+}(z)\vec{\mathrm{g}}_{2q+1+d}(z),\qquad z\in\Sigma_{2},
\end{align*}
where $U_{j0,+}(z):=0$ and $\vec{\mathrm{g}}_0(z):=0$ if $k_{0}=+\infty$. Since $\vec{\mathrm{g}}_0,\ldots, \vec{\mathrm{g}}_{2q+1}$ vanish on $\Sigma_{2}$, we obtain
\begin{align*}
\frac{\det(1-M_{\vec{k}^{(j)}})_{L^2(\Sigma_1\cup\Sigma_2)}}{\det(1-M_{\vec{k}})_{L^2(\Sigma_1\cup\Sigma_2)}}&=
1-\int_{\Sigma_{2}}\frac{\vec{\mathrm{h}}_j(z)}{z} \sum_{k=2q+2}^{2q+1+d} U_{jk,+}(z)\vec{\mathrm{g}}_{k}(z) dz.
\end{align*}
Now we observe that, by the jump relation for $U$,  
\begin{align*}
& \left(U_+(z)-U_-(z)\right)_{jj}=\left(U_+(z)(I-J_U(z)^{-1})\right)_{jj} = - 2\pi i \left(U_+(z)\vec{\mathrm{g}}(z) \vec{\mathrm{h}}^T(z)\right)_{jj} \\
& = - 2\pi i \vec{\mathrm{h}}_j(z) \sum_{k=2q+2}^{2q+1+d} U_{jk,+}(z)\vec{\mathrm{g}}_{k}(z), \qquad z\in\Sigma_{2}. 
\end{align*}
Consequently,
\begin{align*}
\frac{\det(1-M_{\vec{k}^{(j)}})_{L^2(\Sigma_1\cup\Sigma_2)}}{\det(1-M_{\vec{k}})_{L^2(\Sigma_1\cup\Sigma_2)}} & = 
1+\frac{1}{2\pi i}\int_{\Sigma_{2}}U_{jj,+}(z)\frac{dz}{z}
-\frac{1}{2\pi i}\int_{\Sigma_{2}}U_{jj,-}(z)\frac{dz}{z}
.
\end{align*}
Since $U_-$ can be analytically continued to the region outside $\Sigma_{2}$ ($U$ is the analytic continuation), the latter integral can be deformed to a large circle, and is equal to $-1$ by the residue theorem, since $\lim_{z\to\infty}U(z)=I$.
For the former integral, we have by analytic continuation of $U_+(z)$ to the region between $\Sigma_{2}$ and $\Sigma_{1}$ that
\[\frac{1}{2\pi i}\int_{\Sigma_{2}}U_{jj,+}(z)\frac{dz}{z}=\frac{1}{2\pi i}\int_{\Sigma_{1}}U_{jj,-}(z)\frac{dz}{z}=\frac{1}{2\pi i}\int_{\Sigma_{1}}U_{jj,+}(z)\frac{dz}{z}.\]
Here we used the jump relation of $U$ on $\Sigma_{1}$ for the last equality.
Now, $U_+$ can be analytically continued to the whole inner region of $\Sigma_{1}$, so by computing the residue, we obtain
\[\frac{1}{2\pi i}\int_{\Sigma_{2}}U_{jj,+}(z)\frac{dz}{z}=U_{jj}(0).\]
Combining the above results for the two integrals, we finally obtain
\begin{align*}
\frac{\det(1-M_{\vec{k}^{(j)}})_{L^2(\Sigma_1\cup\Sigma_2)}}{\det(1-M_{\vec{k}})_{L^2(\Sigma_1\cup\Sigma_2)}}&=
U_{jj}(0),
\end{align*}
and the claim follows by combining the above with \eqref{lol1}.
\end{proof}

\begin{remark}
Note that we did not use regularity properties (such as analyticity) of $g_{1},\ldots,g_{d},$ $h_{1},\ldots,h_{d}$. 
\end{remark}
\begin{remark}
We have a trivial degenerate case at our disposal: $h_{1}=\ldots=h_{d}=0$ on $\Sigma_1$ and $g_{1}=\ldots=g_{d}=0$ on $\Sigma_{2}$. The kernel is then identically equal to $0$ and the RH solution $U$ is identically equal to $I_{d+2q+1}$ (if $k_{0}=+\infty$) or $I_{d+2q+2}$ (if $k_{0}<+\infty$). The identity \eqref{eq:ratioid multi gap} reads $1=1$ and gives us a modest sanity check. Moreover, we also checked \eqref{eq:ratioid multi gap} numerically in the case of the kernels corresponding to domino tilings of Aztec diamonds and lozenge tilings of hexagons, for various values of the parameters (see also Remark \ref{remark:numerical check}).
\end{remark}

Proposition \ref{prop:dominoRH}
is a special case of Theorem \ref{thm:multi gap ratio}, corresponding to $\gamma_0=1$, $\mathcal I=[k,+\infty)\cap \Z$ and $G(u,v)=W(v)/W(u)$ with $W$ given by \eqref{def:W}.

Let us now discuss some further cases of interest, in the general setting where $K$ is of the form \eqref{eq:def K}, assuming in addition that $K$ is the kernel of a determinantal point process $\mathcal X$ on $\mathbb Z$. Then, the Fredholm determinants under consideration have the following probabilistic interpretations.
\begin{enumerate}
\item The determinant
$\det\left(1- 1_{\{k,k+1,\ldots\}}\mathcal K\right)_{\ell^2(\mathbb Z)}$, corresponding to $q=0, k_0=+\infty, \gamma_0=1$, is the probability that the largest particle in the determinantal point process is strictly less than $k$. Equivalently, it is the \textit{gap probability} of the set $\{k,k+1,\ldots\}$, i.e. it is the probability that no points in the determinantal point process lie on $\{k,k+1,\ldots\}$. By Theorem \ref{thm:multi gap ratio}, we can express ratios of such gap probabilities in terms of a $(d+1)\times (d+1)$ RH problem.

Obtaining precise asymptotics for gap probabilities as the size of the gap gets large, often referred to as \textit{large gap asymptotics}, is in general difficult. While there is a vast literature on large gap asymptotics for point processes associated with random matrix eigenvalues (see e.g. \cite{BP2024, F review gap, Kras review} for some reviews), in the case of tiling models this remains an outstanding problem, which we intend to address in a forthcoming publication.
\item The determinant $\det(1-\sum_{j=0}^q 1_{\mathcal I_j}\mathcal K)_{\ell^2(\mathbb Z)}$, corresponding to general $q$ but with $\gamma_0=\cdots=\gamma_q=1$, is the gap probability of the set $\mathcal I=\sqcup_{j=0}^q\mathcal I_j$. This quantity can be computed by taking suitable telescoping products of the quotients appearing in Theorem \ref{thm:multi gap ratio}.
\item The determinant $\det(1-\sum_{j=0}^q\gamma_j 1_{\mathcal I_j}\mathcal K)_{\ell^2(\mathbb Z)}$ with $\gamma_0,\ldots,\gamma_q\in[0,1]$ is the probability that a random thinning of the determinantal point process $\mathcal X$ has no particles in $\mathcal I$. This random thinning is constructed by removing each particle of $\mathcal X$ in $\mathcal I_j$ independently with probability $1-\gamma_j$. The same determinant also gives access to joint probability distributions of numbers of particles in the sets $\mathcal I_0,\ldots, \mathcal I_q$. See e.g. \cite[Section 2]{CD2018} and \cite[Section 1.2]{C sine} for more details. Again, such quantities can be computed by taking telescoping products of the quotients appearing in Theorem \ref{thm:multi gap ratio}.
\item
A particularly interesting case occurs for $q=1$, $k_0=+\infty$, $k_1=k$, $\gamma_0=\gamma_1=1$ and $k_2=k_3=n$, for some integers $n<k$.
If $\det\left(1- 1_{\{k,k+1,\ldots\}}\mathcal K\right)_{\ell^2(\mathbb Z)}>0$, which is equivalent to the fact that the probability that the largest particle in $\mathcal X$ is smaller than $k$ is non-zero, we have that $1- 1_{\{k,k+1,\ldots\}}\mathcal K$ is invertible such that
\begin{multline*}
\det\left(1- 1_{\{k,k+1,\ldots\}}\mathcal K-1_{\{n\}}\mathcal K\right)_{\ell^2(\mathbb Z)}\\=\det\big(1- 1_{\{n\}}\mathcal K(1- 1_{\{k,k+1,\ldots\}}\mathcal K)^{-1}\big)_{\ell^2(\mathbb Z)}\det\left(1- 1_{\{k,k+1,\ldots\}}\mathcal K\right)_{\ell^2(\mathbb Z)}.
\end{multline*}
As a consequence, writing $\mathcal R=\mathcal K\left(1- 1_{\{k,k+1,\ldots\}}\mathcal K\right)^{-1}$, we obtain
\begin{equation}\label{lol11}
\frac{\det\left(1- 1_{\{k,k+1,\ldots\}}\mathcal K-1_{\{n\}}\mathcal K\right)_{\ell^2(\mathbb Z)}}{\det\left(1- 1_{\{k,k+1,\ldots\}}\mathcal K\right)_{\ell^2(\mathbb Z)}} 
=\det\left(1- 1_{\{n\}}\mathcal R\right)_{\ell^2(\mathbb Z)}=1-R(n,n),
\end{equation}
where $R$ is the kernel of $\mathcal R$. This kernel $R(n,n')$ is the kernel of the determinantal point process $\mathcal X$, conditioned on the event that the largest particle is smaller than $k$ \cite{CG2023}. The identity \eqref{lol11}, combined with Theorem \ref{thm:multi gap ratio}, implies that we can express the one-point function $R(n,n)$ of this conditional determinantal point process in terms of a $(d+3)\times (d+3)$ RH problem.
\item For $q=0$ and $k_0=k_1=k$, we have $\mathcal I=\{k\}$. Then, as an easy consequence of the Fredholm series, we have the identity $\det(1-1_{\{k\}}\mathcal K)_{\ell^2(\mathbb Z)}=1-K(k,k)$
such that we have access to the kernel $K$ on the diagonal through the associated $(d+2)\times (d+2)$ RH problem for $U$. In cases where the double contour integral \eqref{eq:def K} is difficult to analyze (like for uniform hexagon tilings and also for doubly periodic tilings of Aztec diamonds \cite{DK2021}), it may be easier to analyze the kernel via this $(d+2)\times (d+2)$ RH problem, as this avoids to analyze a double contour integral.
\item As an alternative way to approach the kernel $K(k,k)$, we  can take $q=0$ with $k_0=+\infty$: by the identity $\lim_{\gamma_{0}\to 0}\frac{-1}{\gamma_{0}}\log\det\left(1-\gamma_{0} 1_{\{k,k+1,\ldots\}}\mathcal K\right)_{\ell^2(\mathbb Z)}=\sum_{j=k}^\infty K(j,j)$, we obtain that
\begin{align*}
K(k,k)=\lim_{\gamma_{0}\to 0}\frac{-1}{\gamma_{0}}\log\frac{\det\left(1-\gamma_{0} 1_{\{k,k+1,\ldots\}}\mathcal K\right)_{\ell^2(\mathbb Z)}}{\det\left(1-\gamma_{0} 1_{\{k+1, k+2\ldots\}}\mathcal K\right)_{\ell^2(\mathbb Z)}}=\lim_{\gamma_{0}\to 0}\frac{\log U_{11}(0;(k,+\infty))}{\gamma_{0}},
\end{align*}
such that we characterize $K(k,k)$ in terms of a (smaller size) $(d+1)\times (d+1)$ RH problem, at the price of having to take an additional $\gamma_0\to 0$ limit. 
\end{enumerate}

\section{Domino tilings of reduced Aztec diamonds}\label{section:tiling}

We now study in detail the RH problem for $U$ in the situation corresponding to the Aztec diamond with weight $a$ on the vertical dominoes and weight $1$ on the horizontal dominoes, i.e.\ when $W$ is given by \eqref{def:W2}.
First, we explicitly construct the first row of the RH solution $U$ in the one-gap case $\mathcal I=[k,+\infty)\cap \Z$. Afterwards, we do the same for $\mathcal I$ consisting of any finite number of intervals (here and in what follows, by \textit{intervals}, we mean \textit{clusters of consecutive integers}). Here, we assume that all intervals lie on the same level $r=2m-\epsilon$; we hope to generalize our method to the case of multiple levels in future research.

\subsection{One-gap case: proof of Theorem \ref{theorem:tilings}}\label{subsec:proof of main thm Aztec}
We let $N\in\mathbb N$, $m\in\{1,\ldots, N\}$, $\epsilon\in \{0,1\}$, $a\in (0,1]$, $\gamma=1$, $k\in\{1,\ldots, m\}$ and consider the following RH problem.
\subsubsection*{RH problem for $U$}
\begin{itemize}
\item[(a)] $U:\mathbb C\setminus\left(\Sigma_1\cup\Sigma_2\right)\to \mathbb C^{2\times 2}$ is analytic,
where $\Sigma_1$ encircles $a$, and $\Sigma_2$ encircles $0$ and $\Sigma_1$.
\item[(b)] $U_+(z)=U_-(z)J_U(z)$ for $z\in \Sigma_1\cup\Sigma_2$, with $J_U$ as in \eqref{eq:JY2}.
\item[(c)]{As $z\to \infty$, $U(z) = I +\bigO(z^{-1})$}.
\end{itemize}
Denote
\begin{align*}
V(z) = \frac{(z-a)^{N-m+\epsilon}(1+az)^m}{z^{N-m+k{+\epsilon-1}}},
\end{align*}
such that $J_U=\begin{pmatrix}1&1/V\\0&1\end{pmatrix}$ on $\Sigma_1$, and $J_U=\begin{pmatrix}1&0\\-V&1\end{pmatrix}$ on $\Sigma_2$.
Since $V$ is analytic in $\mathbb C\setminus\{0\}$, we can deform the jump contour $\Sigma_2$. For instance, by analytic continuation of $U$ from the region between $\Sigma_1$ and $\Sigma_2$ to the region outside $\Sigma_2$, we can enlarge $\Sigma_2$. 
Since $1/V$ is analytic in $\mathbb C\setminus\{-1/a,a\}$, we can in the same way deform $\Sigma_1$ to a small contour around $a$.

Propositions \ref{prop:dominoFredholm 2} and \ref{prop:integrable}, together with Theorem \ref{thm:multi gap ratio}, imply that 
\begin{align}\label{lol2}
\frac{F^{m,k+1,\epsilon}_{N}(a)}{F^{m,k,\epsilon}_{N}(a)} = U_{11}(0),
\end{align}
where $U(\cdot)=U(\cdot;(k_1=k,k_0=\infty))$ is of size $2\times 2$.

The RH problem for $U$ gives conditions for each row of $U$ independently. In view of \eqref{lol2}, we only need the first row.
Let us denote $U_1, U_2$ for the analytic continuations of $U_{11}, U_{12}$ from the region between $\Sigma_1$ and $\Sigma_2$ to $\mathbb C\setminus\{a\}$.
This means more precisely that
\[
U_1(z)=
\begin{cases}U_{11}(z)&\mbox{in the region inside $\Sigma_2$,}\\
U_{11}(z)-V(z)U_{12}(z)&\mbox{in the region outside $\Sigma_2$,}
\end{cases}
\]
and that
\[
U_2(z)=
\begin{cases}U_{12}(z)&\mbox{in the region outside $\Sigma_1$,}\\
U_{12}(z)-\frac{1}{V(z)}U_{11}(z)&\mbox{in the region inside $\Sigma_1$.}
\end{cases}
\]

Translating the RH conditions in terms of $U_1,U_2$, we obtain that 
\begin{align}
& U_{2}(z) = \bigO(z^{-1}), \quad \mbox{as } z \to \infty, & & U_{1}(z) =\bigO(1), \quad \mbox{as } z \to a, \label{U1 U2 1} \\
& U_{1}(z)+V(z)U_2(z) = 1+\bigO(z^{-1}), \quad \mbox{as } z \to \infty, & & U_{2}(z) + \frac{1}{V(z)}U_{1}(z) = \bigO(1), \quad \mbox{as } z \to a. \label{U1 U2 2}
\end{align}
In particular, from the first relation in \eqref{U1 U2 1}, the first in \eqref{U1 U2 2}, and the fact that $V(z) = \bigO(z^{m-k+1})$ as $z\to \infty$, we see that $U_1$ is entire and of order $\bigO(z^{m-k})$ at infinity, so $U_1:=\tilde q_{m-k}$ is a polynomial of degree at most $m-k$.
Multiplying the first relation in \eqref{U1 U2 2} by $\frac{z^{N-m+k-1+\epsilon}}{(z-a)^{N-m+\epsilon}}$, we also see that
\[\tilde P_{k-1}(z):=(1+az)^mU_2(z)+\frac{z^{N-m+k-1+\epsilon}}{(z-a)^{N-m+\epsilon}}\tilde q_{m-k}(z)\] is $z^{k-1}+\bigO(z^{k-2})$ as $z\to\infty$, and moreover it has a removable singularity at $z=a$ because of the second relation in \eqref{U1 U2 2}.
Hence it
is a monic polynomial of degree $k-1$.
But then 
\[\tilde P_{k-1}(z)-\frac{z^{N-m+k-1+\epsilon}}{(z-a)^{N-m+\epsilon}}\tilde q_{m-k}(z)=\bigO((1+az)^m),\qquad \mbox{as } z\to -\tfrac{1}{a}.\]
This is a linear system of $m$ equations for the $(k-1)+(m-k+1)=m$ coefficients of $\tilde P_{k-1}$ and $\tilde q_{m-k}$.
It means that the rational function $\tilde P_{k-1}/\tilde q_{m-k}$ approximates the (also rational) function 
$\tilde f(z):=\frac{z^{N-m+k-1+\epsilon}}{(z-a)^{N-m+\epsilon}}$ up to order $\bigO((1+az)^m)$ as $z\to -\frac{1}{a}$. Such rational approximations are called Pad\'e approximants: more precisely, $\tilde P_{k-1}/\tilde q_{m-k}$ is the type $[k-1,m-k]$ Pad\'e approximant of $\tilde f$ \cite{VanAssche}.
To further simplify this problem, we can make a change of variables $z=1/\zeta$, such that \[\tilde P_{k-1}(1/\zeta)-\frac{\zeta^{-k+1}}{(1-a\zeta)^{N-m+\epsilon}}\tilde q_{m-k}(1/\zeta)=\bigO((\zeta+a)^m),\qquad \zeta\to -a.\]
Multiplying by $-(1-a\zeta)^{N-m+\epsilon}\zeta^{m-1}$, we find
\[p_{m-k}(\zeta)-\zeta^{m-k}(1-a\zeta)^{N-m+{\epsilon}}q_{k-1}(\zeta)=\bigO((\zeta+a)^m),\qquad \zeta\to -a,\]
where 
\[q_{k-1}(\zeta)=\zeta^{k-1}\tilde P_{k-1}(1/\zeta),\qquad p_{m-k}(\zeta)=\zeta^{m-k}\tilde q_{m-k}(1/\zeta),\]
such that $q_{k-1}$ is of degree $k-1$ and $p_{m-k}$ is of degree $m-k$. 
In other words, $p_{m-k}/q_{k-1}$ is the type $[m-k,k-1]$ Pad\'e approximant of the degree $N-k+\epsilon$ polynomial $f(\zeta)=\zeta^{m-k}(1-a\zeta)^{N-m+\epsilon}$, normalized such that $q_{k-1}(0)=1$.
We can recover $U_{11}(0)$ by the formula
\[U_{11}(0)=U_1(0)=\tilde q_{m-k}(0)=\lim_{\zeta\to\infty}{\zeta^{k-m}p_{m-k}(\zeta)}=:\kappa_{m-k}^{N,k,m},\]
which is the leading coefficient of $p_{m-k}$ at infinity. Observe that the uniqueness and existence of the RH solution $U$ implies the existence and uniqueness of the above Pad\'e approximants.
The above construction proves Proposition \ref{prop:dominoPade}. Taking $\epsilon=1$, Theorem \ref{theorem:tilings} is now an immediate consequence of Propositions \ref{prop:dominoFredholm}, \ref{prop:dominointegrable}, \ref{prop:dominoRH}, and \ref{prop:dominoPade}.

For general values of $\epsilon \in \{0,1\}$, the above proves Theorem \ref{theorem:tilingmulti} (see below) with $q=0$.

\subsection{Multi-gap case}
In this subsection, we let $N\in\mathbb N$, $m\in\{1,\ldots, N\}$, $\epsilon\in \{0,1\}$, $a\in (0,1]$, $q\in \N$, and $\gamma_{0}=\ldots=\gamma_{d}=1$. (The content of this subsection can easily be generalized for general values of $\gamma_{0},\ldots,\gamma_{d}$; we set $\gamma_{0}=\ldots=\gamma_{d}=1$ for simplicity.)

We consider a subset $\mathcal I$ of $\mathbb Z$ which is bounded below and consists of $q+1$ clusters of consecutive points. We restrict ourselves here to the case $k_0=+\infty$, so 
$\mathcal I=\sqcup_{j=0}^q\mathcal I_j$ where $\mathcal I_j=[k_{2j+1},k_{2j}]\cap\mathbb Z$,
with $k_{2j}<k_{2j-1}$ for $j=1,\ldots, q$, and $k_0=+\infty$. We also assume that  $k_1\leq m$, $k_{2q+1}\geq m-N-\epsilon+1$, and that
\begin{align}\label{conditionmultigap}
m-1 \geq (m-k_{1})+\sum_{j=1}^{q}(k_{2j}+1-k_{2j+1}), \quad \mbox{or equivalently} \quad k_{1} \geq 1+ \sum_{j=1}^{q}(k_{2j}+1-k_{2j+1}).
\end{align}
The latter condition is needed to have a non-zero gap probability; if it is not valid, there would be less than $N$ admissible sites in the corresponding point process (recall Remark \ref{remark:notilings}). The RH problem for $U$ would then not be uniquely solvable.
For later use, we write
\begin{align}\label{def of Ical prim}
\mathcal I'=\mathcal I\cap\{1,2,\ldots, m\}=\left(\sqcup_{j=0}^q[k_{2j+1},k_{2j}]\right)\cap\{1,2,\ldots,m\},
\end{align}
and
\begin{align}\label{def of Jcal}
\mathcal J=\cup_{d=1}^q\{k_{2d}+1,\ldots, k_{2d-1}-1\}=\{k_{2q}+1,\ldots, k_1-1\}\setminus\mathcal I'.
\end{align}
These sets will be used in \eqref{expr1} and in the analysis following \eqref{expr1}.
We also denote $A_{N}^{m,\mathcal I}$ and $\tilde A_{N}^{m,\mathcal I}$ for the reduced Aztec diamonds, obtained by assuming that there are no particles in $\mathcal I$ at level $2m-\epsilon$. In analogy with \eqref{def of A tilde}, the notation $A_{N}^{m,\mathcal I}$ corresponds to $\epsilon=1$ and $\tilde A_{N}^{m,\mathcal I}$ to $\epsilon=0$. We further define
\[F_N^{m,\mathcal I}(a)=\sum_{T\in\mathcal T(A_N^{m,\mathcal I})}a^{v(T)},\qquad \tilde F_N^{m,\mathcal I}(a)=\sum_{T\in\mathcal T(\tilde A_N^{m,\mathcal I})}a^{v(T)}\]
for the generating functions of vertical dominoes in tilings of these domains.

The RH problem corresponding to this case, in which we denote 
\begin{align}\label{def of Vj}
{V}_{j}(z) =  
\frac{(z-a)^{N-m+{\epsilon}}(1+az)^m}{z^{N-m{+\epsilon-1}+k_{j}+\epsilon_j}},\qquad \epsilon_j=\begin{cases}0&\mbox{ if $j$ is odd},\\
1&\mbox{ if $j$ is even,}\end{cases}
\end{align}
for simplicity, is the following.
\subsubsection*{RH problem for $U$}
\begin{itemize}
\item[(a)] ${U}:\mathbb C\setminus(\Sigma_1\cup\Sigma_2)\to \mathbb C^{(2q+2)\times (2q+2)}$ is analytic, where $\Sigma_1$ encircles $a$ but not $-\frac{1}{a}$, and $\Sigma_2$ encircles $0$ and $\Sigma_1$.
\item[(b)] ${U}_+(z)={U}_-(z)J_{{U}}(z)$ for $z\in \Sigma_1\cup\Sigma_2$, with \begin{align*}
& J_{U} = 
\begin{pmatrix}
1 & 0 & 0 & \cdots & 0 & 0 & 1/V_1 \\ 
0 & 1 & 0 & \cdots & 0 & 0 & -1/V_2 \\ 
0 & 0 & 1 & \cdots & 0 & 0 & 1/V_3 \\ 
\vdots & \vdots & \vdots & \ddots & \vdots & \vdots & \vdots \\
0 & 0 & 0 & \cdots & 1 & 0 & -1/V_{2q} \\ 
0 & 0 & 0 & \cdots & 0 & 1 & 1/V_{2q+1} \\ 
0 & 0 & 0 & \cdots & 0 & 0 & 1
\end{pmatrix}, & & \mbox{on }\Sigma_1, \\
& J_{U}(z) = \begin{pmatrix}
1 & 0 & 0 & \cdots & 0 & 0 & 0 \\ 
0 & 1 & 0 & \cdots & 0 & 0 & 0 \\ 
0 & 0 & 1 & \cdots & 0 & 0 & 0 \\ 
\vdots & \vdots & \vdots & \ddots & \vdots & \vdots & \vdots \\
0 & 0 & 0 & \cdots & 1 & 0 & 0 \\ 
0 & 0 & 0 & \cdots & 0 & 1 & 0 \\ 
-V_1 & -V_2 & -V_3 & \cdots & -V_{2q} & -V_{2q+1} & 1
\end{pmatrix}, & & \mbox{on }\Sigma_2. 
\end{align*}
\item[(c)]{As $z\to \infty$, ${U}(z) = I +\bigO(z^{-1})$}.
\end{itemize}

Since $V_1,\ldots, V_{2q+1}$ are analytic outside $\Sigma_2$, we can deform $\Sigma_2$ to an arbitrarily large contour, and thus extend $U$ analytically from the region in between $\Sigma_1$ and $\Sigma_2$ to the whole exterior of $\Sigma_2$. Similarly, since $1/V_1,\ldots,1/V_{2q+1}$ are analytic inside $\Sigma_1$ except at $a$, we can shrink $\Sigma_1$ to an arbitrarily small circle around $a$ and thus extend $U$ analytically from the region in between $\Sigma_1$ and $\Sigma_2$ to the whole interior of $\Sigma_1$ except at $a$. 

We now fix a row $j_{*}\in\{1,\ldots, 2q+1\}$ of the matrix $U$ and denote $U_1,\ldots, U_{2q+2}$ for the analytic continuations of $U_{{j_{*},}1},\ldots, U_{{j_{*}},2q+2}$ from the region between $\Sigma_1$ and $\Sigma_2$ to $\mathbb C\setminus\{a\}$.
The RH conditions imply that
\begin{align}
& U_{2q+2}(z) = \bigO(z^{-1}), & & \mbox{as } z \to \infty,\label{U1 U2 multi1}
\\ & U_{j}(z) =\bigO(1), & & \mbox{as } z \to a,\ j=1,\ldots, 2q+1, \label{U1 U2 multi2} \\
& U_{j}(z)+V_j(z)U_{2q+2}(z) = \delta_{j{,j_{\star}}}+\bigO(z^{-1}), & & \mbox{as } z \to \infty,\ j=1,\ldots, 2q+1, \label{U1 U2 multi3}\\ & U_{2q+2}(z) + \sum_{j=1}^{2q+{1}}\frac{(-1)^{j+1}}{V_j(z)}U_{j}(z) = \bigO(1), & &\mbox{as } z \to a. \label{U1 U2 multi4}
\end{align}
In particular, from \eqref{U1 U2 multi1} and \eqref{U1 U2 multi3} together with \eqref{def of Vj}, we see that $U_{j}(z) = \bigO(z^{m-k_{j}-\epsilon_{j}})$ as $z\to\infty$, for $j=1,\ldots,2q+1$. By \eqref{U1 U2 multi2}, $U_j$ for $j=1,\ldots, 2q+1$ is an entire function, hence it is a polynomial of degree 
at most $m-k_j-\epsilon_j$ (here we use that $m-k_{1}\geq 0$). On the other hand, $U_{2q+2}$ has an isolated singularity at $z=a$, which is a pole of order at most $N-m+\epsilon$, by \eqref{U1 U2 multi4}, \eqref{def of Vj} and \eqref{U1 U2 multi2}. Hence by \eqref{U1 U2 multi1}, $U_{2q+2}(z)$ has the form
\[U_{2q+2}(z)=\frac{\tilde s_{N-m+\epsilon-1}(z)}{(z-a)^{N-m+\epsilon}}\] with $\tilde s_{N-m+\epsilon-1}$ a polynomial of degree at most $N-m+\epsilon-1$.
It is more convenient to change variable $z=1/\zeta$, such that
\begin{align}\label{U2qP2 1/zeta}
U_{2q+2}(1/\zeta)=\frac{\zeta s_{N-m+\epsilon-1}(\zeta)}{(1-a\zeta)^{N-m+\epsilon}},
\end{align}
with $s_{N-m+\epsilon-1}$ a polynomial of degree at most $N-m+\epsilon-1$.

Combining \eqref{U1 U2 multi3} for two consecutive values of $j$ and taking into account \eqref{def of Vj}, we obtain that 
\begin{align*}
{\tilde{U}}_{j+1}-z^{k_{j}-k_{j+1}+(-1)^j}{\tilde{U}}_{j}=r_{j},\qquad {\tilde{U}_{j}:=U_{j}-\delta_{j,j_{\star}},} \qquad j=1,\ldots, 2q,
\end{align*}
with $r_j$ a polynomial of degree at most $k_j-k_{j+1}+(-1)^j-1$. (If $k_j-k_{j+1}+(-1)^j-1=-1$, then $r_{j}=0$ and $\tilde{U}_{j}=\tilde{U}_{j+1}$.)
In other words, $\tilde{U}_j$ is the quotient under Euclidean division of $\tilde{U}_{j+1}$ by $z^{k_{j}-k_{j+1}+(-1)^j}$, and $r_j$ is the remainder. Observe that if we know the $m-k_{2q+1}+1$ coefficients of $\tilde{U}_{2q+1}$, this fixes all of the functions $\tilde U_1,\ldots, \tilde U_{2q+1}$ and $r_1,\ldots, r_{2q}$. Concretely, writing
\begin{align}\label{Utilde 2q+1 1/zeta}
\zeta^{m-k_{2q+1}}\tilde U_{2q+1}(1/\zeta)=\sum_{i=k_{2q+1}}^{m}a_{m-i}\zeta^{m-i},
\end{align}
we have
\begin{align*}
&\zeta^{m-k_{j}-\epsilon_j}\tilde U_{j}(1/\zeta)=\sum_{i=k_{j}+\epsilon_j}^ma_{m-i} \zeta^{m-i},\qquad j=1,\ldots, 2q.
\end{align*}
Consequently, substituting this for $j=1,\ldots, 2q+1$ into \eqref{U1 U2 multi4} with $z=1/\zeta$, and using \eqref{def of Vj} and \eqref{U2qP2 1/zeta}, we obtain a telescoping sum and find after a straightforward computation that
\begin{equation}\label{expr1}\frac{\zeta s_{N-m+\epsilon-1}(\zeta) }{(1-a\zeta)^{N-m+\epsilon}}+\frac{\zeta}{(\zeta+a)^{m}(1-a\zeta)^{N-m+\epsilon}}\left(\sum_{i\in\mathcal I'} a_{m-i} \zeta^{m-i}-(-1)^{j_{*}}\zeta^{m-k_{j_{*}}-\epsilon_{j_{*}}}\right)=\bigO(1),\qquad \zeta\to 1/a,
\end{equation}
where we recall that $\mathcal{I}'$ is defined in \eqref{def of Ical prim}.
Alternatively, using that $\frac{V_{2q+1}(1/\zeta)}{V_{j}(1/\zeta)} = \zeta^{k_{2q+1}-k_{j}-\epsilon_{j}}$, we can express \eqref{U1 U2 multi4} with $z=1/\zeta$ as
\begin{align*}
\frac{1}{V_{2q+1}(1/\zeta)}\sum_{j=1}^{2q+1}(-1)^{j+1}\zeta^{k_{2q+1}-k_j-\epsilon_j}\left(U_j(1/\zeta)+V_{j}(1/\zeta)U_{2q+2}(1/\zeta)\right),
\end{align*}
which is $a^{-m}\zeta^{m-k_{1}+1}+\bigO(\zeta^{m-k_{1}+2})$ as $\zeta\to 0$ if $j_{*}=1$ and $\bigO(\zeta^{m-k_{1}+2})$ as $\zeta\to 0$ otherwise, by \eqref{U1 U2 multi3} and \eqref{def of Vj}.
Using this while multiplying \eqref{expr1} with $(\zeta+a)^m/\zeta^{m-k_1+1}$, we obtain that
\[
q_{k_1-1}(\zeta):=\frac{(\zeta+a)^m s_{N-m+\epsilon-1}(\zeta) }{\zeta^{m-k_1}(1-a\zeta)^{N-m+\epsilon}}+\frac{1}{\zeta^{m-k_1}(1-a\zeta)^{N-m+\epsilon}}\left(\sum_{i\in\mathcal I'} a_{m-i} \zeta^{m-i}-(-1)^{j_{*}}\zeta^{m-k_{j_{*}}-\epsilon_{j_{*}}}\right)
\]
is a polynomial of degree at most $k_1-1$ (here we used $k_{2q+1}\geq m-N-\epsilon+1$). If $j_{*}\neq 1$, we have moreover that $q_{k_1-1}(0)=0$, otherwise $q_{k_1-1}(0)=1$.

We find that 
\begin{multline}\label{exprs}
\zeta^{m-k_1}(1-a\zeta)^{N-m+\epsilon}q_{k_{1}-1}(\zeta)-\sum_{i\in\mathcal I'} a_{m-i} \zeta^{m-i}+(-1)^{j_{*}}\zeta^{m-k_{j_{*}}-\epsilon_{j_{*}}} \\
=(\zeta+a)^{m}s_{N-m+\epsilon-1}(\zeta)=\bigO((\zeta+a)^m),\qquad \zeta\to -a.
\end{multline}

Now, \eqref{U1 U2 multi3} with $j=2q+1$, $z=1/\zeta$
gives us
\[
\tilde U_{2q+1}(1/\zeta)+ V_{2q+1}(1/\zeta)U_{2q+2}(1/\zeta)=\bigO(\zeta),\quad  \zeta\to 0,\]
or equivalently, multiplying by $\zeta^{m-k_{2q+1}}$ and using \eqref{def of Vj}, \eqref{U2qP2 1/zeta} and \eqref{Utilde 2q+1 1/zeta},
\[
\sum_{i=k_{2q+1}}^{m}a_{m-i}\zeta^{m-i}+ (\zeta+a)^m s_{N-m+\epsilon-1}(\zeta)=\bigO(\zeta^{m-k_{2q+1}+1}),\quad  \zeta\to 0.\]
Substituting the expression obtained for $(\zeta+a)^{m}s_{N-m+\epsilon-1}(\zeta)$ in \eqref{exprs}, we get
\[
\zeta^{m-k_1}(1-a\zeta)^{N-m+\epsilon}q_{k_{1}-1}(\zeta)+\sum_{i\in \mathcal J}a_{m-i}\zeta^{m-i}+(-1)^{j_{*}}\zeta^{m-k_{j_{*}}-\epsilon_{j_{*}}}=O(\zeta^{m-k_{2q+1}+1}),\qquad \zeta\to 0,\]
where we recall that $\mathcal{J}$ is defined in \eqref{def of Jcal}.
In other words,
\[(1-a\zeta)^{N-m+\epsilon}q_{k_{1}-1}(\zeta)+\sum_{i\in \mathcal J}a_{m-i}\zeta^{k_1-i}+(-1)^{j_{*}}\zeta^{k_1-k_{j_{*}}-\epsilon_{j_{*}}}=O(\zeta^{k_1-k_{2q+1}+1}),\qquad \zeta\to 0.\]
We also have the identity
\[U_{j_{*},j_{*}}(0)=U_{j_{*}}(0)=\tilde U_{j_{*}}(0)+1=a_{m-k_{j_{*}}-\epsilon_{j_{*}}}+1.\]
The RH problem for the row $j_{*}$ of $U$ is thus equivalent to the problem of finding a polynomial $q_{k_1-1}$ of degree at most $k_1-1$ and coefficients $a_0,\ldots, a_{m-k_{2q+1}}$, such that
\begin{align}
&\label{eq:linsys1}\zeta^{m-k_1}(1-a\zeta)^{N-m+\epsilon}q_{k_{1}-1}(\zeta)-\sum_{i\in\mathcal I'} a_{m-i} \zeta^{m-i}+(-1)^{j_{*}}\zeta^{m-k_{j_{*}}-\epsilon_{j_{*}}}=\bigO((\zeta+a)^m),\qquad \zeta\to -a,\\
&\label{eq:linsys2}(1-a\zeta)^{N-m+\epsilon}q_{k_{1}-1}(\zeta)+\sum_{i\in \mathcal J}a_{m-i}\zeta^{k_1-i}+(-1)^{j_{*}}\zeta^{k_1-k_{j_{*}}-\epsilon_{j_{*}}}=O(\zeta^{k_1-k_{2q+1}+1}),\qquad \zeta\to 0.
\end{align}
This system determines $q_{k_1-1}$ and $a_0,\ldots, a_{m-k_{2q+1}}$ uniquely, because of the unique solvability of the RH problem for $U$. 

In conclusion, we proved the following result.

\begin{theorem}\label{theorem:tilingmulti}
Let $j\in \{1,\ldots,2q+1\}$. If $j$ is odd, we further assume $k_j<k_{j-1}$, while if $j$ is even we assume that $k_j+1<k_{j-1}$. Let $\mathcal I^*:=\sqcup_{j=0}^q\mathcal I_j^*$ with $\mathcal I_j^*$ as in Theorem \ref{thm:multi gap ratio}.
We have the identity
\begin{align}\label{lol12}
\frac{F_{N}^{m,{\mathcal I^*}}(a)}{F_N^{m,\mathcal I}(a)}=a_{m-k_j-\epsilon_j}+1 \;\mbox{ if $\epsilon=1$,}\qquad \frac{\tilde F_{N}^{m,{\mathcal I^*}}(a)}{\tilde F_N^{m,\mathcal I}(a)}=a_{m-k_j-\epsilon_j}+1 \; \mbox{ if $\epsilon=0$,}
\end{align}
where $a_{m-k_j-\epsilon_j}$ is determined by solving \eqref{eq:linsys1}--\eqref{eq:linsys2} (with $j_{*}$ replaced by $j$).
\end{theorem}
\begin{remark}The system \eqref{eq:linsys1}--\eqref{eq:linsys2} is a linear system of $m+(k_1-k_{2q+1}+1)$ equations for the unknowns $a_0,\ldots, a_{m-k_{2q+1}}$ and the $k_1$ coefficients of $q_{k_1-1}$. 
\end{remark}
\begin{remark}Note that $\mathcal J$ is empty in case $q=0$, such that \eqref{eq:linsys2} is simply equivalent to the normalization $q_{k_1-1}(0)=1$, and we recover Theorem \ref{theorem:tilings} if $\epsilon=1$ by taking a telescoping product (and for general $\epsilon\in\{0,1\}$, we recover the result of Subsection \ref{subsec:proof of main thm Aztec}). The quantities appearing on the right-hand side of \eqref{lol12} can thus be interpreted as coefficients of a generalization of Pad\'{e} approximants.
\end{remark}

\begin{figure}
\begin{center}
\hspace{-1cm}\begin{tikzpicture}[master]
\node at (0,0) {\includegraphics[width=6.5cm]{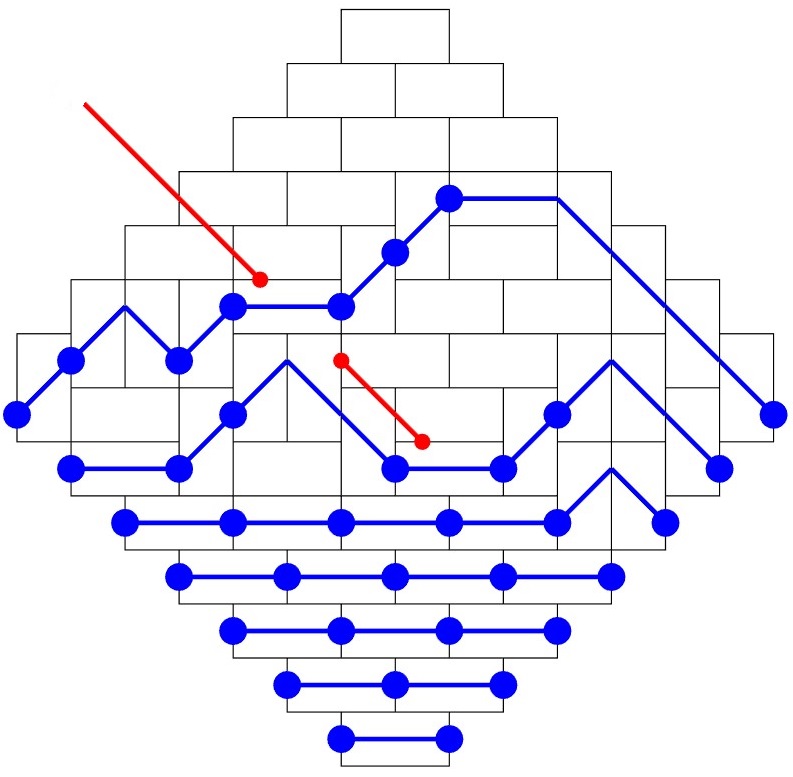}};
\end{tikzpicture} \hspace{0.5cm}
\begin{tikzpicture}[slave]
\node at (0,0) {\includegraphics[width=6.5cm]{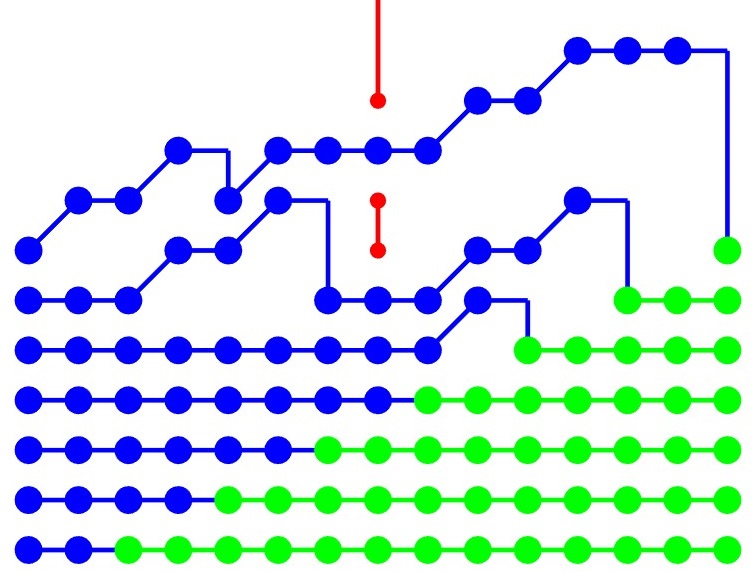}};
\end{tikzpicture}
\end{center}
\caption{\label{fig:gap 2}Here $N=7$ and the gap is $\{(r,k) : (r,k) \in \{7\}\times (\{0,1\} \cup \{3,4,\ldots\}) \}$.}
\end{figure}
\begin{figure}
\begin{center}
\hspace{-1cm}\begin{tikzpicture}[master]
\node at (0,0) {\includegraphics[width=6.5cm]{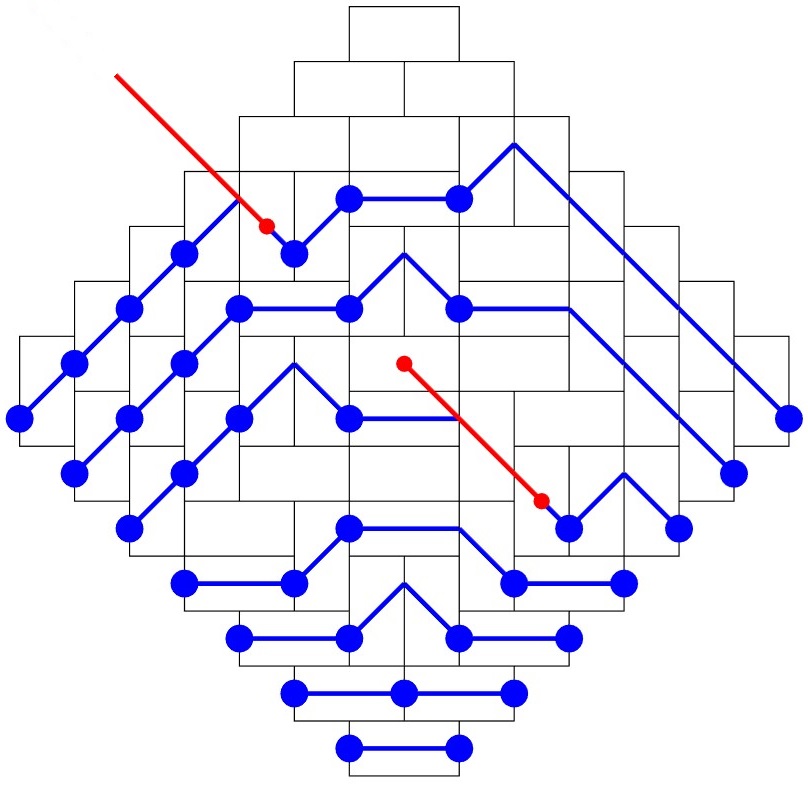}};
\end{tikzpicture} \hspace{0.5cm}
\begin{tikzpicture}[slave]
\node at (0,0) {\includegraphics[width=6.5cm]{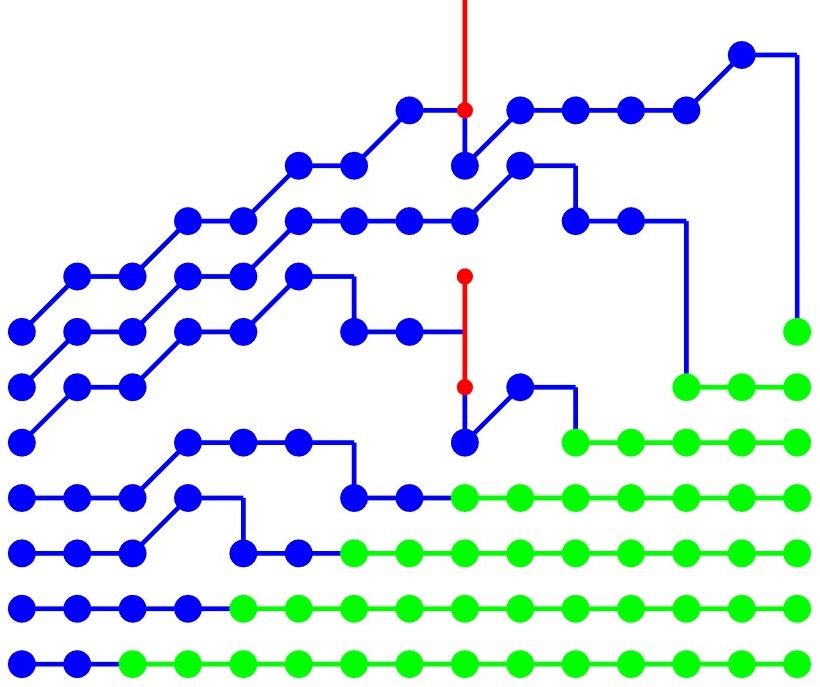}};
\end{tikzpicture}
\end{center}
\caption{\label{fig:gap 3}Here $N=7$ and the gap is $\{(r,k) : (r,k) \in \{8\}\times (\{-1,0,1\}\cup \{4,5,\ldots\}) \}$.}
\end{figure}

\section{Uniform lozenge tilings of reduced hexagons}

In this section, we consider uniformly distributed lozenge tilings of reduced hexagons. Theorem \ref{theorem:hextiling} is proved in Subsection \ref{subsec: hexa semi-infinite q=0}, in which we deal with the case of one semi-infinite interval (or more precisely, one unbounded cluster of consecutive integers, corresponding to $q=0$ and $k_{0}=+\infty$ in the setting of Theorem \ref{thm:multi gap ratio}). The analysis of Subsection \ref{subsec: hexa semi-infinite q=0} extends easily to general $q\in \N$, but we set $q=0$ for clarity. In Subsection \ref{subsec: hexa several finite intervals}, we consider the case of several bounded intervals (corresponding to general $q\in \N$ and $k_{0}<+\infty$ in the setting of Theorem \ref{thm:multi gap ratio}).

\subsection{The semi-infinite interval: proof of Theorem \ref{theorem:hextiling}}\label{subsec: hexa semi-infinite q=0}
We let $L,M,N\in\mathbb N_{>0}$ with $M<L$, $r\in\{1,\ldots, L-1\}$, $\max\{N,N-L+M+r\}\leq k\leq \min\{M+N-1,r+N-1\}$. Recall that $H_{L,M,N}^{r,k}$ is the reduced hexagon defined as the intersection of $H_{L,M,N}$ with the region $\{(x,y)\in\mathbb R^2: y\leq \max\{k, x+k-r\}\}$, see also Figure \ref{fig:hexagons grey regions and gap} and Figure \ref{fig:gap hexa 6} (left). Recall also that $G_{L,M,N}^{r,k}$ is the number of lozenge tilings of $H_{L,M,N}^{r,k}$, and that we have the identity \eqref{id:Fredholmhexagon}.
From Proposition \ref{prop:hexagonFredholm} and Theorem \ref{thm:multi gap ratio}, we know that
\begin{align}\label{lol3}
\frac{G_{L,M,N}^{r,k+1}}{G_{L,M,N}^{r,k}}=U_{11}(0),
\end{align}
where $U$ solves the following RH problem.
\subsubsection*{RH problem for $U$}
\begin{itemize}
\item[(a)] $U:\mathbb C\setminus\left(\Sigma_1\cup\Sigma_2\right)\to \mathbb C^{3\times 3}$ is analytic,
where $\Sigma_1,\Sigma_2$ are closed loops oriented positively. Moreover, $\Sigma_1$ encircles $0$, and $\Sigma_2$ encircles $\Sigma_1$.
\item[(b)] $U_+(z)=U_-(z)J_U(z)$ for $z\in \Sigma_1\cup\Sigma_2$, with $J_U$ given by
\begin{align}\label{eq:JU Hexagon}
J_U(z)=\begin{pmatrix}
1 & 1_{\Sigma_1}(z)z^{k}h_1(z) & 1_{\Sigma_1}(z)z^{k}h_2(z)
  \\
- 1_{\Sigma_2}(z) z^{-k}g_{1}(z) & 1 & 0 \\
- 1_{\Sigma_2}(z) z^{-k}g_{2}(z) & 0 & 1
\end{pmatrix},
\end{align}
i.e.
\begin{align*}
& J_U(z) = \begin{pmatrix}
1 & \frac{(1+z)^{L-r}}{z^{M+N-k}}Y_{11}(z) & -\frac{(1+z)^{L-r}}{z^{M+N-k}}Y_{21}(z)
  \\
0 & 1 & 0 \\
0 & 0 & 1
\end{pmatrix}, & & z \in \Sigma_{1}, \\
& J_U(z) = \begin{pmatrix}
1 & 0 & 0 \\[0.05cm]
-\frac{(1+z)^{r}}{z^{k}}Y_{21}(z) & 1 & 0 \\[0.05cm]
-\frac{(1+z)^{r}}{z^{k}}Y_{11}(z) & 0 & 1
\end{pmatrix}, & & z \in \Sigma_{2}.
\end{align*}
\item[(c)]{As $z\to \infty$, $U(z) = I +\bigO(z^{-1})$}.
\end{itemize}
Here $Y$ is itself the solution of the $2\times 2$ RH problem given in Subsection \ref{subsection: lozenge tilings}.

\subsubsection*{RH problem for $Y$}
\begin{itemize}
\item[(a)] $Y:\C\setminus \gamma \to \C^{2\times 2}$ is analytic, where $\gamma$ is a closed curve oriented positively and surrounding $0$.
\item[(b)] $Y$ satisfies
\begin{align*}
Y_{+}(z) = Y_{-}(z) \begin{pmatrix}
1 & \frac{(1+z)^{L}}{z^{M+N}} \\
0 & 1
\end{pmatrix}, \qquad z \in \gamma.
\end{align*}
\item[(c)] $Y(z) = (I+O(z^{-1}))z^{N\sigma_{3}}$ as $z\to \infty$.
\end{itemize}

We can considerably simplify the RH problem for $U$ by employing the RH conditions for $Y$.
First, we observe that we can take $\Sigma_1=\Sigma_2=\gamma$ in the RH problem for $U$, where we recall that $\gamma$ is the contour of the RH problem for $Y$. The jump matrix for $U$ then becomes
\begin{align*}
J_U(z) = \begin{pmatrix}
1 & \frac{(1+z)^{L-r}}{z^{M+N-k}} Y_{11}(z) & -\frac{(1+z)^{L-r}}{z^{M+N-k}} Y_{21}(z) \\
-\frac{(1+z)^{r}}{z^{k}} Y_{21}(z) & 1-\frac{(1+z)^{L}}{z^{M+N}} Y_{11}(z) Y_{21}(z) & \frac{(1+z)^{L}}{z^{M+N}} Y_{21}(z)^{2} \\
-\frac{(1+z)^{r}}{z^{k}} Y_{11}(z) & -\frac{(1+z)^{L}}{z^{M+N}} Y_{11}(z)^{2} & 1+\frac{(1+z)^{L}}{z^{M+N}} Y_{11}(z) Y_{21}(z)
\end{pmatrix}, \qquad z\in\gamma.
\end{align*}

We have the factorization
\begin{equation}
J_U=\tilde Y J\tilde Y^{-1},\qquad \tilde Y=\begin{pmatrix}1&0&0\\
0&Y_{21}&Y_{22}\\
0&Y_{11}&Y_{12}
\end{pmatrix},\qquad J(z)=
\begin{pmatrix}
1&0&\frac{(1+z)^{L-r}}{z^{M+N-k}} \\
-\frac{(1+z)^r}{z^k}&1&-\frac{(1+z)^L}{z^{M+N}} \\
0&0&1
\end{pmatrix}.
\end{equation}
\begin{figure}
\begin{center}
\hspace{-1cm}\begin{tikzpicture}[master]
\node at (0,0) {\includegraphics[width=6.5cm]{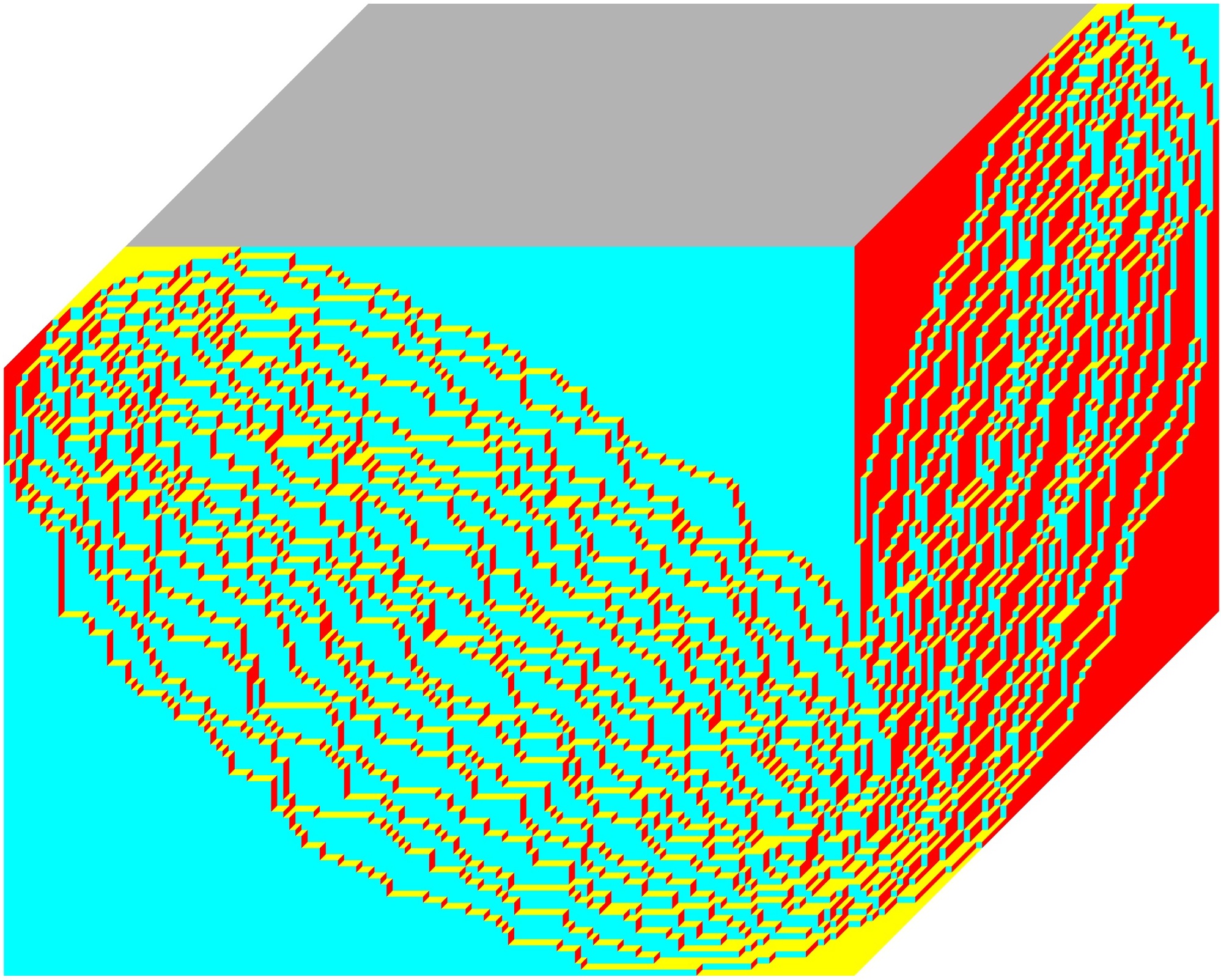}};
\end{tikzpicture} \hspace{0.5cm} \begin{tikzpicture}[slave]
\node at (0,0) {\includegraphics[width=6.5cm]{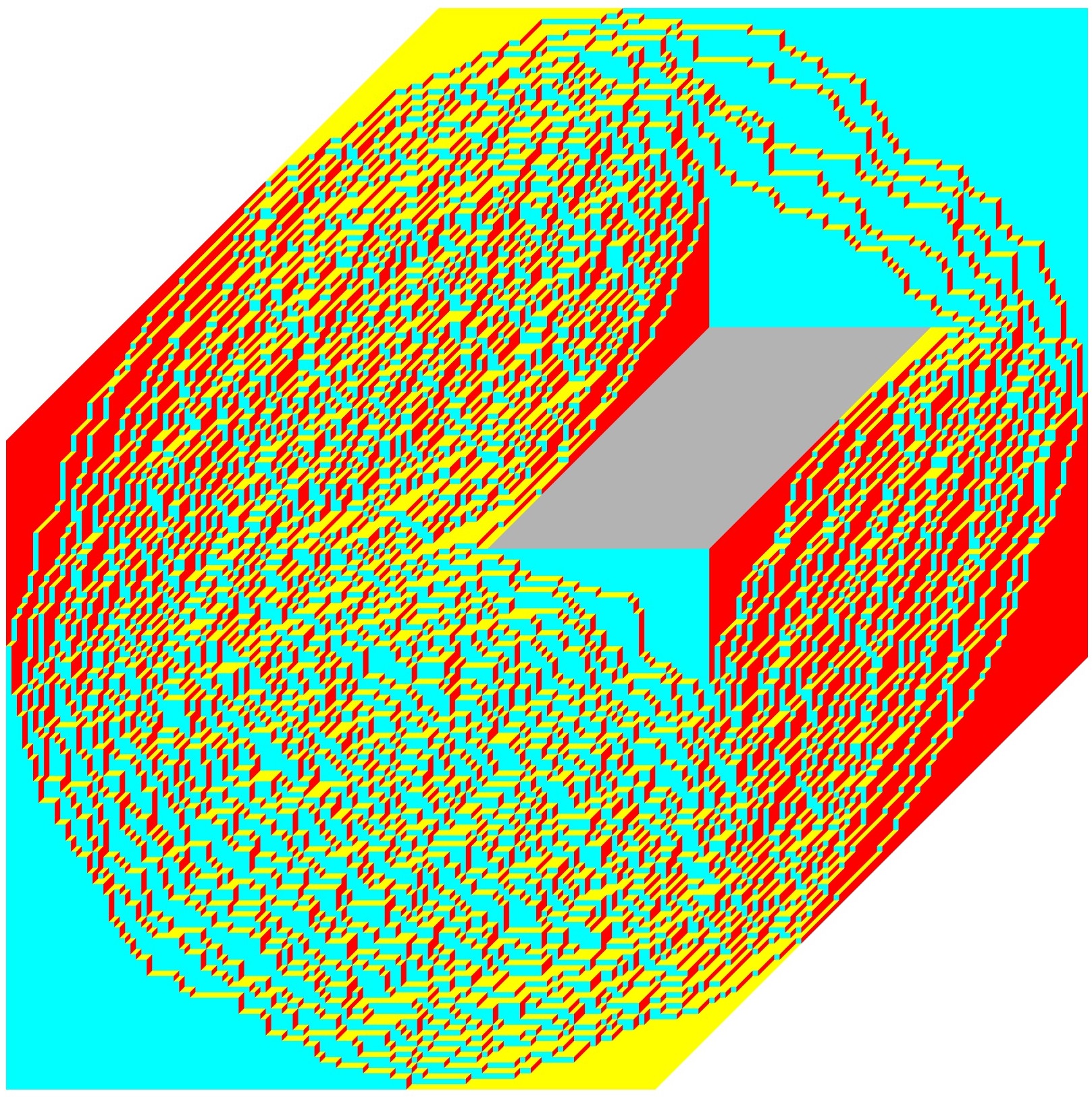}};
\end{tikzpicture} 
\end{center}
\caption{\label{fig:gap hexa 6}Left: a tiling of $H_{L,M,N}^{r,k}$ with $L=200$, $M=60$, $N=100$, $r=140$ and $k=120$, chosen uniformly at random. Right: a tiling of $H_{L,M,N}^{r,[k_{1},k_{0}]\cap \Z}$ with $L=200$, $M=80$, $N=120$, $r=130$, $k_{1}=100$ and $k_{0}=140$, chosen uniformly at random.}
\end{figure}
Define
\begin{equation}\label{lol4}
R(z)=U(z)\tilde Y(z),
\end{equation}
and observe that $\tilde Y$ has the jump relation $\tilde Y_+=\tilde Y_-\tilde J_Y$ on $\gamma$, with \[\tilde J_Y(z)=\begin{pmatrix}1&0&0\\
0&1&\frac{(1+z)^L}{z^{M+N}}\\
0&0&1\end{pmatrix}.\]
It is straightforward to verify that $R$ satisfies the following conditions.
\subsubsection*{RH problem for $R$}
\begin{itemize}
\item[(a)] $R:\mathbb C\setminus\gamma\to\mathbb C^{3\times 3}$ is analytic.
\item[(b)] For $z\in\gamma$, we have
$R_+=R_-J\tilde J_Y$.
\item[(c)] As $z\to\infty$, we have
\[R(z)=\left(\begin{pmatrix}1&0&0\\
0&0&1\\
0&1&0\end{pmatrix}+\bigO(z^{-1})\right)\begin{pmatrix}1&0&0\\
0&z^N&0\\
0&0&z^{-N}\end{pmatrix}.\]
\end{itemize}
The major advantage of the RH problem for $R$, compared to the one for $U$, is that its jump relation depends only on simple explicit functions, not on Jacobi polynomials. We expect that a similar procedure to simplify the jump matrices can be used for the kernels of the form \eqref{double contour integrale duits kuijlaars} appearing in doubly periodic tiling models.
Now we denote $R_1,R_2,R_3$ for the analytic continuations of $R_{11}, R_{12}, R_{13}$ from the region outside $\gamma$ to $\mathbb C\setminus\{0\}$.
Then, the above RH conditions for $R$ imply the following discrete RH conditions for $(R_1,R_2,R_3)$.
\subsubsection*{Discrete RH problem for $(R_1,R_2,R_3)$}
\begin{itemize}\item[(a)] $(R_1,R_2,R_3)$ is analytic in $\mathbb C\setminus \{0\}$.
\item[(b)] As $z\to 0$,
\begin{align}\label{lol5}
(R_1(z),R_2(z),R_3(z))\begin{pmatrix}1&0&\frac{(1+z)^{L-r}}{z^{M+N-k}}\\
-\frac{(1+z)^r}{z^k}&1&0\\0&0&1\end{pmatrix}=\bigO(1).
\end{align}
\item[(c)] As $z\to\infty$, we have
\[(R_1(z),R_2(z),R_3(z))=\left(1+\bigO(z^{-1}),\bigO(z^{N-1}),\bigO(z^{-N-1})\right).\]
\end{itemize}
From these conditions, we deduce first that
\begin{itemize}
\item $R_2$ is a polynomial $\tilde p_{N-1}$ of degree at most $N-1$,
\item $z^kR_1(z)=\tilde P_k(z)$ is a monic polynomial of degree $k$,
\item $z^{M+N}R_3(z)=q_{M-1}$ is a polynomial of degree at most $M-1$.
\end{itemize}
Next, the condition near $0$ for the first entry tells us that
\[z^{-k}\tilde P_k(z)-\frac{(1+z)^r}{z^k}\tilde p_{N-1}(z)=\bigO(1),\qquad z\to 0.\]
Thus,
\begin{equation}\label{eq:systemhex1}
p_{N-1+r-k}(z):=z^{-k}\tilde P_k(z)-\frac{(1+z)^r}{z^k}\tilde p_{N-1}(z)
\end{equation}
is a polynomial of degree at most $N-1+r-k$.
Finally, the condition near $0$ for the third entry tells us that
\[z^{-M-N}q_{M-1}(z)+\frac{(1+z)^{L-r}}{z^{M+N}}\tilde P_k(z)=\bigO(1),\qquad z\to 0.\]
Thus,
\begin{align}\label{def of P in hexa}
P_{L-M-N+k-r}(z):=z^{-M-N}q_{M-1}(z)+\frac{(1+z)^{L-r}}{z^{M+N}}\tilde P_k(z)
\end{align}
is a monic polynomial of degree $L-M-N+k-r$.
Solving \eqref{eq:systemhex1} for $\tilde P_k$ and substituting this in \eqref{def of P in hexa}, we obtain
\[q_{M-1}(z)+(1+z)^{L-r}z^kp_{N-1+r-k}(z)+(1+z)^L\tilde p_{N-1}(z)-z^{M+N}P_{L-M-N+k-r}(z)=0.\]
In particular, this implies that
\[q_{M-1}(z)+(1+z)^{L-r}z^kp_{N-1+r-k}(z)-z^{M+N}P_{L-M-N+k-r}(z)=\bigO((z+1)^L),\qquad z\to -1.\]
On the other hand, using \eqref{lol3}, \eqref{lol4}, \eqref{lol5} and \eqref{eq:systemhex1}, we obtain
\[\frac{G_{L,M,N}^{r,k+1}}{G_{L,M,N}^{r,k}}=U_{11}(0)=\lim_{z\to 0}\left(R_{1}(z)-\frac{(1+z)^r}{z^k}R_2(z)\right)=p_{N-1+r-k}(0).\]
This proves Theorem \ref{theorem:hextiling}.

\subsection{$q+1$ finite intervals}\label{subsec: hexa several finite intervals}
Let $L,M,N\in\mathbb N_{>0}$ with $M<L$, $r\in\{1,\ldots, L-1\}$, $q\in \N$, $k_{0},k_{1},\ldots,k_{2q},k_{2q+1}\in \Z$ be such that \eqref{inequalities between the kj's} holds and such that 
\begin{align}
\begin{cases}
\ds \sum_{j=0}^{q}(k_{2j}+1-k_{2j+1}) \leq \min \{r,M,L-r\}, \\
\ds \max\{0,r-L+M\} \leq k_{2q+1} \leq k_{0} \leq \min \{M+N-1,r+N-1\}.
\end{cases} \label{cond on k0 and k1}
\end{align}
Let us define $\mathcal{I}$ as in \eqref{def:I}. We also define $H_{L,M,N}^{r,\mathcal{I}}$ to be the intersection of the hexagon $H_{L,M,N}$ with the complement of the $q+1$ lozenges delimited by the corners $\{(r,k_{2j+1}-\frac{1}{2}),(r+k_{2j}-k_{2j+1},k_{2j}+\frac{1}{2}),(r,k_{2j}+\frac{1}{2}),(r-k_{2j}+k_{2j+1},k_{2j+1}-\frac{1}{2})\}_{j=0}^{q}$, see Figures \ref{fig:gap hexa 4} (left) and \ref{fig:gap hexa 6} (right) for two examples with $q=0$ and Figure \ref{fig:gap hexa 5} (left) for an example with $q=2$. 

The conditions \eqref{cond on k0 and k1} imply that $H_{L,M,N}^{r,\mathcal{I}}$ is tileable: indeed, the first condition in \eqref{cond on k0 and k1} ensures that there are at least $N$ sites available at level $r$ for the $N$ particles belonging to the non-intersecting paths lying in $H_{L,M,N}^{r,\mathcal{I}}$, and the second condition in \eqref{cond on k0 and k1} implies that $\mathcal{I}$ lies inside the hexagon $H_{L,M,N}$.

\begin{figure}
\begin{center}
\hspace{-1cm}\begin{tikzpicture}[master]
\node at (0,0) {\includegraphics[width=6.5cm]{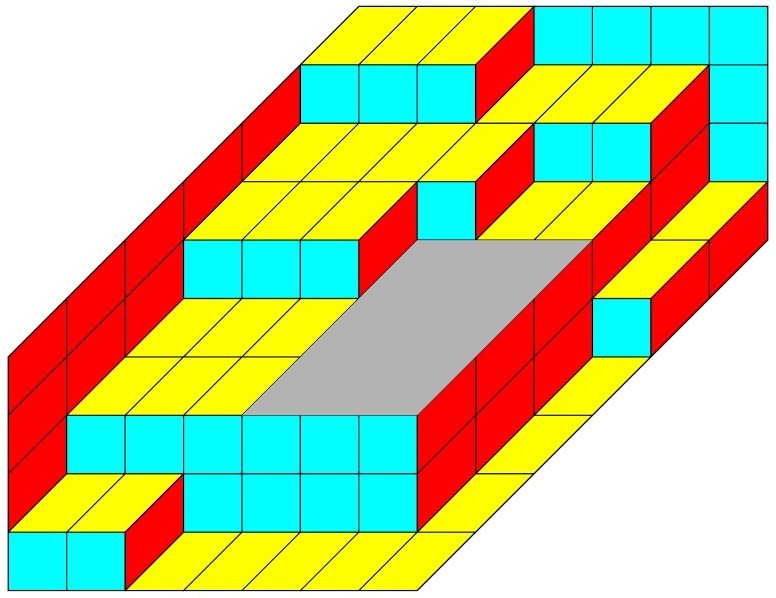}};
\end{tikzpicture} \hspace{0.5cm} \begin{tikzpicture}[slave]
\node at (0,0) {\includegraphics[width=6.5cm]{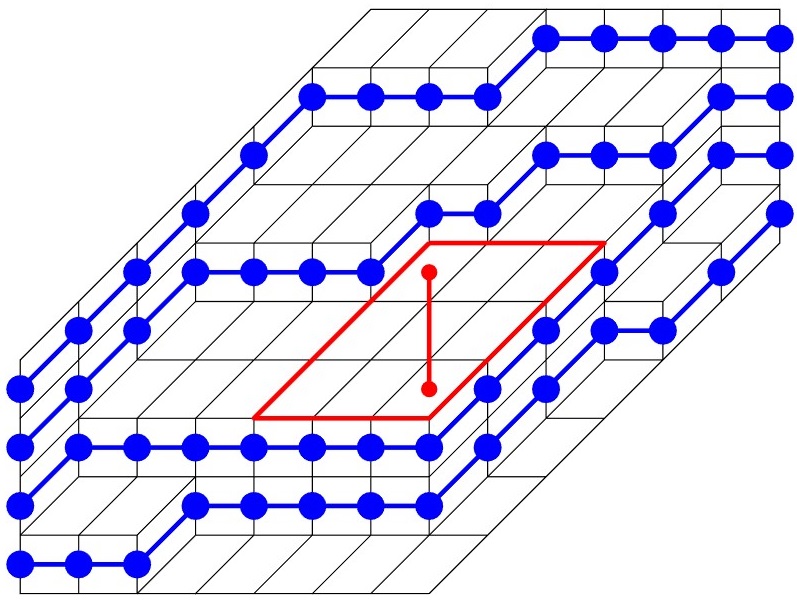}};
\end{tikzpicture} 
\end{center}
\caption{\label{fig:gap hexa 4}{A tiling of $H_{L,M,N}^{r,[k_{1},k_{0}]\cap \Z}$ with $L=13$, $M=6$, $N=4$, $r=7$, $k_{1}=3$ and $k_{0}=5$ (left) corresponds to a system of non-intersecting paths with no points on $\{r\}\times ([k_{1},k_{0}]\cap \Z)$ (right).}}
\end{figure}

Recall that the kernel for the $N$ points at time $r$ is given by \eqref{def:Khexagon}--\eqref{def:hhex}. Since any lozenge tiling of $H_{L,M,N}^{r,\mathcal{I}}$ corresponds to a system of non-intersecting paths avoiding $\{r\}\times \mathcal{I}$ (see Figure \ref{fig:gap hexa 4}), we have
\begin{equation}\label{id:Fredholmhexagon k1k0}
G_{L,M,N}^{r,\mathcal{I}}=G_{L,M,N}\det\left(1-1_{\mathcal{I}}K_{L,M,N,r}\right)_{\ell^2(\mathbb Z)},
\end{equation}
Suppose from now on that $k_{0},k_{1}$ are such that \eqref{cond on k0 and k1} also holds with $k_{0}$ replaced by $k_{0}+1$. Let $\mathcal{I}^*$  be the same cluster as $\mathcal I$, but with $k_{0}$ replaced by $k_{0}+1$, i.e. $\mathcal{I}^{*} = (\sqcup_{j=1}^{q}[k_{2j+1},k_{2j}] \sqcup [k_{1},k_{0}+1])\cap \Z$. Then, by Theorem \ref{thm:multi gap ratio}, we have
\begin{align*}
\frac{G_{L,M,N}^{r,\mathcal{I}^*}}{G_{L,M,N}^{r,\mathcal{I}}} = U_{00}(0),
\end{align*}
where $U$ is the solution to the following RH problem. (In this subsection, since $k_{0}<+\infty$, we adopt the same convention as in Theorem \ref{thm:multi gap ratio}, i.e. the row and column indices of $U$ start at $0$ instead of $1$.)

\begin{figure}
\begin{center}
\hspace{-1cm}\begin{tikzpicture}[master]
\node at (0,0) {\includegraphics[width=6.5cm]{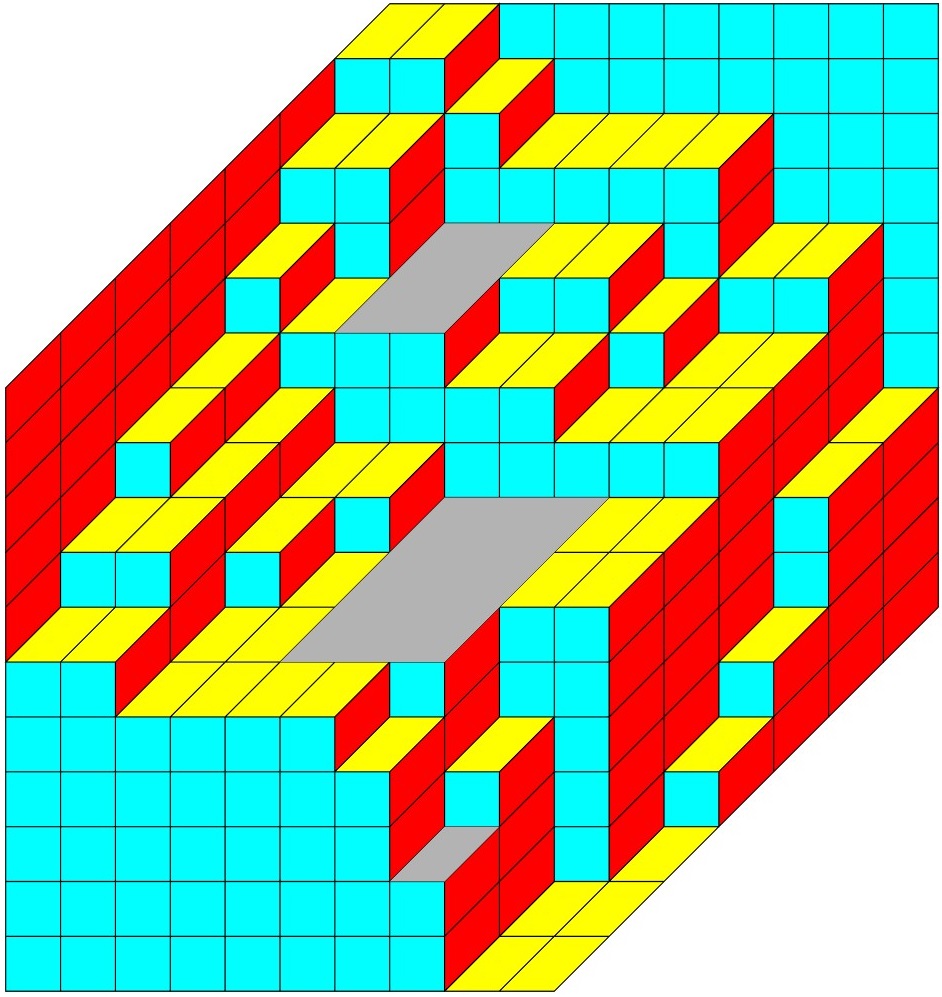}};
\end{tikzpicture} \hspace{0.5cm} \begin{tikzpicture}[slave]
\node at (0,0) {\includegraphics[width=6.5cm]{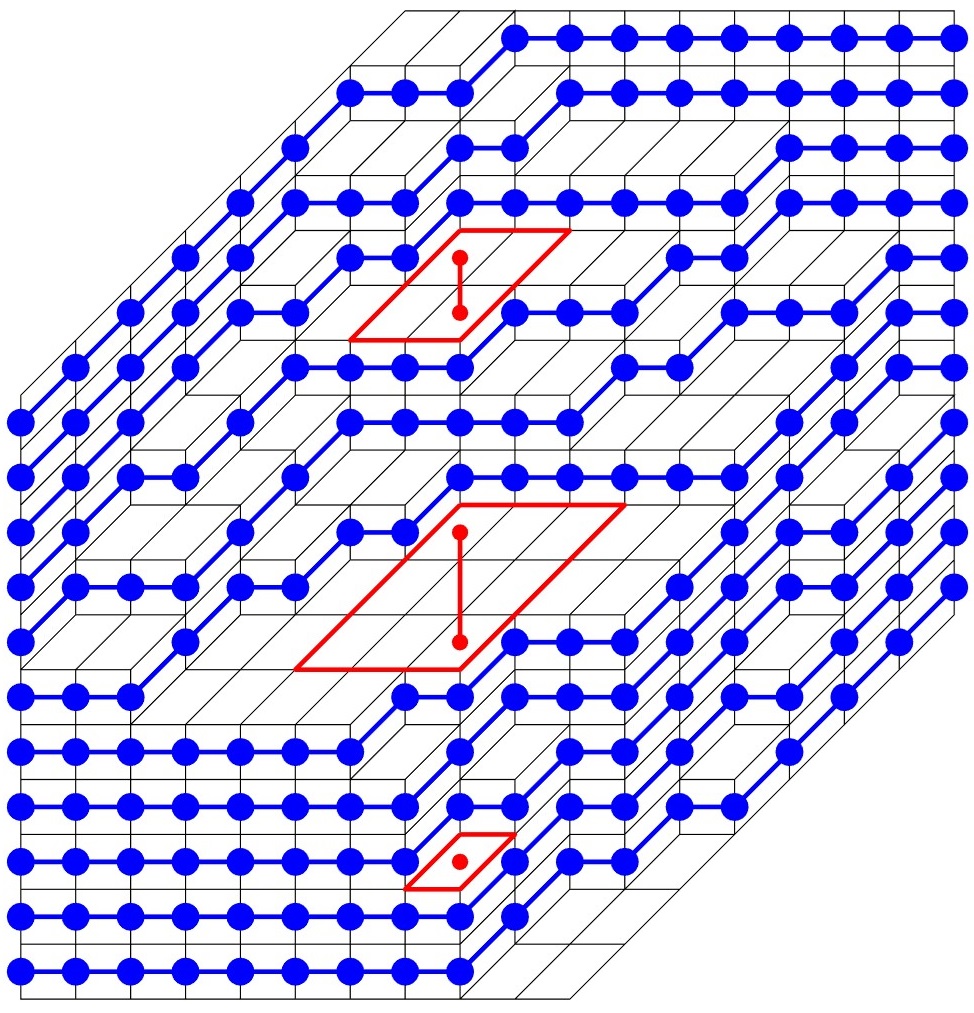}};
\end{tikzpicture} 
\end{center}
\caption{\label{fig:gap hexa 5}{A tiling of $H_{L,M,N}^{r,([k_{5},k_{4}]\cup [k_{3},k_{2}] \cup [k_{1},k_{0}])\cap \Z}$ with $L=17$, $M=7$, $N=11$, $r=8$, $k_{5}=k_{4}=2$, $k_{3}=6$, $k_{2}=8$, $k_{1}=12$ and $k_{0}=13$ (left) corresponds to a system of non-intersecting paths with no points on $\{r\}\times (([k_{5},k_{4}]\cup [k_{3},k_{2}] \cup [k_{1},k_{0}])\cap \Z)$ (right).}}
\end{figure}
 
\subsubsection*{RH problem for $U$}
\begin{itemize}
\item[(a)] $U:\mathbb C\setminus\left(\Sigma_1\cup\Sigma_2\right)\to \mathbb C^{(4+2q)\times (4+2q)}$ is analytic,
where $\Sigma_1$ encircles $0$, and $\Sigma_2$ encircles $\Sigma_1$.
\item[(b)] $U_+(z)=U_-(z)J_U(z)$ for $z\in \Sigma_1\cup\Sigma_2$, with $J_U$ given by \eqref{JU for k0 finite} with $d=2$, $\gamma_{0}=\ldots=\gamma_{q}=1$, and with $g_{1},g_{2},h_{1},h_{2}$ as in \eqref{def:ghex}--\eqref{def:hhex}, i.e.
\begin{align}
& J_U(z) = \begin{pmatrix}
1 & 0 & \cdots & 0 & 0 & -\frac{(1+z)^{L-r}}{z^{M+N-k_{0}-1}}Y_{11}(z) & \frac{(1+z)^{L-r}}{z^{M+N-k_{0}-1}}Y_{21}(z) \\
0 & 1 & \cdots & 0 & 0 & \frac{(1+z)^{L-r}}{z^{M+N-k_{1}}}Y_{11}(z) & - \frac{(1+z)^{L-r}}{z^{M+N-k_{1}}}Y_{21}(z) \\
\cdots & \cdots & \cdots & \cdots & \cdots & \vdots & \vdots \\
0 & 0 & \cdots & 1 & 0 & -\frac{(1+z)^{L-r}}{z^{M+N-k_{2q}-1}}Y_{11}(z) & \frac{(1+z)^{L-r}}{z^{M+N-k_{2q}-1}}Y_{21}(z) \\
0 & 0 & \cdots & 0 & 1 & \frac{(1+z)^{L-r}}{z^{M+N-k_{2q+1}}}Y_{11}(z) & -\frac{(1+z)^{L-r}}{z^{M+N-k_{2q+1}}}Y_{21}(z) \\
0 & 0 & \cdots & 0 & 0 & 1 & 0 \\
0 & 0 & \cdots & 0 & 0 & 0 & 1 \\
\end{pmatrix}, & & z \in \Sigma_{1}, \nonumber \\
& J_U(z) = \begin{pmatrix}
1 & 0 & \cdots & 0 & 0 & 0 & 0 \\
0 & 1 & \cdots & 0 & 0 & 0 & 0 \\
- \frac{(1+z)^{r}}{z^{1+k_{0}}}Y_{21} & - \frac{(1+z)^{r}}{z^{k_{1}}}Y_{21} & \cdots & - \frac{(1+z)^{r}}{z^{1+k_{2q}}}Y_{21} & - \frac{(1+z)^{r}}{z^{k_{2q+1}}}Y_{21} & 1 & 0 \\[0.05cm]
- \frac{(1+z)^{r}}{z^{1+k_{0}}}Y_{11} & - \frac{(1+z)^{r}}{z^{k_{1}}} Y_{11} & \cdots & - \frac{(1+z)^{r}}{z^{1+k_{2q}}}Y_{11} & - \frac{(1+z)^{r}}{z^{k_{2q+1}}} Y_{11} & 0 & 1 \\
\end{pmatrix}, & & z \in \Sigma_{2}. \label{lol8}
\end{align}
In \eqref{lol8}, $Y_{11}$ and $Y_{21}$ should be understood as $Y_{11}(z)$ and $Y_{21}(z)$, respectively.
\item[(c)]{As $z\to \infty$, $U(z) = I +\bigO(z^{-1})$}.
\end{itemize}
By deforming the contours $\Sigma_{1}$ and $\Sigma_{2}$ so that $\Sigma_{1}=\Sigma_{2}=:\gamma$, the jump matrix for $U$ becomes
\begin{align}\label{lol9}
J_{U}(z) = \begin{pmatrix}
I_{2q+2}& \begin{pmatrix}
-z^{1+k_{2\ell}}h_j(z) \\
z^{k_{2\ell+1}}h_j(z)
\end{pmatrix}_{\substack{\ell=0,\ldots,q \\ j=1,\ldots,2}}  \\
-\begin{pmatrix}
z^{-1-k_{2\ell}}g_j(z) \\
z^{-k_{2\ell+1}}g_j(z)
\end{pmatrix}_{\substack{ \ell=0,\ldots,q  \\ j=1,\ldots,2 }}^{T} 
& I_2
\end{pmatrix}.
\end{align}
To be more explicit, if $q=0$, \eqref{lol9} is given by
\begin{align*}
J_U(z) = \begin{pmatrix}
1 & 0 & -\frac{(1+z)^{L-r}}{z^{M+N-k_{0}-1}}Y_{11}(z) & \frac{(1+z)^{L-r}}{z^{M+N-k_{0}-1}}Y_{21}(z) \\
0 & 1 & \frac{(1+z)^{L-r}}{z^{M+N-k_{1}}}Y_{11}(z) & - \frac{(1+z)^{L-r}}{z^{M+N-k_{1}}}Y_{21}(z) \\
- \frac{(1+z)^{r}}{z^{1+k_{0}}}Y_{21}(z) & - \frac{(1+z)^{r}}{z^{k_{1}}}Y_{21}(z) & 1 & 0 \\[0.05cm]
- \frac{(1+z)^{r}}{z^{1+k_{0}}}Y_{11}(z) & - \frac{(1+z)^{r}}{z^{k_{1}}} Y_{11}(z) & 0 & 1 \\
\end{pmatrix}.
\end{align*}
For general $q$, $J_{U}$ can be factorized as
\begin{equation*}
J_U=\tilde Y J\tilde Y^{-1},\qquad \tilde Y=\begin{pmatrix}
I_{2q+2} & 0_{(2q+2)\times 2}\\
0_{2 \times (2q+2)}& \begin{pmatrix}
Y_{21} & Y_{22}\\
Y_{11} & Y_{12}
\end{pmatrix} \end{pmatrix},
\end{equation*}
where
\begin{align*}
J(z)=
\begin{pmatrix}
1 & 0 & \cdots & 0 & 0 & 0 & - \frac{(1+z)^{L-r}}{z^{M+N-k_{0}-1}} \\
0 & 1 & \cdots & 0 & 0 & 0 & \frac{(1+z)^{L-r}}{z^{M+N-k_{1}}} \\
\cdots & \cdots & \cdots & \cdots & \cdots & \vdots & \vdots \\
0 & 0 & \cdots & 1 & 0 & 0 & - \frac{(1+z)^{L-r}}{z^{M+N-k_{2q}-1}} \\
0 & 0 & \cdots & 0 & 1 & 0 & \frac{(1+z)^{L-r}}{z^{M+N-k_{2q+1}}} \\
- \frac{(1+z)^{r}}{z^{1+k_{0}}} & - \frac{(1+z)^{r}}{z^{k_{1}}} & \cdots & - \frac{(1+z)^{r}}{z^{1+k_{2q}}} & - \frac{(1+z)^{r}}{z^{k_{2q+1}}} & 1 & 0 \\
0 & 0 & \cdots & 0 & 0 & 0 & 1 \\
\end{pmatrix}.
\end{align*}
Define
\begin{equation}\label{lol4 bis}
R(z)=U(z)\tilde Y(z),
\end{equation}
and observe that $\tilde Y$ has the jump relation $\tilde Y_+=\tilde Y_-\tilde J_Y$ on $\gamma$, with 
\begin{align*}
\tilde J_Y(z)=\begin{pmatrix}
I_{2q+2} & 0_{(2q+2)\times 2}\\
0_{2 \times (2q+2)}& \begin{pmatrix}
1 & \frac{(1+z)^L}{z^{M+N}}\\
0 & 1
\end{pmatrix} \end{pmatrix}.
\end{align*}
It is straightforward to verify that $R$ satisfies the following conditions.
\subsubsection*{RH problem for $R$}
\begin{itemize}
\item[(a)] $R:\mathbb C\setminus\gamma\to\mathbb C^{(4+2q)\times (4+2q)}$ is analytic.
\item[(b)] For $z\in\gamma$, we have
$R_+=R_-J\tilde J_Y$.
\item[(c)] As $z\to\infty$, we have
\begin{align*}
R(z)=\left( \begin{pmatrix}
I_{2q+2} & 0_{(2q+2)\times 2}\\[0.05cm]
0_{2 \times (2q+2)}& \begin{pmatrix}
0 & 1\\
1 & 0
\end{pmatrix} \end{pmatrix} +\bigO(z^{-1})\right)\begin{pmatrix}
I_{2q+2} & 0_{(2q+2)\times 2}\\[0.05cm]
0_{2 \times (2q+2)}& \begin{pmatrix}
z^{N} & 0\\
0 & z^{-N}
\end{pmatrix} \end{pmatrix}.
\end{align*}
\end{itemize}
Now we denote $R_{0},R_1,\ldots,R_{2q+2},R_{2q+3}$ for the analytic continuations of $R_{0,0}, R_{0,1}, \ldots, R_{0,2q+2}$, $R_{0,2q+3}$ from the region outside $\gamma$ to $\mathbb C\setminus\{0\}$.
Then, the above RH conditions for $R$ imply the following discrete RH conditions for $(R_{0},R_1,\ldots,R_{2q+2},R_{2q+3})$.
\subsubsection*{Discrete RH problem for $(R_{0},R_1,\ldots,R_{2q+2},R_{2q+3})$}
\begin{itemize}\item[(a)] $(R_{0},R_1,\ldots,R_{2q+2},R_{2q+3})$ is analytic in $\mathbb C\setminus \{0\}$.
\item[(b)] As $z\to 0$,
\begin{align}
& (R_{0}(z),R_1(z),\ldots,R_{2q+2}(z),R_{2q+3}(z)) \nonumber \\
& \times \begin{pmatrix}
1 & 0 & \cdots & 0 & 0 & 0 & - \frac{(1+z)^{L-r}}{z^{M+N-k_{0}-1}} \\
0 & 1 & \cdots & 0 & 0 & 0 & \frac{(1+z)^{L-r}}{z^{M+N-k_{1}}} \\
\cdots & \cdots & \cdots & \cdots & \cdots & \vdots & \vdots \\
0 & 0 & \cdots & 1 & 0 & 0 & - \frac{(1+z)^{L-r}}{z^{M+N-k_{2q}-1}} \\
0 & 0 & \cdots & 0 & 1 & 0 & \frac{(1+z)^{L-r}}{z^{M+N-k_{2q+1}}} \\
- \frac{(1+z)^{r}}{z^{1+k_{0}}} & - \frac{(1+z)^{r}}{z^{k_{1}}} & \cdots & - \frac{(1+z)^{r}}{z^{1+k_{2q}}} & - \frac{(1+z)^{r}}{z^{k_{2q+1}}} & 1 & \frac{(1+z)^{L}}{z^{M+N}} \\
0 & 0 & \cdots & 0 & 0 & 0 & 1 \\
\end{pmatrix} =\bigO(1). \label{lol5 bis}
\end{align}
\item[(c)] As $z\to\infty$, we have
\begin{align*}
(R_{0}(z),R_1(z),\ldots,R_{2q+2}(z),R_{2q+3}(z))=\left(1+\bigO(z^{-1}),\bigO(z^{-1}),\ldots,\bigO(z^{-1}),\bigO(z^{N-1}),\bigO(z^{-N-1})\right).
\end{align*}
\end{itemize}
From these conditions, we deduce that
\begin{itemize}
\item $R_{2q+2}$ is a polynomial $p_{N-1}$ of degree at most $N-1$,
\item $z^{k_{0}+1}R_0(z)= P_{k_{0}+1}(z)$ is a monic polynomial of degree $k_{0}+1$, 
\item $z^{k_{1}}R_1(z)=:p_{k_{1}-1}(z), z^{1+k_{2}}R_2(z)=:p_{k_{2}}(z), \ldots, z^{1+k_{2q}}R_{2q}(z)=:p_{k_{2q}}(z)$ and $z^{k_{2q+1}}R_{2q+1}(z)=:p_{k_{2q+1}-1}(z)$ are polynomials of degree at most $k_{1}-1, k_{2}, \ldots, k_{2q},k_{2q+1}-1$, respectively,
\item $z^{M+N}R_{2q+3}(z)=:p_{M-1}$ is a polynomial of degree at most $M-1$.
\end{itemize}
(Above, we slightly abused notation, since the polynomials $p_{N-1}, p_{M-1}, p_{k_{1}-1},$ etc may differ even if e.g. $N=M$ or $N=k_{1}$.) The first $2q+2$ conditions in \eqref{lol5 bis} imply as $z\to 0$ that
\begin{align}
& \frac{P_{k_{0}+1}(z)}{z^{k_{0}+1}}-\frac{(1+z)^r}{z^{k_{0}+1}}p_{N-1}(z)=\bigO(1), & & \frac{p_{k_{1}-1}(z)}{z^{k_{1}}}-\frac{(1+z)^r}{z^{k_{1}}}p_{N-1}(z)=\bigO(1), \nonumber \\
& \frac{p_{k_{2}}(z)}{z^{k_{2}+1}}-\frac{(1+z)^r}{z^{k_{2}+1}}p_{N-1}(z)=\bigO(1), & & \frac{p_{k_{3}-1}(z)}{z^{k_{3}}}-\frac{(1+z)^r}{z^{k_{3}}}p_{N-1}(z)=\bigO(1), \nonumber \\
& \qquad \vdots & & \qquad \vdots \nonumber \\
& \frac{p_{k_{2q}}(z)}{z^{k_{2q}+1}}-\frac{(1+z)^r}{z^{k_{2q}+1}}p_{N-1}(z)=\bigO(1), & & \frac{p_{k_{2q+1}-1}(z)}{z^{k_{2q+1}}}-\frac{(1+z)^r}{z^{k_{2q+1}}}p_{N-1}(z)=\bigO(1). \label{lol6}
\end{align}
Given $n_{1},n_{2},n_{3}\in \N$ satisfying $n_{1}\leq n_{2} \leq n_{3}$ and a polynomial $P(z) = \sum_{j=0}^{n_{3}}\alpha_{j}z^{j}$, we define $[P(z)]_{n_{1}}^{n_{2}}:= \sum_{j=n_{1}}^{n_{2}}\alpha_{j}z^{j}$. The asymptotic formulas \eqref{lol6} imply that
\begin{align}
& P_{k_{0}+1}(z) = z^{k_{0}+1} + \big[ (1+z)^{r}p_{N-1}(z) \big]_{0}^{k_{0}}, \nonumber \\
& p_{k_{2j+1}-1}(z) = \big[ (1+z)^{r}p_{N-1}(z) \big]_{0}^{k_{2j+1}-1}, & & j=0,\ldots,q \nonumber \\
& p_{k_{2j}}(z) = \big[ (1+z)^{r}p_{N-1}(z) \big]_{0}^{k_{2j}}, & & j=1,\ldots,q. \label{lol7}
\end{align}
Finally, the last condition in \eqref{lol5 bis} together with \eqref{lol7} yields
\begin{align}\label{lol10}
\frac{q_{M-1}(z)}{z^{M+N}} + \frac{(1+z)^{L}}{z^{M+N}}p_{N-1}(z) -\frac{(1+z)^{L-r}}{z^{M+N}}\bigg( z^{k_{0}+1} + \sum_{j=0}^{q}\big[ (1+z)^{r}p_{N-1}(z) \big]_{k_{2j+1}}^{k_{2j}} \bigg)  =\bigO(1),\qquad z\to 0.
\end{align}
The above asymptotic formula gives $N+M$ conditions for the $N+M$ unknown coefficients of $q_{M-1}$ and $p_{N-1}$. Since
\begin{multline*}
U_{00}(0) = \lim_{z\to 0}\bigg(R_{0}(z)-\frac{(1+z)^r}{z^{k_{0}+1}}R_{2q+2}(z)\bigg) \\
= \lim_{z\to 0}\bigg(1 + \frac{\big[ (1+z)^{r}p_{N-1}(z) \big]_{0}^{k_{0}}}{z^{k_{0}+1}}-\frac{(1+z)^r}{z^{k_{0}+1}}p_{N-1}(z)\bigg) = 1- \alpha_{k_{0}+1},
\end{multline*}
where $\alpha_{k_{0}+1} := \frac{1}{(k_{0}+1)!} \frac{d^{k_{0}+1}}{dz^{k_{0}+1}}\big( (1+z)^{r}p_{N-1}(z) \big)\big|_{z=0}$ is the coefficient of order $k_{0}+1$ in the Taylor expansion of $(1+z)^{r}p_{N-1}(z)$ around $z=0$, we have proved the following.
\begin{theorem}\label{thm:hexagon multi gap}
Let $L,M,N\in\mathbb N_{>0}$ with $M<L$, $r\in\{1,\ldots, L-1\}$, $q\in \N$, $k_{0},k_{1},\ldots,k_{2q},k_{2q+1}\in \Z$, be such that \eqref{inequalities between the kj's} holds and such that
\begin{align}
\begin{cases}
\ds k_{0}+2-k_{1}+\sum_{j=1}^{q}(k_{2j}+1-k_{2j+1}) \leq \min \{r,M,L-r\}, \\
\ds \max\{0,r-L+M\} \leq k_{2q+1} \leq k_{0}+1 \leq \min \{M+N-1,r+N-1\}.
\end{cases} \label{cond on k0 and k1+1}
\end{align}
Find polynomials $q_{M-1}$ and $p_{N-1}$ of degrees at most $M-1$ and $N-1$, respectively, such that \eqref{lol10} holds. Then
\begin{align*}
\frac{G_{L,M,N}^{r,\mathcal{I}^*}}{G_{L,M,N}^{r,\mathcal{I}}} = 1- \alpha_{k_{0}+1},
\end{align*}
where $\alpha_{k_{0}+1} := \frac{1}{(k_{0}+1)!} \frac{d^{k_{0}+1}}{dz^{k_{0}+1}}\big( (1+z)^{r}p_{N-1}(z) \big)\big|_{z=0}$.
\end{theorem}

\paragraph{Acknowledgements.} The authors are grateful to Philippe Ruelle for interesting discussions and for pointing out the references \cite{CP2013, CP2015}. CC is a Research Associate of the Fonds de la Recherche
Scientifique - FNRS. CC also acknowledges support from the Swedish Research Council, Grant No. 2021-04626, and from the European Research Council (ERC), Grant Agreement No. 101115687.
TC acknowledges support by FNRS Research Project T.0028.23 and by the
Fonds Sp\'ecial de Recherche of UCLouvain.

\end{document}